\documentclass[12pt,a4paper]{amsart}
\usepackage{amsmath}
\usepackage{mathbbol}
\theoremstyle{plain}

\usepackage{extarrows}
\usepackage{enumerate,amssymb,amsthm,amscd}
\usepackage{fullpage}

\usepackage{etex}

\usepackage[arrow, matrix, curve]{xy}
\usepackage{MnSymbol}

\usepackage[colorlinks]{hyperref}

\advance\hoffset-5mm \advance\textwidth40mm
\input{diagrams.tex}

\usepackage{amsfonts}

\def\bdi{\begin{diagram}}
\def\edi{\end{diagram}}

\theoremstyle{plain}

\newtheorem{thm}{Theorem}[section]
\newtheorem{cor}[thm]{Corollary}
\newtheorem{lem}[thm]{Lemma}
\newtheorem{prop}[thm]{Proposition}
\theoremstyle{definition}
\newtheorem{defi}[thm]{Definition}
\newtheorem{defis}[thm]{Definitions}
\newtheorem{conj}[thm]{Conjecture}
\newtheorem{conv}[thm]{Convention}
\newtheorem{nota}[thm]{Notation}
\newtheorem{rem}[thm]{Remark}
\newtheorem{rems}[thm]{Remarks}
\newtheorem{exa}[thm]{Example}
\newtheorem{exas}[thm]{Examples}
\newtheorem{prob}[thm]{Problem}
\newtheorem{probs}[thm]{Problems}
\newtheorem{ques}[thm]{Question}
\newtheorem{sit}[thm]{}

\newcommand{\Spec}{ \operatorname{{\rm Spec}}}
\newcommand{\Frac}{ \operatorname{{\rm Frac}}}

\newcommand{\Sing}{ \operatorname{{\rm Sing}}}

\newcommand{\RF}{ \operatorname{{\rm RF}}}

\newcommand{\Div}{ \operatorname{{\rm Div}}}
\newcommand{\CDiv}{ \operatorname{{\rm CDiv}}}
\newcommand{\Aut}{ \operatorname{{\rm Aut}}}

\newcommand{\LND}{ \operatorname{{\rm LND}}}

\newcommand{\SL}{ \operatorname{{\bf SL}}}
\newcommand{\ML}{ \operatorname{{\rm ML}}}

\newcommand{\Der}{ \operatorname{{\rm Der}}}

\def\deg{\mathop{\rm deg}}

\def\ord{\mathop{\rm ord}}

\def\tp{\mathop{\rm tp}}

\def\ML{\mathop{\rm ML}}
\def\Pic{\mathop{\rm Pic}}

\def\card{\mathop{\rm card}}

\renewcommand{\epsilon}{\varepsilon}

\def\and{\quad\mbox{and}\quad}

\newcommand{\Cl}{ \operatorname{\rm Cl}}

\renewcommand{\div}{ \operatorname{\rm div}}

\newcommand{\C}{\ensuremath{\mathbb{C}}}

\newcommand{\Q}{\ensuremath{\mathbb{Q}}}
\newcommand{\Z}{\ensuremath{\mathbb{Z}}}
\newcommand{\N}{\ensuremath{\mathbb{N}}}
\newcommand{\G}{\ensuremath{\mathbb{G}}}

\def\fF{{\mathfrak F}}
\def\fG{{\mathfrak G}}
\def\fH{{\mathfrak H}}

\def\fR{{\mathfrak R}}

\newcommand{\cB}{{\ensuremath{\mathcal{B}}}}

\newcommand{\cF}{{\ensuremath{\mathcal{F}}}}

\newcommand{\cS}{{\ensuremath{\mathcal{S}}}}
\newcommand{\cE}{{\ensuremath{\mathcal{E}}}}

\newcommand{\cO}{{\ensuremath{\mathcal{O}}}}
\newcommand{\cC}{{\ensuremath{\mathcal{C}}}}
\newcommand{\cD}{{\ensuremath{\mathcal{D}}}}
\newcommand{\cH}{{\ensuremath{\mathcal{H}}}}

\newcommand{\cN}{{\ensuremath{\mathcal{N}}}}

\newcommand{\cX}{{\ensuremath{\mathcal{X}}}}
\newcommand{\cY}{{\ensuremath{\mathcal{Y}}}}

\newcommand{\p}{\partial}

\newcommand{\id}{{\rm id}}
\newcommand{\supp}{{\rm supp\,}}
\newcommand{\h}{{\rm ht}}

\renewcommand{\rho}{\varrho}

\def\bals#1\eals{\begin{align*}#1\end{align*}}
\def\bal#1\eal{\begin{align}#1\end{align}}

\def\SAut{\mathop{\rm SAut}}

\def\A{{\mathbb A}}

\def\NN{{\mathbb N}}
\def\ZZ{{\mathbb Z}}

\def\PP{{\mathbb P}}
\def\V{{\mathbb V}}

\def\fG{{\mathfrak G}}

\renewcommand{\phi}{\varphi}

\newcommand{\bnum}{\begin{enumerate}}
\newcommand{\enum}{\end{enumerate}}
\renewcommand{\emptyset}{\varnothing}

\newcommand{\la}{\label}

\newcommand{\brem}{\begin{rem}}
\newcommand{\brems}{\begin{rems}}
\newcommand{\erem}{\end{rem}}
\newcommand{\erems}{\end{rems}}
\newcommand{\bprob}{\begin{prob}}
\newcommand{\eprob}{\end{prob}}
\newcommand{\bprobs}{\begin{probs}}
\newcommand{\eprobs}{\end{probs}}
\newcommand{\bques}{\begin{ques}}
\newcommand{\eques}{\end{ques}}
\newcommand{\bexa}{\begin{exa}}
\newcommand{\bexas}{\begin{exas}}
\newcommand{\eexa}{\end{exa}}
\newcommand{\eexas}{\end{exas}}
\newcommand{\bdefi}{\begin{defi}}
\newcommand{\edefi}{\end{defi}}
\newcommand{\bdefis}{\begin{defis}}
\newcommand{\edefis}{\end{defis}}
\newcommand{\bcor}{\begin{cor}}
\newcommand{\ecor}{\end{cor}}
\newcommand{\blem}{\begin{lem}}
\newcommand{\elem}{\end{lem}}
\newcommand{\bconv}{\begin{conv}}
\newcommand{\econv}{\end{conv}}
\newcommand{\bconj}{\begin{conj}}
\newcommand{\econj}{\end{conj}}
\newcommand{\bprop}{\begin{prop}}
\newcommand{\eprop}{\end{prop}}
\newcommand{\bthm}{\begin{thm}}
\newcommand{\ethm}{\end{thm}}
\newcommand{\bnota}{\begin{nota}}
\newcommand{\enota}{\end{nota}}
\newcommand{\bsit}{\begin{sit}}
\newcommand{\esit}{\end{sit}}
\newcommand{\be}{\begin{equation}}
\newcommand{\ee}{\end{equation}}
\newcommand{\bproof}{\begin{proof}}
\newcommand{\eproof}{\end{proof}}
\def\ba{\begin{array}}
\def\ea{\end{array}}

\newcommand{\nlin}{\unitlength1mm\begin{picture}(0,9.25)
                      \put(0,0.75){\line(0,1){8.5}}
                     \end{picture}}

\newcommand{\vlin}[1]{\hspace{0.75mm}\unitlength1mm\begin{picture}
(#1,0)
                      \put(0,0){\line(1,0){#1}}
                     \end{picture}\hspace{0.75mm}\rule[-3mm]{0mm}
                     {4mm}}

\def\llin{\vlin{11.5}}
\newcommand{\lin}{\vlin{8.5}}

\newcommand{\co}[1]{\unitlength1mm\begin{picture}(0,8)
   \put(0,0){\circle{1.5}}
   \put(0,3){\makebox(0,5)[b]{$#1$}}
                     \end{picture}}
\newcommand{\mybox}{\unitlength1mm\begin{picture}(0,1.5)
   \put(-0.75,-0.75){\line(0,1){1.5}}
   \put(-0.75,-0.75){\line(1,0){1.5}}
   \put(0.75,0.75){\line(0,-1){1.5}}
   \put(0.75,0.75){\line(-1,0){1.5}}
   \end{picture}}

\newcommand{\xbox}{\unitlength1mm\begin{picture}(0,1.5)
   \put(0,0){$\mybox$}
   \put(-0.75,0){\line(1,0){1.5}}
   \put(0,-0.75){\line(0,1){1.5}}
   \end{picture}}

\newcommand{\cou}[2]{\unitlength1mm\begin{picture}(0,8)
   \put(0,0){\circle{1.5}}
   \put(0,3){\makebox(0,5)[b]{$#1$}}
   \put(0,-7){\makebox(0,4)[t]{$#2$}}
     \end{picture}
     \rule[-7mm]{0mm}{7mm}}

\newcommand{\xbrl}[2]{\unitlength1mm\begin{picture}(0,8)
   \put(0,0){\xbox}
   \put(-5,0){\makebox(0,5)[b]{$#1$}}
  \put(5,0){\makebox(0,5)[b]{$#2$}}
     \end{picture}
     \rule[-7mm]{0mm}{7mm}}

\newcommand{\xbshiftup}[2]{\unitlength1mm\begin{picture}(0,9.25)
                      \put(0,10){\xbrl{#1}{#2}}
                     \end{picture}}

\newcommand{\xbshiftright}[2]{\unitlength1mm\begin{picture}(0,8)
   \put(0,0){\xbrl{#1}{#2}}
                     \end{picture}}

\thanks{
This work
started during a stay of the last two authors at the Max Planck Institut
f\"ur Mathematik (MPIM) at Bonn in October of 2014, continued during the stay of the third author at the MPIM in March-June of 2015 and a short visit of the first author  at the MPIM in March of 2015. The authors thank this
institution for its hospitality, support, and excellent working conditions.}

\begin{document}
\title[Cancellation for surfaces. I]{Cancellation for surfaces revisited. I}

\author{
H.\ Flenner, 
S.\ Kaliman, and
M.\ Zaidenberg}
\address{
Fakult\"at f\"ur
Mathematik, Ruhr Universit\"at Bochum, Geb.\ NA 2/72,
Universit\"ats\-str.\ 150, 44780 Bochum, Germany}
\email{Hubert.Flenner@rub.de}
\address{Department of Mathematics,
University of Miami, Coral Gables, FL 33124, USA}
\email{kaliman@math.miami.edu}
\address{Universit\'e Grenoble Alpes,
Institut Fourier,
CS 40700,
38058 Grenoble cedex 09, France}
\email{Mikhail.Zaidenberg@ujf-grenoble.fr}

\begin{abstract} The celebrated Zariski Cancellation Problem asks as to when the existence of an isomorphism $X\times\A^n\cong X'\times\A^n$ for (affine) algebraic varieties  $X$ and $X'$ implies that $X\cong X'$. In this  paper we provide a criterion for cancellation by the affine line (that is,  $n=1$) in the case where $X$ is a normal affine surface admitting an $\A^1$-fibration $X\to B$ 
over a smooth affine curve $B$. If $X$ does not admit such an $\A^1$-fibration then the cancellation by the affine line is known to hold for $X$ by a result of Bandman and Makar-Limanov (\cite{BML3}). It occurs that for a smooth $\A^1$-fibered affine surface $X$ over $B$ the cancellation by an affine line holds if and only if $X\to B$ is a line bundle, and, for a normal such $X$, if and only if $X\to B$ is a cyclic quotient of a line bundle (an orbifold line bundle). When the cancellation does not hold for $X$ we include $X$ in a non-isotrivial deformation family $X_\lambda\to B$, $\lambda\in\Lambda$, 
of $\A^1$-fibered surfaces with cylinders $X_\lambda\times\A^1$ isomorphic  over $B$. This gives large families of examples of non-cancellation for surfaces which extend the known examples constructed by Danielewski (\cite{Da}), tom Dieck (\cite{tD}), Wilkens (\cite{Wi}), Masuda and   Miyanishi (\cite{MM1}), e.a.  
\end{abstract}
\date{}
\maketitle

\thanks{
{\renewcommand{\thefootnote}{} \footnotetext{ 2010
\textit{Mathematics Subject Classification:}
14R20,\,32M17.\mbox{\hspace{11pt}}\\{\it Key words}: cancellation, affine
surface, group action, one-parameter subgroup, transitivity.}}

{\footnotesize \tableofcontents}

\section*{Introduction}
This is the first in a series of papers addressed to as Part I and Part II. We introduce here into the results of the both parts. 

Let $X$ and $Y$ 
be algebraic varieties  over 
a field $\mathbb{k}$. 
The celebrated Zariski Cancellation Problem, in its most general form, asks
under which circumstances 
the existence of a biregular (resp., birational) isomorphism 
$X\times\A^n\cong Y\times\A^n$ implies that $X\cong Y$ 
where $\A^n$ stands for the affine $n$-space over $\mathbb{k}$. In this and the subsequent papers 
we are interested in the biregular cancellation problem, hence the symbol `$\cong$' 
stands for a biregular isomorphism. We say that $X$ is a {\em Zariski factor} if, whenever $Y$ is an algebraic variety,
$X\times\A^n\cong Y\times\A^n$ implies $X\cong Y$ 
whatever is $n\in\N$. We say that $X$ is a {\em strong Zariski factor} if
any isomorphism $\Phi\colon X\times\A^n\to Y\times\A^n$ 
fits in a commutative diagram 
$$\bdi
X\times\A^n &\rTo<{\Phi} & Y\times\A^n \\
\dTo<{} &&\dTo>{} \\
X &\rTo>{\phi}<{\cong} & Y\, 
\edi$$
where the vertical arrows are the canonical projections. This property is usually called a {\em strong cancellation}. 
We say that $X$ is a {\em Zariski $1$-factor}
if $X\times\A^1\cong Y\times\A^1$ implies that $X\cong Y$, and a {\em strong Zariski $1$-factor} if the strong cancellation holds for $X$ with $n=1$. The latter implies
that the cylinder structure on $X\times\A^1$ is unique, see \cite[Thm.\ 2.18]{LZ}.

By a theorem of Abhyankar, Heinzer and Eakin (\cite[Thm.\ 6.5]{AHE}) any affine curve $C$ is a Zariski factor, and if $C\not\cong\A^1$ then $C$ is a strong Zariski factor. More generally,
by the Iitaka-Fujita Theorem (\cite{IF}) any algebraic variety  of non-negative log-Kodaira dimension is a strong Zariski factor. 
Due to a theorem by Bandman and Makar-Limanov (\cite[Lem.\ 2]{BML3}\footnote{Cf.\ \cite{Dr}; see \cite[Thm.\ 3.1]{CML2} for the positive characteristic case.}) the following holds.

\bthm[Bandman and Makar-Limanov]\label{thm: bml} The affine varieties which do not admit 
any effective $\G_a$-action are strong Zariski $1$-factors. \ethm

There are examples of smooth rational   affine surfaces of
negative log-Kodaira dimension which  are $\A^1$-fibered over $\PP^1$ and do not admit any effective $\G_a$-action, and so, are strong Zariski 1-factors, see \cite[Ex.\ 3]{BML3}, \cite[3.7]{GMMR}. Some of these affine surfaces are not Zariski 2-factors, see \cite{Du5, Du6}.

In this paper we concentrate on the Zariski Cancellation Problem for normal affine surfaces over
an algebraically closed field $\mathbb{k}$ of characteristic zero. From Theorem \ref{thm: bml} one can deduce the following criteria. 

\bcor\label{cor: strong-cancellation}
A normal affine surface $X$ is a strong Zariski $1$-factor if and only if it does not admit any effective $\G_a$-action, if and only if it is not fibered over a smooth affine curve $C$ with general fibers isomorphic to the affine line $\A^1$. 
\ecor

See, e.g., \cite[Thm.\ 2.18]{LZ} for the first part and \cite[Lem.\ 1.6]{FZ} for the second.

Recall (see e.g., \cite{FZ}) that a {\em parabolic $\mathbb{G}_m$-surface} is a normal affine surface $X$ equipped with an $\A^1$-fibration $\pi\colon X\to C$ over a smooth affine curve $C$ and with an effective $\G_m$-action along the fibers of $\pi$. Any fiber of $\pi$ on such a surface $X$ is isomorphic to $\A^1$. There is exactly one singular point of $X$ in each multiple fiber of $\pi$ and no further singularities. Any singular point $x\in X$ is a cyclic quotient singularity. If a parabolic $\mathbb{G}_m$-surface $X\to C$ is smooth then this is a line bundle over $C$. Any parabolic  $\G_m$-surface admits an effective  $\mathbb{G}_a$-action along the fibers of $\pi$ (\cite[Thm.\ 3.12]{FZ-lnd}).

 By the celebrated Miyanishi-Sugie-Fujita Theorem (\cite{MS, Fu}; see also \cite[Ch.\ 3, Thm.\ 2.3.1]{Mi}) the affine plane $\A^2$ is a Zariski factor. An analogous result holds for the parabolic $\G_m$-surfaces. Moreover, the following criterion holds.

\bthm\label{thm: Zariski 1-factor} For a normal affine 
surface $X$ equipped with an $\A^1$-fibration $X\to C$ over a smooth affine curve $C$
the following conditions are equivalent:
\begin{itemize}
\item[(i)] $X$ is a Zariski factor;
\item[(ii)]  $X$ is a Zariski $1$-factor; 
\item[(iii)] $X$ is a parabolic $\G_m$-surface. 
\end{itemize}
\ethm

The implication (i)$\Rightarrow$(ii) is immediate; see Theorem \ref{thm: Zariski-1-factors} for (ii)$\Rightarrow$(iii) and Theorem \ref{thm: Gm-surfaces-are-zar-factors} for (iii)$\Rightarrow$(i).

From Theorems \ref{thm: bml} and \ref{thm: Zariski 1-factor} 
one can deduce the following characterization.

\bcor\label{cor: characterization-1-factors} A normal affine 
surface  $X$ is a Zariski $1$-factor if and only if either $X$ does not admit any  effective $\mathbb{G}_a$-action, or $X$ is a  parabolic $\G_m$-surface. 
\ecor 

The Danielewski surfaces $$X_m=\{z^mt-u^2-1=0\}\subset\A^3,\qquad
m\in\N\,,$$ are examples of non-Zariski 1-factors  (\cite{Da, Fi}). Being pairwise non-homeomorphic (\cite{Fi}) these surfaces have isomorphic cylinders: 
$X_m\times\A^1\cong X_{m'}\times\A^1$ $\forall m,m'\in\N$. 
For non-Zariski 1-factors one can consider the following problems.

\medskip

{\bf Problems.} {\em Given an affine algebraic variety $X$, describe the moduli space $\mathcal{C}_m(X)$ 
of isomorphism classes of the  
affine  algebraic varieties $Y$ such that $X\times\A^m\cong Y\times\A^m$. 
Study the behavior of $\mathcal{C}_m(X)$ upon deformation of $X$.}

\medskip

Note that $X$ is a Zariski 1-factor if and only if $\mathcal{C}_1(X)=\{X\}$. 
There is no example of an affine non-Zariski 1-factor $X$
for which the moduli space $\mathcal{C}_1(X)$ were known. 
For the first Danielewski surface $X_1$ the moduli space
$\mathcal{C}_1(X_1)$ has infinite number of irreducible components. In \cite{Wi}  and \cite[Thm.\ 2.8]{MM1} this sequence is extended to a family 
of surfaces in $\A^3$ with similar properties. These examples show that $\mathcal{C}_1(X_1)$ possesses 
an infinite number of components which are infinite dimensional ind-varieties. 

In both Parts I and II we concentrate on the normal affine surfaces $\A^1$-fibered over affine curves. 
In particular, we show that, unless such a surface $X$ is a parabolic $\mathbb{G}_m$-surface, $X$ deforms in a large family of surfaces with isomorphic cylinders (see Theorems \ref{thm: GDF-cancellation-fixed-graph} and \ref{thm: cancellation-fixed-graph}). 
Moreover, the deformation space contains infinitely many connected components of growing dimensions. 

In Part II we prove the following theorem.
To an $\A^1$-fibered surface $\pi\colon X\to B$ over a smooth affine curve $B$ with reduced fibers one associates a non-separated one-dimensional affine scheme $\breve B$ over $B$ (the \emph{Danielewski-Fieseler quotient}) and an effective \emph{type divisor} $\tp.\div(\pi)$ on $\breve B$.

\bthm\la{thm: main0}
For two $\A^1$-fibered surfaces $\pi\colon X\to B$ and $\pi'\colon X'\to B$  with reduced fibers over the same smooth affine curve $B$, the cylinders $X\times\A^1$ and $X'\times\A^1$ are isomorphic over $B$ if and only if the corresponding Danielewski-Fieseler quotients $\breve B$ and $\breve B'$ are isomorphic over $B$ and the type divisors $\tp.\div(\pi)$ and $\tp.\div(\pi')$ on the common Danielewski-Fieseler quotient $\breve B$ are linearly equivalent.
\ethm

The proofs exploit the affine modifications (\cite{KZ}), in particular, the Asanuma modification (\cite{As})
and 
the flexibility techniques of \cite{AFKKZ}, in particular, the interpolation by automorphisms. 
As an illustration
we analyze from our viewpoint the examples of non-cancellation due to Danielewski (\cite{Da}), Fieseler (\cite{Fi}), Wilkens (\cite{Wi}), tom Dieck (\cite{tD}), 
Miyanishi and Masuda (\cite{MM1}), and the examples of Danielewski-Fieseler surfaces due to Dubouloz and Poloni (\cite{DP0}, \cite{Pol}).

\begin{rem} The results of Parts I and II were reported by the third author on the conference "Complex analyses and dynamical systems - VII" (Nahariya, Israel, May 10--15, 2015), in a seminar at the Bar Ilan University (Ramat Gan, Israel, May 24, 2015), and in the lecture course "Affine algebraic surfaces and the Zariski cancellation problem" at the University of Rome Tor Vergata (September--November, 2015; see the program in \cite{Zai}). When this paper was written the third author assisted at the lecture course  by Adrien Dubouloz on the cancellation problem for affine surfaces in the 39th Autumn School in Algebraic Geometry (Lukecin, Poland, September 19--24, 2016). In this course Adrien Dubouloz advertised a result on non-cancellation for smooth $\A^1$-fibered affine surfaces similar to our result (see, in particular, Theorem \ref{thm: def-0} below and Theorem \ref{thm: Zariski 1-factor} in the case of smooth surfaces), and indicated nice ideas of proofs done by completely different methods. He also posed the question whether the non-degenerate affine toric surfaces are Zariski 1-factors. This had been answered affirmatively by our Theorem \ref{thm: Zariski 1-factor}. \end{rem}

\section{Generalities}
\subsection{Cancellation and the Makar-Limanov invariant} \la{ss:ML}
The special automorphism group $\SAut X$ of an affine variety $X$ is the subgroup of  the group $\Aut X$ 
generated by all its 
$\G_a$-subgroups (\cite{AFKKZ}). The {\em Makar-Limanov invariant} $\ML(X)$ is the  subring of invariants 
of the action of $\SAut X$ on $\cO_X(X)$. The $\SAut X$-orbits are locally closed in $X$ (\cite{AFKKZ}). 
The \emph{complexity} $\kappa$ of the action of $\SAut X$ on $X$ is the codimension of its general orbit, 
or, which is the same, the transcendence degree of the ring $\ML(X)$ (\cite{AFKKZ}). 
We design this integer $\kappa$ as the {\em Makar-Limanov complexity} of $X$, and we say that $X$ 
belongs to the class $(\ML_\kappa)$. 

By the Miyanishi-Sugie Theorem (\cite{MS}, \cite[Ch.\ 2, Thm.\ 2.1.1, Ch.\ 3, Lem.\ 1.3.1 and Thm.\ 1.3.2]{Mi}) 
a normal affine surface $X$ with $\bar{k}(X)=-\infty$ 
contains a cylinder, that is, a principal 
Zariski open subset $U$ of the form $U\cong C\times\A^1$ where $C$ is a smooth affine curve. It possesses as well an 
$\A^1$-fibration $\mu\colon X\to B$ over a smooth curve $B$ which extends the first projection $U\to C$ 
of the cylinder. If $B$ is affine then 
$X$ admits an effective action of the additive group $\G_a=\G_a(\mathbb{k})$ along the rulings of $\mu$. 

Conversely, suppose that there is an effective $\G_a$-action on $X$. Then 
the algebra of invariants $\cO_X(X)^{\G_a}$ is finitely generated and normal (\cite[Lem.\ 1.1]{Fi}). 
Hence 
$B=\Spec \cO_X(X)^{\G_a}$ is a smooth affine curve and the morphism $\mu\colon X\to B$ induced 
by the inclusion $\cO_X(X)^{\G_a}\hookrightarrow  \cO_X(X)$
defines an $\A^1$-fibration 
(an affine ruling) on $X$. Such an $\A^1$-fibration is trivial over a Zariski open subset of $B$. 
It extends the first projection of a principal cylinder on $X$. If an $\A^1$-fibration on a surface $X$ over an affine base is 
unique (non-unique, respectively) then $X$ is of class $(\ML_1)$ (of class $(\ML_0)$, respectively). 
It  is of class $(\ML_2)$ if $X$ does not admit any $\A^1$-fibration over an affine curve. 
In the latter case $X$ still could  admit an $\A^1$-fibration 
over a  projective curve. It does admit such a fibration if and only if $\bar k (X)=-\infty$.

The cancellation problem is closely related to the  problem on stability of the Makar-Limanov 
invariant upon passing to a cylinder. The latter is  discussed, e.g.,  in \cite{BML1}--\cite{BML3} 
and \cite{CML1}-\cite{CM}. Suppose, for instance,  that $\ML(X)=\cO_X(X)$. Then by \cite[Thm.\ 3.1]{CML2} 
(cf.\ also \cite{Dr}), $\ML(X\times\A^1)=\cO_X(X)$. This means 
that the cylinder structure on $X\times\A^1$ is unique. Hence {\em an affine variety $X$ which does not admit 
any effective $\G_a$-action is a Zariski $1$-factor}. 
In particular, any smooth, affine surface of class $(\ML_2)$ is a Zariski 1-factor. Therefore, in the future 
we restrict to surfaces of classes  $(\ML_0)$ and $(\ML_1)$. 

In the Danielewski example, $X_1\in (\ML_0)$ whereas $X_m\in (\ML_1)$ for $m\ge 2$. 
Thus,  the Makar-Limanov complexity is not an invariant of cancellation (see also \cite{Du4} for an example of the Koras-Russell cubic threefold).  By contrast, the Euler characteristic, the Picard number (for a rational variety), the log-plurigenera, 
and the log-irregularity 
are cancellation invariants, see, e.g.,  Iitaka's Lemma in \cite[Ch.\ 2, Lem.\ 1.15.1]{Mi} and \cite[(9.9)]{Fu}. 
                                                                                                                                                                                                                                                                                                                                                                                                                                                        
\subsection{Non-cancellation and Gizatullin surfaces}\label{ss: non-cancellation} Let $X$ be a smooth 
affine surface. Recall (\cite{Gi}) 
that $\SAut X$ acts on $X$ with an open orbit if and only if $X\in\ML_0$. 
In the latter case $X$ is a 
{\em Gizatullin surface}, i.e., a normal affine surface completable by a chain 
of smooth rational curves 
and different from $\A^1\times (\A^1\setminus \{0\})$. Furthermore, 
the group $\SAut (X\times \A^1)$  
also acts with an open orbit on the cylinder  $X\times \A^1$. Thus,
the Makar-Limanov invariant 
$\ML(X\times \A^1)$ is 
trivial: $\ML(X\times \A^1)=\ML(X)=\mathbb{k}$. 

The following conjecture is inspired by \cite[\S 4, Thm. 1]{BML3} and the unpublished notes \cite{BML4} 
kindly offered to one of us by the authors.

\smallskip

\bconj\la{conj: flexible}
 Let $X$ be a 
normal affine surface such that the group $\SAut (X\times \A^1)$ acts with 
an open orbit in $X\times \A^1$. 
Then $\mathcal{C}_1(X)$ contains (the class of) a Gizatullin surface.
\econj

\smallskip

 Due to \cite[Thm.\ 1]{BML3} (see also 
an alternative proof in Part II) this conjecture is true for the Danielewski-Fieseler surfaces, that is, for the $\A^1$-fibered surfaces $\pi\colon X\to \A^1$ with a unique degenerated fiber, provided this fiber is reduced. 

\subsection{The Danielewski--Fieseler construction}\la{ss:DF-construction}
The Danielewski--Fieseler examples of non-cancellation exploit the properties of the
{\em Danielewski--Fieseler quotient}. 
Assume that the $\G_a$-action on $X$ is free. Then the geometric orbit space $X/\G_a$ 
is a non-separated pre-variety (an algebraic space) obtained by gluing together several 
copies of $B:=\Spec\mathcal{O}_X(X)^{\mathbb{G}_a}$ along a common 
Zariski open subset. The morphism $\mu$ can be factorized into $X\to X/\G_a\to B$. 
An ingenious observation by Danielewski is as follows. Consider two non-isomorphic smooth affine  
$\G_a$-surfaces $X$ and $Y$ with free $\G_a$-actions and with the same  
Danielewski--Fieseler quotient $F=X/\G_a=Y/G_a$. 
Then the affine threefold 
$W=X\times_F Y$ carries two induced free $\G_a$-actions. 
Moreover, $W$ carries 
two different structures of principal $\G_a$-bundles (torsors) over 
$X$ and over $Y$, respectively. Since $X$ and $Y$ are affine varieties, 
by Serre's Theorem (\cite{Se}) both these bundles are trivial, and so,
$X\times\A^1\cong W\cong Y\times\A^1$. This is exactly what happens for  two different Danielewski surfaces $X=X_m$ and $Y=X_{m'}$, $m\neq m'$, and in other clasical examples, see Section \ref{sec:examples}. 
The question arises as to how universal is the Danielewski-Fieseler construction. More precisely, 

\smallskip

{\bf Question.} {\em Let $X$ and $Y$ be non-isomorphic smooth affine surfaces 
with isomorphic cylinders $X\times\A^1\cong Y\times\A^1$. Assume that both $X$ and $Y$ 
possess free $\G_a$-actions. Do there  exist 
$\A^1$-fibrations on $X$ and on $Y$ over the same affine base and with the same 
Danielewski--Fieseler quotient?}

\smallskip

Recall (\cite[Def.\ 0.1]{Du0}) that a {\em Danielewski-Fieseler surface} is a smooth
affine surface $X$ equipped with an $\A^1$-fibration $\mu\colon X\to\A^1$ which represents 
a (trivial) line bundle over $\A^1\setminus\{0\}$ and such that the  divisor $\mu^*(0)$ is reduced.
Such a surface admits a free $\G_a$-action along the $\mu$-fibers if and only if it is isomorphic to a surface in $\A^3$ 
with equation $xy-p(z)=0$ where $p\in \mathbb{k}[z]$ has simple roots (\cite[Cor.\ 4.13]{Du0}). 
Theorem \ref{thm: GDF-cancellation-fixed-graph} below deals, more generally, with  
normal affine surfaces $\A^1$-fibered over affine 
curves and such that any fiber of the $\A^1$-fibration is reduced. Abusing the language 
we abbreviate these as {\em GDF-surfaces}, see Definition \ref{def: affine A1}. The Danielewski trick does not work for them, in general, because such a surface does not need to admit a free  $\G_a$-action.
However, we 
show (see Theorems \ref{thm: GDF-cancellation-fixed-graph}  and \ref{thm: GDF 1-factors}) 

\bthm\la{thm: def-0}
A GDF-surface is a Zariski $1$-factor if and only if it is the total space of a line bundle. 
\ethm

 The proof involves affine modifications, in particular, the 
{\em Asanuma modification}. 

\subsection{Affine modifications}
Most of the known examples of non-cancellable 
affine surfaces exploit the Danielewski--Fieseler quotient,
see, e.g., \cite{MM1, Wi}. 
By contrast,  in this paper we use an alternative  construction of non-cancellation due to T.~Asanuma (\cite{As}). 
Recall first the notion of an affine modification  (see \cite{KZ}). 

\bdefi[\emph{Affine modification}]\label{def: aff-modif} Let $X=\Spec \mathfrak{A}$ be a normal affine variety where 
$\mathfrak{A}=\mathcal{O}_X(X)$ is the structure ring of $X$. Let further $I\subset \mathfrak{A}$ be 
an ideal, and let $f\in I$, $f\neq 0$. Consider the Rees algebra $\mathfrak{A}[tI]=
\bigoplus_{n\ge 0} t^nI^n$ with $I^0=\mathfrak{A}$ where  $t$ is an independent variable. 
Consider further the quotient  $\mathfrak{A}'=\mathfrak{A}[tI]/(1-tf)$ by the principal ideal of $\mathfrak{A}[tI]$
generated by $1-tf$.
The affine variety $X'=\Spec \mathfrak{A}'$ is called the {\em affine modification of $X$ along the divisor $D=f^*(0)$ with center $I$}. 
The inclusion $\mathfrak{A}\hookrightarrow \mathfrak{A}'$ induces a birational morphism $\rho\colon X'\to X$ which contracts the exceptional divisor $E=(f\circ\rho)^{-1}(0)$ 
on  $X'$ to the center $\mathbb{V}(I)\subset X$. In fact, 
any birational morphism of affine varieties $X'\to X$ is an affine modification (\cite[Thm.\ 1.1]{KZ}). 
\edefi

\brems\label{rem: aff-modif} 1. If $I=(\mathfrak{a}_1,\ldots,\mathfrak{a}_l)$ where $\mathfrak{a}_i\in \mathfrak{A}$, $i=1,\ldots, l$ then $\mathfrak{A}'=\mathfrak{A}[I/f]=\mathfrak{A}[\mathfrak{a}_1/f,\ldots,\mathfrak{a}_l/f]$.

2. Assume that $f\in I_1\subset I$ where $I_1$ is an ideal of $\mathfrak{A}$. Letting 
$\mathfrak{A}_1=\mathfrak{A}[I_1/f]$ one obtains the equality $\mathfrak{A}'=\mathfrak{A}_1[I_2/f]$ where $I_2$ is the ideal generated by $I$ in $\mathfrak{A}_1$. 
The inclusion $\mathfrak{A}\hookrightarrow \mathfrak{A}_1\hookrightarrow \mathfrak{A}'$ leads to a factorization of the morphism 
$X'\to X$ into a composition of affine modifications, that is, 
birational morphisms of affine varieties $X'\to X_1\to X$ 
where $X_1={\rm spec}\, \mathfrak{A}_1$ (cf.\  also \cite[Prop.\ 1.2]{KZ} for a different kind of factorization).

3. Geometrically speaking, the variety $X'=\Spec \mathfrak{A}'$ is obtained via blowing up  
$X=\Spec \mathfrak{A}$ at the ideal $I\subset \mathfrak{A}$ and deleting 
a certain transform of the divisor $D$ on $X'$, see \cite{KZ} for  details. However, in general $\mathbb{V}(I)$ might have components of codimension 1 which are then also components of the divisor $f^*(0)$. These components survive the modification. Thus, it is worth to distinguish 
between a {\em geometric} affine modification and an {\em algebraic} one. 

Indeed, given a birational morphism of affine varieties $\sigma\colon X'\to X$ with exceptional divisor $E\subset X'$ and center $C=\sigma_*(E)$ of codimension at least 2, the divisor $D$ of the associated modification can be defined as the closure of $X\setminus \sigma(X')$ in $X$. However, this $D$ is not necessarily a principal divisor. So, in order to represent $\sigma\colon X'\to X$ via an affine modification one needs to find a principal divisor on $X$ with support containing $D$. Thus, although
the data $(D,C)$ is uniquely defined for $\sigma$, there are many different affine modifications which induce the same birational morphism $\sigma\colon X'\to X$
(cf.\  \cite{Du-1} and also Remark \ref{rem: non-complete-intersection} for the case of $\A^1$-fibered affine surfaces). 
\erems

The following lemma will be used on several occasions. It generalizes  \cite[Cor.\ 2.2]{KZ} with a similar proof. 

\begin{lem}\label{lem: preserving-isomorphisms} Let $X'\to X$ and $Y'\to Y$ be affine modifications along principal divisors $D_X=\div f_X$ and $D_Y=\div f_Y$ with centers $I_X$ and $I_Y$, respectively,  where $f_X\in I_X\setminus\{0\}$ and $f_Y\in I_Y\setminus\{0\}$. If an isomorphism $\phi\colon X\stackrel{\cong}{\longrightarrow} Y$ sends $f_Y$ to $f_X$ {\rm (}hence, $D_X$ to $D_Y${\rm )} and $I_Y$ onto $I_X$ then $\varphi$ admits a lift to an isomorphism $\phi'\colon X'\stackrel{\cong}{\longrightarrow} Y'$. 
\end{lem}

We need also the following version of this lemma.

\blem\la{lem: lift} 
Let $X$ and $Y$ be affine varieties, and let $\sigma\colon X\to Y$ be an affine modification along a principal divisor $\cD=f^{*}(0)$ in $Y$ with center an ideal $I\subset\cO_{Y}(Y)$ where $f\in I\setminus\{0\}$. Let $\alpha\in\Aut Y$ be such that $\alpha(f)=f$ and both $\alpha$, $\alpha^{-1}$ induce the identity on the $s$th infinitesimal neighborhood of $\cD$ for some  $s\ge 1$, that is, $$\alpha\equiv \id\mod\, f^s\quad\mbox{and}\quad \alpha^{-1}\equiv \id\mod f^s\,.$$ Then $\alpha$ can be lifted to an automorphism $\widetilde\alpha\in\Aut X$ such that \be\la{eq: congru} \widetilde\alpha\equiv \id\mod\, f^{s-1}\quad\mbox{and}\quad \widetilde\alpha^{-1}\equiv \id\mod\, f^{s-1}\,.\ee 
\elem

\bproof Let $\mathfrak{A}=\cO_Y(Y)$ and $\mathfrak{A}'=\cO_X(X)=\mathfrak{A}[\mathfrak{a}_1/f,\ldots,\mathfrak{a}_l/f]$ where $\mathfrak{a}_1,\ldots,\mathfrak{a}_l$ are
generators of $I$. One has $\alpha^*(\mathfrak{a}_i)-\mathfrak{a}_i\in (f^s)$, that is, $\alpha^*(\mathfrak{a}_i)=\mathfrak{a}_i+f^s\mathfrak{b}_i$ for some $\mathfrak{b}_i\in \mathfrak{A}$, $i=1,\ldots,l$. Extending $\alpha^*$ to an automorphism of the fraction field $\Frac \mathfrak{A}$ one has $\alpha^*(\mathfrak{a}_i/f)=\mathfrak{a}_i/f + f^{s-1}\mathfrak{b}_i$, $i=1,\ldots,l$. Thus, $\alpha^*(\mathfrak{A}')\subset \mathfrak{A}'$ and,  similarly, $(\alpha^{-1})^*(\mathfrak{A}')\subset \mathfrak{A}'$. So, $\widetilde\alpha^*:=\alpha^*|_{\mathfrak{A}'}\in\Aut \mathfrak{A}'$ yields  an automorphism $\widetilde\alpha$ of $X$ verifying (\ref{eq: congru}). 
\eproof

It is easily seen that the affine modification of the linear space $\A^n$ with center in 
a linear subspace of codimension $\ge 2$  and with divisor a hyperplane is isomorphic 
to $\A^n$. Similarly, 
certain affine \emph{Asanuma modifications} of a cylinder 
give again a cylinder. This simple and elegant fact is due to Asanuma (\cite{As}); we follow here  \cite[Lem.\ 7.9]{Ka}.

\begin{lem}\label{lem: as-trick0} Let $X$ be an affine variety, $D=\div f$ a principal effective divisor on $X$ where $f\in\cO_X(X)\setminus\{0\}$, and $I\subset\cO_X(X)$ an ideal with support contained in $D$. Let $X'\to X$ be the affine modification of $X$ along $D$ with  center $I$. Consider the cylinder $\mathcal{X}=X\times\A^1=\Spec \mathcal{O}_X(X)[v]$, the divisor $\cD=D\times\A^1$ on $\cX$, the ideal $\tilde I\subset \cO_{\cX}(\cX)$   generated by $I$, and
the ideal $J=(\tilde I,v)\subset A[v]$ supported on $D\times\{0\}\subset \mathcal{D}$. Then the affine  modifications of  $\mathcal{X}$ along $\mathcal{D}$ with center $\tilde I$ and with center $J$ are both isomorphic to the cylinder $\cX'=X'\times\A^1$.
\end{lem}

\begin{proof} The affine modification  of 
$\cX$ along $\cD$ with center $\tilde I$ yields the cylinder  $\cX'$. 
Let  $\mathfrak{a}_1, \ldots, \mathfrak{a}_l\in I$ be generators of $I$, see \ref{def: aff-modif} and \ref{rem: aff-modif}. 
Then $$\mathcal{O_{\mathcal{X}'}}(\mathcal{X}')=\cO_X(X)[\mathfrak{a}_1/f,\ldots, \mathfrak{a}_l/f, v]\cong \cO_X(X)[\mathfrak{a}_1/f,\ldots, \mathfrak{a}_l/f, v'/f]=\mathcal{O}_{\mathcal{X}''}(\mathcal{X}'')\,,$$ 
where $v'=vf$ is a new variable, and $\mathcal{X}''\to \mathcal{X}$ is the affine modification of $\mathcal{X}$ along $\mathcal{D}$ with center  $J$. This
gives the desired isomorphism.  \end{proof}

\section{$\A^1$-fibered surfaces via affine modifications} 
\subsection{Covering trick and GDF surfaces}\label{ss: GDF}
Throughout the paper we deal with the following class of $\A^1$-fibered surfaces.

\begin{defi}[\emph{a GDF surface}]\label{def: affine A1} Let $X$ 
be a normal affine surface over $\mathbb{k}$. A morphism
$\pi\colon  X\to B$  onto a smooth affine curve $B$ is called an {\em  $\A^1$-fibration} 
if the fiber $\pi^{-1}(b)$ over a general point $b\in B$ is isomorphic to the affine line $\A^1$ 
over $\mathbb{k}$. An $\A^1$-fibered surface $\pi\colon X\to B$ is called a 
{\em generalized Danielewski-Fieseler surface}, or a {\em GDF surface} for short, 
if all the fibers $\pi^*(b)$, $b\in B$, 
are reduced. In the case where  $B=\A^1$ and $\pi^{-1}(0)$ is the only reducible fiber of $\pi$ such surfaces were studied in  \cite{Du0} under the name \emph{Danielewski-Fieseler surfaces}.

Any GDF surface is smooth, see, e.g.,  \cite{Du0}  or Lemma \ref{lem: standard GDF-procedure}(b) below. 

We say that a GDF surface $\pi\colon  X\to B$ is {\em marked} if a \emph{marking} $z\in\mathcal{O}_B(B)\setminus\{0\}$ is given such that $z\circ\pi  \in\mathcal{O}_X(X)$ vanishes to order one along any degenerate fiber of $\pi$. Abusing notation we often view $z$ as a function on $X$ identifying $z$ and $z\circ\pi$. The components of the divisor $z^*(0)$ will be called \emph{special fiber components}.

A GDF surface $\pi\colon X\to B$ equipped with actions of a finite group $G$ on $X$ and on $B$ making the morphism $\pi$ $G$-equivariant is called a {\em GDF $G$-surface}. Assume that $G=\mu_d$  is the group of $d$ths roots of unity,  and choose  a $\mu_d$-quasi-invariant marking  $z\in\cO_B(B)$ of weight 1. Then we say that $\pi\colon X\to B$ is a {\em marked GDF $\mu_d$-surface}.
 \end{defi} 

Lemma \ref{lem: br-covering} below  is well known; for the sake of completeness we indicate a proof.  This lemma
says that, starting with a normal affine $\A^1$-fibered surface and  applying a suitable cyclic 
Galois base change, it is possible to obtain a marked GDF $\mu_d$-surface.
The proof uses the following branched covering construction.

\bdefi[\emph{Branched covering construction}]\label{sit: construction} 
Consider a normal affine  $\A^1$-fibered surface $\pi'\colon  Y\to C$  over a smooth affine curve $C$. Fix a finite set of points $p_1,\ldots,p_t\in C$ such that for any $p\in C\setminus\{p_1,\ldots,p_t\}$ the fiber ${\pi'}^*(p)$ is reduced and irreducible. 
Let 
$d$ be the least common 
multiple of the multiplicities of the components of the divisor $\sum_{i=1}^t{\pi'}^*(p_i)$ on $Y$. 
Choose a regular function $h\in\mathcal{O}_C(C)$ 
with only simple zeros which vanishes in the points 
$p_1, \ldots , p_t$ and eventually somewhere else.
Letting $\A^1= {\rm spec}\,\mathbb{k}[z]$
consider the smooth curve $B\subset C \times \A^1$ given by equation $z^d-h(p)=0$ where $(p,z)\in C\times\A^1$ 
along with the morphism ${\rm pr}_1\colon B\to C$.  By abuse of notation we denote the function $z|_{B}\in\cO_B(B)$ still by $z$.  
Let $X$ be the normalization  
of the cross-product $Y\times_C B$, and let $\pi\colon X\to B$ and $\varphi\colon X\to Y$ 
be the induced morphisms.  
\edefi

\begin{lem}\label{lem: br-covering}  In the notation of {\rm \ref{sit: construction}} the following holds.
\begin{itemize}\item
The cyclic group $\mu_d$ of order $d$ acts naturally on $B$ so that
$C=B/\mu_d$; \item 
the morphism ${\rm pr}_1\colon B\to C$ is
ramified to order $d$ over the zeros of $h$. The function $z\in\cO_B(B)$  is a $\mu_d$-quasi-invariant  of weight $1$, and $\div z={\rm pr}_1^*(\div h)$ is a reduced effective $\mu_d$-invariant divisor on $B$; \item the morphism $\varphi\colon X\to Y$ of $\A^1$-fibrations 
is a  cyclic covering with the Galois group $\mu_d$, the reduced branching divisor $h^*(0)$ on $Y$, and the ramification divisor $z^*(0)$ on $X$;
\item
the  $\mu_d$-equivariant morphism $\pi\colon X\to B$ and the marking $z\in\cO_B(B)$ define  a structure of 
a marked GDF $\mu_d$-surface on $X$.
\end{itemize}\end{lem}

\begin{proof} The map $\nu_d\colon\A^1\to\A^1$, $z\mapsto z^d$, is the quotient morphism of 
the natural $\mu_d$-action on $\A^1$. The first three statements follow from the fact that the curve $B$ 
along with the morphism $z\colon B\to \A^1$ is obtained using the morphism  $h\colon C\to\A^1$ 
via the base change $\nu_d\colon\A^1\to\A^1$ that fits in the commutative diagram
\be\la{diagr: 2.3}\bdi
X &\rTo<{/\mu_d} & Y\\
\dTo<{\pi} &&\dTo>{\pi'} \\
B &\rTo<{/\mu_d} & C\\
\dTo<{z} &&\dTo>{h} \\
\A^1 &\rTo>{\nu_d} & \A^1
\edi\ee
The remaining assertions can be reduced to a simple computation in  local charts. Indeed, let $(t,u)$ be coordinates in a local analytic chart $U$ in $Y$ centered at a smooth point $y\in Y$ which is a general point of a fiber component $F$ over $p_i$ of multiplicity $n$ in the divisor $(\pi')^*(p_i)$. We may choose $t$ so that $h\circ\pi'|_U=t^n$ and $F\cap U=t^*(0)$. Then $Y\times_C B$ is given locally in $\A^3$ with coordinates $(z,t,u)$ by equation $z^d-t^n=0$ where $n|d$ by our choice of $d$. This is a union of $n$ smooth surface germs $\{z^{d/n}-\zeta t=0\}$ where $\zeta^n=1$, meeting transversely along the line $z=t=0$ that projects  in $Y$ onto $F\cap U$. Passing to a normalization one gets $n$ smooth disjoint surface germs, say, $V_1,\ldots,V_n$ in $X$ over $U$. The function $z\in\cO_X(X)$ gives in each chart $V_j$ a local coordinate such that $\phi^*(F)=z^*(0)$ has multiplicity one in $V_j$. 
We leave the further details to the reader.  \end{proof}

\bsit[\emph {Cancellation Problem for surfaces: a reduction}] \label{rem: GDF-cyl} 
The following reasoning is borrowed in \cite{MM1, MM2, tD}. It occurs that in order 
to construct (families  of) $\A^1$-fibered surfaces with isomorphic cylinders 
it suffices to construct (families  of) $\A^1$-fibered GDF $G$-surfaces 
with $G$-equivariantly isomorphic cylinders. 

Suppose that a Galois base change $B\to C$ with a Galois group $G$ 
applied to two distinct $\A^1$-fibered surfaces $\pi'_j\colon  Y_j\to C$, $j=0,1$, 
yields two $\A^1$-fibered GDF $G$-surfaces $\pi_j\colon  X_j\to B$, $j=0,1$, with  
$G$-isomorphic over $B$ cylinders $X_0\times\mathbb{A}^1\cong_{G,B}X_1\times\mathbb{A}^1$  where in the both cases $G$ acts identically on the second factor $\mathbb{A}^1$.  
Clearly, one has $(X_j\times\mathbb{A}^1)/G\cong Y_j\times\mathbb{A}^1$, $j=0,1$. 
Passing to the quotients yields an isomorphism over $C$  of cylinders 
$Y_0\times\mathbb{A}^1\cong_C Y_1\times\mathbb{A}^1$ that fits in the commutative diagram
\begin{diagram}[notextflow]
   X_0\times\mathbb{A}^1  &    &\rTo^{\cong_G} &      &   X_1\times\mathbb{A}^1   \\
      & \rdTo_{/G} &      &      & \vLine^{}& \rdTo_{/G}  \\
\dTo^{} &    &   Y_0\times\mathbb{A}^1   & \rTo^{\cong}  & \HonV   &    &  Y_1\times\mathbb{A}^1  \\
      &    & \dTo^{}  &      & \dTo   \\
   B  & \hLine & \VonH   & \rTo^{\rm id} &   B   &    & \dTo_{} \\
      & \rdTo_{/G} &      &      &      & \rdTo_{/G}  \\
      &    &   C   &      & \rTo^{\rm id}  &    &  C  \\
\end{diagram}

\esit

In the sequel we will concentrate on the following problem.
Consider the cylinders $X\times\A^1$ and $X'\times\A^1$ over two
$\A^1$-fibered GDF surfaces $\pi\colon X\to B$ and $\pi'\colon X'\to B$ 
with the same smooth affine base $B$. Suppose that $\pi$ and $\pi'$ 
are equivariant with respect to actions of a finite group $G$ on $X, X'$, 
and $B$. We extend these actions to $G$-actions on the cylinders $X\times\A^1$ 
and $X'\times\A^1$  identically on the second factor.  

\bprob\label{prob: G-isom-cyl} 
 {\em Find a criterion for two GDF $\,\,G$-surfaces $\pi_X\colon X\to B$ and $\pi_{X'}\colon X'\to B$ over the same base $B$ as to when 
the cylinders $X\times\A^1$ and $X'\times\A^1$ 
are $G$-equivariantly isomorphic.} 
\eprob

In Theorems \ref{thm: GDF-cancellation-fixed-graph}, \ref{thm: cancellation-fixed-graph} and Propositions \ref{prop: isomorphism} and \ref{prop:equiv-stretching}  we provide 
some sufficient conditions in the case where $G=\mu_d$. 
Actually, these conditions guarantee the existence of a  $G$-equivariant isomorphism 
$X\times\A^1\stackrel{\cong_{B,G}}{\longrightarrow} X'\times\A^1$ which respects the 
natural projections $X\times\A^1\to B$ and $X'\times\A^1\to B$ and induces the identity on $B$. 

\subsection{Pseudominimal completion and extended divisor}\label{extdiv}
\bdefi[\emph {Pseudominimal  resolved completion}]\label{resi} 
Any $\A^1$-fibration $\pi\colon  X\to B$ on a normal affine surface $X$ over a smooth 
affine curve $B$ extends to  a $\PP^1$-fibration $\tilde\pi:\tilde X\to \bar B$ on 
a complete surface $\tilde X$ over a smooth completion $\bar B$ of $B$ such that 
$D=\tilde X\setminus X$ is a simple normal crossing divisor carrying no singular point of $\tilde X$. 
Let
$\rho:\bar X\to\tilde X$ be the minimal resolution of singularities (all of these singularities are located in $X$). 
Abusing notation, we consider $D$ as a divisor in $\bar X$. 
We call $(\bar X,D)$ a {\em resolved completion} of 
$X$. 

Consider the induced $\PP^1$-fibration 
$\bar\pi:=\tilde\pi\circ\rho:\bar X\to \bar B$. There is a unique ({\em horizontal}) component 
$S$ of $D$ which is a section of $\bar\pi$, while all the other ({\em vertical}) components of 
$D$ are fiber components. Let $\bar B\setminus B=\{c_1,\ldots,c_s\}$. Contracting subsequently 
the $(-1)$-components of $D$ different from $S$ we may assume in addition that  $D$ does not contain any  $(-1)$-component of a fiber.
Such a resolved completion $(\bar X,D)$ is called {\em pseudominimal}. Notice that the trivializing completions used regularly in the sequel (see Definition \ref{def: special-ext}) are not necessarily pseudominimal. 
\edefi

\begin{defi}[\emph {Extended divisor}] \label{def: ext-graph}
Let  $(\bar X,D)$ be a resolved completion 
of $X$ along with the associate $\PP^1$-fibration $\bar\pi\colon\bar X\to\bar B$, and
let $b_1,\ldots,b_n$ be the points of $B$ such that the fibers $\bar\pi^*(b_i)$ over $b_i$ 
in $\bar X$ are degenerate, i.e., are either non-reduced or reducible. 
The reduced divisor 
\be\label{eq: ext-div} D_{\rm ext}=D\cup  \Lambda\quad\mbox{where}
\quad \Lambda=
\bigcup_{j=1}^n\bar \pi^{-1}(b_j)\,\ee 
is called the {\em extended divisor} of $(\bar X,D)$, and the weighted dual graph 
$\Gamma_{\rm ext}$ of $D_{\rm ext}$ the {\em extended graph} of  
$(\bar X,D)$. We say that $\Gamma_{\rm ext}$ is {\em pseudominimal} if the completion $(\bar X,D)$ is. The graph $\Gamma_{\rm ext}$ is a rooted tree with the horizontal section $S\subset D$ as a root. The dual graph $\Gamma(D)$ of the boundary divisor $D$ is a rooted subtree of $\Gamma_{\rm ext}$. 

For a subgraph $\Gamma'$ of a graph $\Gamma$ we let $\Gamma\ominus\Gamma'$ denote the graph obtained from $\Gamma$ by deleting the vertices of $\Gamma'$ along with all their incident edges of $\Gamma$. The connected components of  $\Gamma_{\rm ext}\ominus \Gamma(D)$ are called the
{\em feathers} of $D_{\rm ext}$.  Under the pseudominimality assumption
all the $(-1)$-components of $\Lambda$ are among the feather components. 
\end{defi}

\bdefi[\emph{Standard completion}]\la{rem: unique-ext-div} Consider a pseudominimal resolved completion $\bar\pi\colon\bar X\to\bar B$. The fibers $\bar\pi^{-1}(c_i)$ where $c_i\in\bar B\setminus B$,  $i=1,\ldots,s$  are reduced and irreducible $0$-curves. Performing, if necessary, elementary transformations in one of them we may assume that also the section $S$ is a $0$-curve. Such a completion will be called \emph{standard}, cf., e.g., \cite[5.11]{FKZ-uniqueness}. 
By \cite[Lem.\ 5.12]{FKZ-uniqueness}, 
if two $\A^1$-fibrations $\pi\colon X\to B$ and $\pi'\colon X'\to B$ are isomorphic over $B$ then the corresponding  standard extended divisors $D_{\rm ext}$ and $D'_{\rm ext}$ and the corresponding (unweighted) extended graphs  $\Gamma_{\rm ext} $  and  $\Gamma'_{\rm ext} $ are. 
\edefi

\brem[\emph{Fiber structure}]\la{rem: affine-fibers}
Recall  (see \cite[Ch.\ 3, Lem.\ 1.4.1 and 1.4.4]{Mi})
that any degenerate 
fiber of $\pi\colon X\to B$ is a disjoint union of components 
isomorphic to $\A^1$, any singular point of $X$ is a cyclic quotient singularity, and two such 
singular points cannot belong to the same component. The minimal resolution of a singular 
point has as exceptional divisor  in $\bar X$
a chain of rational curves without $(-1)$-component and with a negative definite intersection form. This chain meets just one other fiber component 
at a terminal component of the chain.
\erem

\bdefi[\emph {Bridges}]\la{def: bridges}  Any feather $\fF$ of $D_{\rm ext}$  (see \ref{def: ext-graph}) is a chain of smooth
rational curves on $\tilde X$ with dual graph
$$
\Gamma(\fF):\qquad \co{F_0}\lin \co{F_{1}}\lin\ldots\lin
\co{F_{ k}}\quad .
$$
The subchain $\fR=\fF\ominus F_0=F_1+\ldots+ F_k$ (if non-empty)
contracts to a cyclic quotient singularity of $X$. The component $F_0$ called the {\em bridge}
of $\fF$ is
attached to a unique component $C$ of $D$. 
The
bridge $F_0$ is the closure in $\bar X$ of a fiber component $F_0\setminus C\cong\A^1$ of $\pi$. 
Vice versa, for each fiber component $F$ of  $\pi$ the closure $\bar F\subset\bar X$ of the proper transform of $F$ is 
a bridge of a unique feather. In the case of a smooth surface $X$ one has
$k=0$, i.e., any feather $\fF$ consists in a bridge: $\fF=F_0$.
\edefi

\subsection{Blowup construction}\label{ss: blowup constructions}

\bdefi[\emph{Blowup construction}]\label{def: simbl} Let   $\pi\colon  X\to B$ be an $\A^1$-fibration  on a normal affine surface $X$ over a smooth affine curve $B$, and let  $(\bar X,D)$ be a resolved completion 
of $X$ along with the associate $\PP^1$-fibration $\bar\pi\colon\bar X\to\bar B$ and with a section `at infinity' $S$. In any degenerate fiber $\bar\pi^*(b_i)$ on $\bar X$, $i=1,\ldots,n$, there is 
a unique component, say, $C_i$ meeting  $S$.  The next fact
 is well known; for the reader's convenience we provide a brief argument.  

\blem\la{lem: contracting-in-fiber} Let $C_0$ be the component of a reducible fiber $\bar\pi^{-1}(b)$, $b\in B$, such that $C_0\cdot S=1$. Then
the rest of
the fiber $\bar\pi^{-1}(b)\ominus C_0$ can be blown down to a
smooth point. \elem

\bproof Since $S\cdot\bar\pi^*(b)=S\cdot C_0=1$, $C_0$ has multiplicity $1$ in the fiber. We proceed by induction on the number $N$ of components in the fiber $\bar\pi^{-1}(b)$.
The statement is clearly true for $N=1$. Suppose now that $N>1$. Then there exists a $(-1)$-component $E$ in the fiber. If $E\neq C_0$ then contracting $E$ one can use the induction hypothesis. Assume now that $C_0$ is the only $(-1)$-component of $\bar\pi^{-1}(b)$. Since $C_0$ has multiplicity 1 it can be contracted to a smooth point of the resulting fiber sitting on a component, say, $C_1$ of multiplicity 1. By the induction hypothesis after blowing down $C_0$ the rest of the resulting fiber but $C_1$ can be blown down. Thus there is a $(-1)$-component of the fiber  $\bar\pi^{-1}(b)$ disjoint from $C_0$. However, the latter  contradicts our assumption that $C_0$ is a unique $(-1)$-component of the fiber $\bar\pi^{-1}(b)$.
\eproof

Performing such a contraction for every $i=1,\ldots,n$
 one arrives at a geometrically minimal ruling (that is, a locally trivial
$\PP^1$-fibration) $\bar\pi_0\colon\bar X_0 \to \bar B$. 
The image $S_0\subset\bar X_0$ of $S$ 
is a section
of $\bar\pi_0$. Thus $\bar X$ can be obtained starting with
a geometrically ruled surface $\bar X_0$ via a sequence of blowups  of points
\be\la{eq: blowupseq} 
\bar X = \bar X_m \stackrel{\bar\rho_m}{\longrightarrow} \bar X_{m-1} 
\longrightarrow  
\ldots \longrightarrow  \bar X_1 \stackrel{\bar\rho_1}{\longrightarrow} 
\bar X_0\,
\ee 
with centers contained in
$\bar\pi_0^{-1}(b_i)\setminus S_0\subset\bar X_0$, $i=1,\ldots,n$, and at infinitely
near points. For $j=0,\ldots,m$ we let $\bar\pi_j\colon \bar X_j\to\bar B$ be the induced $\PP^1$-fibrations.
\edefi

\bdefi[\emph {Well ordered blowup construction}]\label{def: well-ordered} 
In the rooted tree  $\Gamma_{\rm ext}$ with a root $S$, the $(-1)$-vertices on the maximal distance from $S$ are disjoint from $S$ and mutually disjoint due to Lemma \ref{lem: contracting-in-fiber}. 
Hence the corresponding fiber components can be simultaneously contracted. 
Repeating this procedure
one arrives finally at a smooth geometrically ruled surface 
$\bar\pi_0\colon\bar X_0\to \bar B$ along with a specific sequence \eqref{eq: blowupseq} of blowups
where every $\bar\rho_j$,
$j=1,\ldots,n$,  is a blowup with  center in a reduced zero dimensional subscheme of $\bar X_{j-1}\setminus (\bar\pi_{j-1}^{-1}(\bar B\setminus B)\cup S_{j-1})$ where $S_j$ is the proper transform on $\bar X_j$ of $S_0\subset\bar X_0$.  We call such a sequence \eqref{eq: blowupseq} a
{\em well ordered blowup construction}. 
\edefi

The following lemma is a generalization of Theorem 2.1  in \cite{Fi}.

\blem\label{lem: Sumihiro} Let $\pi\colon X\to B$ be an $\A^1$-fibered GDF $G$-surface where $G$ is a finite group. 
Then there is a $G$-equivariant resolved completion $(\bar X, D)$ 
of $X$ obtained via  a  $G$-equivariant well ordered blowup construction
\eqref{eq: blowupseq}.
\elem

\bproof By Sumihiro Theorem (\cite[Thm.\ 3]{Su}) there exists 
a $G$-equivariant projective completion 
$(\tilde X,\tilde  D)$ of $X$. The minimal resolution of singularities of the pair $(\tilde X, \tilde D)$ is  $G$-equivariant (being unique). In this way we arrive at a $G$-equivariant smooth projective completion 
$(\bar X, D)$ of $X$ by a $G$-invariant simple normal crossing  divisor $D$.
The closures in $\bar X$ of the fibers of 
$\pi\colon X\to B$ form a (nonlinear) $G$-invariant
pencil. Its base points also admit a $G$-equivariant resolution. 
Hence we may assume that $\bar X$ comes
equipped with 
a  $G$-equivariant $\PP^1$-fibration $\bar\pi\colon\bar X\to\bar B$ 
along with a $G$-invariant section 
$S$ of $\bar\pi$. 

In particular, the root $S$ of the extended graph $\Gamma_{\rm ext}$ of $(\bar X, D)$ 
is fixed by the induced $G$-action on $\Gamma_{\rm ext}$. This action stabilizes as well the set 
of $(-1)$-vertices on the maximal distance from $S$. Therefore, the simultaneous contraction 
of the corresponding fiber components is $G$-equivariant. By recursion
one arrives at a $G$-equivariant well ordered blowup construction.
\eproof

\brems\label{rem: D} 1. Under a well ordered blowup construction (\ref{eq: blowupseq}) no blowup
is done near the section $S_0$ of $\bar\pi_0$ neither with center over the points $c_i\in \bar B\setminus B$, 
$i=1,\ldots,k$. So, the  fibers  in $\bar X_i$ over these points 
remain reduced and irreducible. 

2. Let  a component $F$ of $D_{\rm ext}$ different from $S$
be created by one of the blowups $\bar\rho_\nu\colon\tilde X_{\nu}\to \tilde X_{\nu-1}$ in \eqref{eq: blowupseq}. We  claim that then
the center $P_\nu$ of the blowup 
$\bar\rho_\nu$ belongs to the image of $D$ in $\bar  X_{\nu-1}$. Indeed, otherwise the last 
$(-1)$-curve, say, $E$ over $P_\nu$ would neither be a bridge of a feather, nor a component of $D$. 
Hence $E$ should be a component of a feather, say, $\fF$, different from the bridge 
component $F_0$. However, the latter contradicts the minimality of $\fF\ominus F_0$, that is, the minimality 
of the resolution of singularities of $X$. 
\erems

Recall the following notions. 

\bsit \label{inner-outer} Let $D$ be a simple normal crossing divisor on a smooth surface $Y$. 
A blowup of $Y$ at  a point $p\in D$ is called {\em outer} if $p$ is a smooth point of $D$ 
and {\em inner} if $p$ is  a node.
\esit 

We use the following notation.

\bnota\label{not: lambda-delta}
Given a blowup construction \eqref{eq: blowupseq} for any $j=0,\ldots,m$ we let $S_j$ be the proper transform on $\bar X_j$ of $S_0\subset\bar X_0$ and
\be\label{eq: lambda-delta} D_{j,\rm ext}=S_j\cup \Delta_{j}
\cup  \Lambda_j\subset \bar X_j\quad\mbox{where}
\quad \Delta_{j}=\bigcup_{i=1}^k\bar \pi_j^{-1}(c_i)\,\,\,
\mbox{and}\,\,\,  \Lambda_j=
\bigcup_{i=1}^n\bar \pi_j^{-1}(b_i)\,.\ee 
\enota

The following lemma should be well known. For (b) see, e.g., \cite[(2.2)]{Du0}  and the proof of Proposition 6.3.23 in \cite{FKZ0}. Recall that a \emph{leaf} of a rooted tree is an extremal vertex different from the root. 

\blem\label{lem: standard GDF-procedure} Let $\pi \colon X\to B$ be a
normal affine $\A^1$-fibered  surface over a smooth affine curve $B$. Consider a resolved completion $(\bar X=\bar X_m,D)$ of $X$ obtained via a well ordered
blowup construction  \eqref{eq: blowupseq} starting with a ruled surface
 $ \bar\pi_0\colon  \bar X_0\to \bar B$. 
Then the following hold. 
\begin{itemize}\item[(a)]
$\pi \colon X\to B$  is a GDF surface if and only if all the blowups $\bar\rho_\nu$  in {\rm (\ref{eq: blowupseq})}, 
$\nu=1,\ldots,m$, 
are outer {\rm (}with respect to
the divisor $D_{0,\rm ext}$ on $ \bar X_0$ and its subsequent total transforms 
$D_{\nu,\rm ext}$ on $ \bar X_\nu${\rm )}. 
\item[(b)] If $\pi \colon X\to B$ is a GDF surface then $X$ is smooth and 
every feather $\fF$ of 
$D_{\rm ext}=D_{m,\rm ext}$ consists in a single $(-1)$-component  $F_0$ which is a bridge. 
\item[(c)] Let $\pi\colon X\to B$ be a GDF surface  with a pseudominimal resolved completion $(\bar X,D)$, see Definition {\rm\ref{resi}}. For a fiber component $F$ of $\pi$ the following are equivalent:
\begin{itemize}\item[$\bullet$] $\bar F$ is a leaf of the rooted tree $\Gamma_{\rm ext}$; \item[$\bullet$]  $\bar F$ is a feather; \item[$\bullet$] $\bar F$ is a $(-1)$-vertex of $\Gamma_{\rm ext}$. \end{itemize}
\end{itemize}\elem

\bproof 
Suppose that for some $\nu\in \{1,\ldots,m\}$ the blowup $\bar\rho_\nu$ is inner. 
Assume also that the center $P_\nu \in\bar X_{\nu-1}$ of 
$\bar\rho_\nu$ lies on the fiber over $b_i\in B$ and on the image $D_{\nu-1,\rm ext}$ of 
$D_{\rm ext}$. Then all the components of the fiber $ \bar\pi^*(b_i)$ which are born over 
$P_\nu$ including the last $(-1)$-component, say, $\bar F$, have multiplicities $>1$. 
Notice that $\bar F=\bar F_0$
 is a bridge component of a feather, say, $\fF$. Hence $\bar F$ is the closure in $\bar X$ of a component $F$ of the  fiber 
$\pi^*(b_i)\subset X$. Thus, the fiber $\pi^*(b_i)$ is not reduced. This contradiction shows that for 
a GDF surface $\pi\colon X\to B$ 
all the $\bar\rho_\nu$, $\nu=1,\ldots,m$, are outer. 

To show the converse suppose that all the   $\bar\rho_\nu$  
in \eqref{eq: blowupseq}, $\nu=1,\ldots,m$
are outer. Then 
all the resulting degenerate fibers are reduced. Hence $\pi\colon X\to B$ 
is a GDF surface. This proves (a).

Assume further that a feather $\fF$ of $D_{\rm ext}$ has more than one component. 
The component of $\fF$ which  
appears the last in the blowup construction \eqref{eq: blowupseq} is
the bridge component $F_0$ of $\fF$. 
Hence $F_0$ results from a blowup $\bar\rho_\nu$ with center $P_\nu$ which 
 lies on the component $\bar F_1$ of $\fF$ and
on the image in 
$\bar X_{\nu-1}$ of a component $C$ of $D$,
see Remark \ref{rem: D}. 
Thus, $P_\nu$ is a 
nodal point of the divisor $D_{\nu-1,\rm ext}$ on $\bar X_{\nu-1}$. 
It follows that $\bar\rho_\nu$ 
is inner. So, the bridge $\bar F_0$  of $\fF$ has multiplicity $>1$ 
in its fiber. 

This proves 
that for a GDF surface $\pi\colon X\to B$ every feather $\fF$ of 
$D_{\rm ext}$ consists in a single bridge
component  $\bar F_0$. Consequently, the surface $X$ is smooth. 
Furthermore, assuming that $\bar F_0^2<-1$ an outer 
blowup was done in 
\eqref{eq: blowupseq} with center on $F_0$ creating a new component, say, $E$ of $D$.
The graph distance dist$(E, S)$ in $\Gamma_{\rm ext}$ is bigger than dist$(\bar F_0, E)$. Hence 
$\bar F_0$ disconnects $S$ and $E$ in $D$. The latter contradicts the facts 
that the affine surface $X$ is connected
at infinity, i.e., its boundary divisor $D$ is connected. Therefore,  $\bar F_0^2=-1$. This shows (b). 

The same argument proves 
that $\bar F_0$ is an extremal vertex (a \emph{tip}) of   $\Gamma_{\rm ext}$  different from $S$. Conversely, if $\bar F$ is a tip  of
$\Gamma_{\rm ext}$ different from $S$ then $\bar F^2=-1$. Indeed, since all the blowups 
in \eqref{eq: blowupseq} are outer then no further blowup was done 
near $\bar F$ after creating $\bar F$. Due to the pseudominimality assumption,  $\bar F$ is a feather of $D_{\rm ext}$. Now (c) follows.  
\eproof

\bdefi[\emph{Fiber trees, levels, and types}]\label{def: fiber tree}  Given an SNC 
completion $\bar\pi\colon\bar X\to\bar B$ of a GDF surface $\pi \colon X\to B$ and a point $b\in B$  the dual graph  $\Gamma_b=\Gamma_b(\bar\pi)$ of the fiber $\bar\pi^{-1}(b)$ will be called a  \emph{fiber tree}.  It depends on the completion chosen. 
This is a rooted tree with a root $v_0\in\Gamma_b$ being the  neighbor of $S$ in $\Gamma_{\rm ext}$. We say that a vertex $v$ of $\Gamma_b$ has \emph{level} $l$ if the tree distance between $v$ and $v_0$ equals $l$. Thus, the root $v_0$ is a unique vertex of $\Gamma_b$ on level $0$.
By a \emph{height} $\h(\Gamma_b)$ we mean the highest level of the vertices in $\Gamma_b$. Remind that the  \emph{leaves} of a rooted tree are its extremal vertices different from the root. By the \emph{type} $\tp(\Gamma_b)$ we mean the sequence of nonnegative integers $(n_1,n_2,\ldots, n_h)$ where $h=\h(\Gamma_b)$ and  $n_i$ is the number of leaves of $\Gamma_b$ on level $i$. \edefi

\brem\label{rem: can-characterization} 
The fiber tree $\Gamma_b$ is an unweighted tree. However, one can easily reconstruct  the weights. Namely, for a vertex $v$ of weight $w(v)$ and of degree $\deg(v)$ in $\Gamma_b$ one has $w(v)=-\deg(v)$. In particular, the $(-1)$-vertices are the tips, and the $(-2)$-vertices are the linear ones.
\erem

\bdefi[\emph{Graph divisor}]\la{def: graph-divisor} Let $\fG$ be the set of all finite weighted rooted trees contractible to the root which acquires then weight zero. By a \emph{graph divisor} on a smooth affine curve $B$ we mean a formal sum
$$ \cD=\sum_{i=1}^n \Gamma_i b_i,\quad \mbox{where}\quad \Gamma_i\in \fG\,.
$$ If all the $\Gamma_i$ are chains then we call $\cD$ a \emph{chain divisor}.
The \emph{height} of a graph divisor $\cD$ is the maximal height of the trees $\Gamma_i$, $i=1,\ldots,n$.

Let $\pi\colon X\to B$ be an $\A^1$-fibered surface with a marking $z\in\cO_B(B)$ where $z^*(0)=b_1+\ldots+b_n$ and with a resolved completion $\bar\pi\colon\bar X\to\bar B$. To 
the corresponding extended graph $\Gamma_{\rm ext}$ we associate its \emph{graph divisor} $\cD(\pi)=\sum_{i=1}^n \Gamma_{b_i} b_i$ where $\Gamma_{b_i}$ is the fiber tree of the fiber $\pi^{-1}(b_i)$. If $\pi\colon X\to B$ is a $\mu_d$-surface and the marking $z$ is $\mu_d$-quasi-invariant then there is an induced $\mu_d$-action on the graph divisor $\cD(\pi)$. \edefi

\subsection{GDF surfaces via affine modifications}\label{ss: aff-modif}

Let $X\to B$ be a GDF surface.  In this subsection we describe
a recursive procedure which allows to recover $X$ starting with the product 
$B\times\A^1$ via a sequence of fibered modifications, see Corollary \ref{cor: final decomposition}.

\bdefi[\emph{Fibered modification}]\footnote{Cf.\ \cite[Def.\ 4.2]{Du0}.}  \label{def: simple-aff-modif} A {\em fibered modification} between two $\A^1$-fibered GDF 
surfaces $\pi\colon X\to B$ and $\pi'\colon X'\to B$
is an affine modification $\varrho\colon X'\to X$ 
which consists in blowing up a reduced 
zero-dimensional subscheme of $X$ and deleting the proper transform of the union of those 
fiber components of $\pi$ which carry centers of blowups.
Thus, $\rho$ is a birational morphism of $B$-schemes.
\edefi

\brem\label{rem: non-complete-intersection} 
Let $F$ be a reduced curve on a smooth affine surface $X$, let $\Sigma\subset F$ be a reduced zero dimensional subscheme, and let $\sigma\colon X'\to X$ be the composition of  blowing up $X$ with the center $\Sigma$ and deleting the proper transform  $F'$ of $F$.  We claim that 
$X'$ is again affine, and so, by \cite[Thm.\ 1.1]{KZ}, the  birational morphism $X'\to X$ is an 
affine modification. 

Indeed, there exists a completion $\bar X$ of $X$ and  
an ample divisor $A$ on $\bar X$ with support ${\rm supp} \,A=\bar X\setminus X$. 
Let $\bar X'$ be the surface obtained from $\bar X$ by blowing up with center $\Sigma$. 
Consider on $\bar X'$ the proper transforms $A'$ and $\bar F'$ 
of $A$ and $\bar F$, respectively, where $\bar F$ is the closure of $F$ in $\bar X$. By Kleiman ampleness criterion the divisor $nA'+\bar F'$ on $\bar X'$ is ample provided 
that $n$ is sufficiently large. Hence the surface $X'=\bar X'\setminus {\rm supp}\,(nA'+\bar F')$ 
is affine, as claimed. 

In general, $F$ is not a principal divisor on 
$X$. To represent $\sigma\colon X'\to X$ via an affine modification  along a principal divisor
let us choose functions $f,g\in\cO_X(X)$  such that $f$ vanishes on $F$ to order 1 and the restriction $g|_F$  vanishes with order 1 on $\Sigma$. Let $I\subset\mathcal{O}_X(X)$  be  
the ideal generated by $f,g$, and by the regular functions on $X$ vanishing on $ \Sigma\cup (\mathbb{V}(f)\setminus F)$. Then $\sigma\colon X'\to X$ is the affine modification along the divisor $f^*(0)$ with the center $I$. 

Let $\pi\colon X\to B$ be a GDF surface, $F$ be a fiber component of $\pi$, and let $f=\pi^*z$ where $z\in\cO_B(B)$ has a simple zero at the point $\pi(F)\in B$. Then $\pi'=\pi\circ\sigma\colon X'\to B$ is again a GDF surface and $\sigma\colon X'\to X$ is a fibered modification. This justifies Definition \ref{def: simple-aff-modif}.
\erem

For a GDF surface
 $\pi\colon X\to B$ one has the following decomposition. 

\bprop\label{prop: decomposition into aff modif}  \begin{itemize}\item[(a)]
Any  GDF surface
 $\pi\colon X\to B$ can be obtained starting with a line bundle $\pi_0\colon X_0\to B$ 
over $B$ via a sequence of  fibered modifications
\be\la{eq: seq-aff-modif} 
X =  X_m \stackrel{\varrho_m}{\longrightarrow} X_{m-1} \stackrel{}{\longrightarrow}  
\ldots \stackrel{}{\longrightarrow}   X_1 \stackrel{\varrho_1}{\longrightarrow} 
X_0\,.
\ee 
This sequence can be extended to the corresponding  completions yielding a well ordered blowup sequence \eqref{eq: blowupseq}. 

\item[(b)] Suppose, furthermore, that $\pi\colon X\to B$ is a  GDF $G$-surface where $G$ is a finite group.
Then 
\eqref{eq: seq-aff-modif} can be chosen so that the intermediate surfaces $X_\nu$ 
come equipped with $G$-actions making the morphisms 
$\varrho_{\nu+1}\colon X_{\nu+1} \to X_{\nu}$ and 
 $\pi_\nu\colon X_\nu\to B$ $G$-equivariant for all  $\nu=0,\ldots,m-1$.
\end{itemize}
\eprop

\bproof (a) To construct 
\eqref{eq: seq-aff-modif} we exploit a well ordered blowup construction  
\eqref{eq: blowupseq} which starts with a $\PP^1$-bundle $\bar\pi_0\colon \bar X_0\to\bar B$
and finishes with a pseudominimal completion $\bar\pi_m\colon \bar X_m\to\bar B$ of $\pi\colon X\to B$.  

For any $\nu=0,\ldots,m$ we let $D_{\nu,\rm ext}$ ($\Delta_\nu$, $S_\mu$, respectively) be the image on  $\bar X_\nu$ of the extended divisor $D_{\rm ext}=D_{m,\rm ext}$ (the divisor $\Delta=\Delta_m$, the section $S=S_m$,  respectively) on $ \bar X_m=\bar X$. Let $\Gamma_{\nu,\rm ext}$ be the weighted dual graph of $D_{\nu,\rm ext}$ and $\Lambda_{\nu,\,\max}$ be the union of the fiber components of $\bar\pi_\nu\colon\bar X_\nu\to\bar B$ which correspond to the extremal vertices of $\Gamma_{\nu,\rm ext}$ 
on maximal distance from $S_\nu$. Let also $D_\nu$ be the union of the remaining components of $D_{\nu,\rm ext}$. Then $\Lambda_{\nu+1,\rm max}$ is the exceptional divisor of the blowup $\bar\rho_\nu\colon\bar X_{\nu+1}\to\bar X_\nu$ with center on 
 $\Lambda_{\nu,\,\max}\setminus D_\nu$. 

Consider the open surface 
$X_\nu=\bar X_\nu\setminus D_{\nu}$. We claim that $X_\nu$ is affine and $\bar\rho_\nu(X_{\nu+1})\subset X_\nu$. Indeed, the latter
 follows since  $\bar\rho_\nu(\Lambda_{\nu+1,\rm max}\setminus D_{\nu+1})\subset \Lambda_{\nu,\,\max}\setminus D_\nu\subset X_\nu$ due to  the  above observation. 
To prove the former we use the Kleiman ampleness criterion (cf. Remark \ref{rem: non-complete-intersection}). Notice that a fiber component $F$ of $D_{\rm ext}$ has level $l$ if dist$(F,S)=l+1$ in $\Gamma_{\rm ext}$. For any $\nu\ge l$ the proper transform of $F$  in $\bar X_\nu$ has the same level $l$.  Choose a sequence of positive integers $$s_0\gg a_0\gg a_1\gg\ldots\gg a_{m-1}\gg 0\,.$$ Let $A_\nu$ be an effective divisor on $\bar X_\nu$ with support $D_\nu$ such that $a_l$ is the multiplicity  in $A_\nu$ of any fiber component $F$ of $D_\nu$ of level $l$, $l=0,\ldots,\nu-1$, and $s_0$ is the multiplicity of $S_\nu$ in $A_\nu$.  Performing elementary transformations in a fiber over a point $c_i\in\bar B\setminus B$ we may assume that $S_\nu^2>0$. Suppose to the contrary that there is an irreducible curve $C$ in $\bar X_\nu$ with $C\cdot A_\nu= 0$. If $C$ is not a component of $D_{\nu,\rm ext}$ then $\bar \pi(C)\subset B$, hence $\bar \pi|_C=$cst which gives a contradiction. If $C=S_\nu$ then clearly $C\cdot A_\nu>0$ which contradicts our choice of $C$ and $A_\nu$. The same contradiction occurs if $C$ is a fiber component of $D_{\nu,\rm ext}$.  Due to the Kleiman ampleness criterion the divisor $A_\nu$ with support $D_\nu$ is ample. Thus the surface $X_\nu=\bar  X_\nu\setminus D_{\nu}$ is affine.

Letting now $\rho_{\nu+1}=\bar\rho_{\nu+1}|_{X_{\nu+1}}\colon X_{\nu+1}\to X_{\nu}$, $\nu=0,\ldots,m-1$ one obtains a desired sequence \eqref{eq: seq-aff-modif} 
of fibered modifications. 
This proves (a).

To show (b) it suffices to start with a $G$-equivariant version of sequence \eqref{eq: blowupseq} constructed in the proof of Lemma \ref{lem: Sumihiro}(b). By our construction, $\bar\rho_{\nu}$ is $G$-equivariant and $D_{\nu,\rm ext}$, $D_{\nu}$, and $X_\nu$ are $G$-invariant. Hence $\rho_{\nu+1}=\bar\rho_{\nu+1}|_{X_{\nu+1}}$ is $G$-equivariant too for any $\nu=0,\ldots,m-1$.
\eproof

The following proposition is an affine analog of the Nagata-Maruyama Theorem on the projective ruled surfaces (\cite{NM}; see also \cite{La}). It allows to replace the line bundle $X_0\to B$ in (\ref{eq: seq-aff-modif})
by the trivial bundle 
$B\times \A^1\to B$. For the  corresponding completions this amounts to a stretching which extends feathers 
by chains of type $[[-1,-2,\ldots,-2]]$ in $D$ near $S$, so loosing the pseudominimality property.

\blem\label{lem: line-bdl-from-direct-prod} Let  $ X$ be 
the total space of a line bundle $\pi\colon X\to B$  
over a smooth affine curve $B$. 
Then the following hold.
 \begin{itemize}\item[(a)] $X$ can be obtained starting with the product 
$B\times\A^1$ over $B$
via a sequence of fibered modifications over $B$,
\be\label{eq: decomposition-2}
X =  Z_M \stackrel{\rho_M}{\longrightarrow} Z_{M-1} \stackrel{}{\longrightarrow}  
\ldots \stackrel{}{\longrightarrow}   Z_1 \stackrel{\rho_1}{\longrightarrow} 
Z_0=B\times\A^1\,
\ee 
where $\pi_i\colon Z_i\to B$ is the projection of a line bundles and the center of $\rho_i$ belongs to the exceptional divisor of $\rho_{i-1}$  for each $i=0,\ldots,M$.
\item[(b)]  If in addition $\pi\colon X\to B$ is a marked GDF $\mu_d$-surface 
then for $i=0,\ldots,M$ the morphisms $\rho_i\colon Z_i\to Z_{i-1}$ in {\rm (\ref{eq: decomposition-2})} and $\pi_i\colon Z_i\to B$ 
can be chosen to be $\mu_d$-equivariant with respect to suitable 
$\mu_d$-actions on the surfaces $Z_i$ and the given $\mu_d$-action on $B$.
\end{itemize}\elem

\bproof (a) Under our assumptions $X$ is affine and admits an effective 
$\mathbb{G}_m$-action along the fibers of $\pi$. This action induces a grading 
$\mathcal{O}_{X}(X)=\bigoplus_{i\ge 0} \mathfrak{A}_i$ where $\mathfrak{A}_0=\mathcal{O}_{B}(B)$ and $\mathfrak{A}_1\neq\{0\}$. 
If $u\in \mathfrak{A}_1$ then the restriction of $u$ to a general fiber of $\pi$ yields a coordinate
on this fiber. It follows that $\psi=({\rm id_B}, u)\colon X\to B\times\A^1$
is a birational 
morphism over $B$, hence an affine modification 
(see \cite[Thm.\ 1.1]{KZ}). 
Since $\psi$ is $\mathbb{G}_m$-equivariant its exceptional divisor $E$, its center $C$, and its divisor $D$ are 
$\mathbb{G}_m$-invariant. Since $u$ is a $\mathbb{G}_m$-quasi-invariant of weight 1 it vanishes along the zero section $Z\subset X$ with order 1. Thus, one has $u^{-1}(0)= Z\cup F_1\cup\ldots\cup F_n$ where $F_i=\pi^{-1}(b_i)$, $b_i\in B$, $i=1,\ldots,n$. Then $$E=F_1\cup\ldots\cup F_n,\quad C=\{b_1,\ldots,b_n\}\times\{0\},\quad \mbox{and}\quad D=\{b_1,\ldots,b_n\}\times\A^1\subset B\times\A^1\,.$$ 
So, $\psi$ consists in blowing up a 
subscheme with support $C$
and deleting the proper transform of $D$. Therefore, $\psi$ factorizes through 
the $\mathbb{G}_m$-equivariant fibered modification 
$\rho_1\colon Z_1\to B\times\A^1$ which consists in  
 blowing up the reduced
subscheme $C$ and deleting the proper transform of $D$. 
 Similarly, the resulting birational morphism of line 
bundles $X\to Z_1$ over $B$ can be factorized over $B$ into $X\to Z_2\to Z_1$ where the center of $Z_2\to Z_1$ is a reduced subscheme of $C$. After a finite number of steps we get a $\mathbb{G}_m$-equivariant resolution of indeterminacies of the inverse birational map $\psi^{-1} \colon B\times\A^1\dashrightarrow X$, hence also
 a desired decomposition of $\psi$ into
a sequence (\ref{eq: decomposition-2}) of fibered modifications.

(b) Under the assumptions of (b) consider the induced
$\mu_d$-action on $Z_0=B\times\A^1$ identical on the second factor. 
In order that $\psi=({\rm id_B}, u)\colon X\to B\times\A^1$
were $\mu_d$-equivariant it suffices to choose $u\in \mathfrak{A}_1^{\mu_d}$ being a $\mu_d$-invariant.
Since $\mu_d$ acts via automorphisms of the line bundle $\pi\colon X\to B$ 
it normalizes the $\mathbb{G}_m$-action on $X$. Hence it induced a representation of 
$\mu_d$
via automorphisms of the graded $k$-algebra $\cO_X(X)=\bigoplus_{i\ge 0} \mathfrak{A}_i$. 
Let $$\mathfrak{A}_1^{(i)}=\{\mathfrak{a}\in \mathfrak{A}_1\,|\,\zeta.\mathfrak{a}=\zeta^i\mathfrak{a}\,\,\forall\zeta\in\mu_d\}\,.$$
Any element $\mathfrak{a}\in \mathfrak{A}_1$ belongs to the $\mu_d$-invariant subspace $E$ 
spanned by the orbit 
$\mu_d(\mathfrak{a})$. The finite dimensional representation of $\mu_d$ 
in $E$ splits into a sum of
one-dimensional representations. Consequently, $\mathfrak{a}$ 
can be written as a sum of $\mu_d$-quasi-invariants.
It follows that $\mathfrak{A}_1=\bigoplus_{i=0}^{d-1} \mathfrak{A}_1^{(i)}$.

We claim that there exists a nonzero invariant $u\in\mathfrak{A} _1^{(0)}=\mathfrak{A}_1^{\mu_d}$.
Indeed,  for some $i\in\{0,\ldots,d-1\}$ there exists a nonzero  $\mu_d$-quasi-invariant
$h\in \mathfrak{A}_1^{(i)}$ of weight $i$.
The marking  $z\in \mathfrak{A}_0=\pi^*(\mathcal{O}_B(B))$ is a 
$\mu_d$-quasi-invariant of weight 1,  see Definition \ref{def: affine A1}. Then 
$u=z^{d-i}h\in \mathfrak{A}_1^{\mu_d}$ as desired.

The resulting birational morphism  $\psi=({\rm id_B}, u)\colon X\to B\times\A^1$ over $B$
is $\mu_d$-equivariant. So, this is an affine modification with 
$\mu_d$-invariant center $C$ and divisor $D$. 

Hence $\psi$ factorizes through 
the $\mu_d$-equivariant fibered  modification 
$\rho_1\colon Z_1\to B\times\A^1$ which consists in  
 blowing up $B\times\A^1$ with center the  reduced zero dimensional subscheme $C\subset B\times\{0\}$ 
and deleting the proper transform of $D$. By recursion one arrives at a 
sequence (\ref{eq: decomposition-2}) of $\mu_d$-equivariant morphisms.
\eproof

\brem\label{rem:moving-support}  In the notation as in the proof of Lemma \ref{lem: line-bdl-from-direct-prod}(a) let $\div u=Z+\sum_{i=1}^n m_iF_i$. Then the effective divisor $\mathfrak{b}:=m_1b_1+\ldots+m_nb_n\in\Div(B)$ in this proof can be replaced by any representative $\mathfrak{b}'\in |\mathfrak{b}|$. Indeed, let $\mathfrak{b}'=\mathfrak{b}+\div f$ for a rational function $f$ on $B$. Then $u'=uf\in \mathfrak{A}_1\subset\cO_X(X)$ and $\div u'=Z+\pi^*(\mathfrak{b}')$. \erem

Letting in \eqref{eq: seq-aff-modif} $G=\mu_d$ and extending this sequence
on the right by those in
\eqref{eq: decomposition-2} with a suitable new enumeration we arrive at  our
final sequence of fibered modifications. 

\bcor\label{cor: final decomposition} \begin{itemize}\item[(a)]
Any 
GDF surface
 $\pi\colon X\to B$ can be obtained 
starting with a product $X_0=B\times\A^1$
via a sequence of  fibered modifications
\be\la{eq: seq-aff-modif-final} 
X =  X_m \stackrel{\varrho_m}{\longrightarrow} X_{m-1} \stackrel{}{\longrightarrow}  
\ldots \stackrel{}{\longrightarrow}   X_1 \stackrel{\varrho_1}{\longrightarrow} 
X_0=B\times\A^1\,
\ee 
such that the center of $\varrho_i$ is contained in the exceptional divisor of $\varrho_{i-1}$.  

\item[(b)] Suppose furthermore that  $\pi\colon X\to B$ is a marked GDF $\mu_d$-surface.
Then any intermediate surface $X_i$,   $i=0,\ldots,m-1$, 
comes equipped with the induced $\mu_d$-action so that the morphisms 
$\varrho_{i+1}\colon X_{i+1} \to X_{i}$ and 
 $\pi_i\colon X_i\to B$ are $\mu_d$-equivariant.
\end{itemize}
\ecor

\bproof
Let us prove (a) leaving the proof of (b) to the reader. Both \eqref{eq: seq-aff-modif} and \eqref{eq: decomposition-2} are chosen well ordered, that is, the center of $\varrho_i$ is contained in the exceptional divisor of $\varrho_{i-1}$. 
Let the centers of blowups in \eqref{eq: seq-aff-modif} and \eqref{eq: decomposition-2} are situated over  the points $b_1,\ldots,b_n\in B$ and $b_1',\ldots,b'_m\in B$, respectively. Let $\mathfrak{b}=b_1+\ldots+b_n\in\Div(B)$ and $\mathfrak{b}'=b_1'+\ldots+b_m'\in\Div(B)$. The linear system $|\mathfrak{b}'|$ is base point free. Hence $\mathfrak{b}'\in |\mathfrak{b}'|$ can be chosen so that $b_i\neq b_j'$ $\forall i,j$, see Remark \ref{rem:moving-support}. Then the blowups in \eqref{eq: seq-aff-modif} and \eqref{eq: decomposition-2} are independent (they commute, in a sense). So, we may perform the blowups in \eqref{eq: seq-aff-modif} and \eqref{eq: decomposition-2}  at each step simultaneously so that the resulting sequence \eqref{eq: seq-aff-modif-final} will be well ordered.
\eproof

\brems\label{rem: extended-to-completions} 1.
The morphisms in (\ref{eq: seq-aff-modif-final}) 
can be extended to suitable completions yielding a sequence of birational morphisms
\be\la{eq: seq-contractions} 
\hat X =  \hat X_m \stackrel{\hat \varrho_m}{\longrightarrow} 
\hat X_{m-1} \stackrel{}{\longrightarrow}  
\ldots \stackrel{}{\longrightarrow}   \hat X_1 \stackrel{\hat \varrho_1}{\longrightarrow} 
 \hat X_0=\bar B\times\PP^1\,,
\ee 
where $\hat \pi_i\colon\hat  X_i\to\bar B$ is a $\mu_d$-equivariant $\PP^1$-fibration extending 
$\pi_i\colon X_i\to B$ and $\hat \varrho_i\colon \hat X_i\to\hat X_{i-1}$ is 
a simultaneous contraction of a $\mu_d$-invariant union of disjoint $(-1)$-components of 
$\hat \pi_i$-fibers, $i=0,\ldots,m$. Inspecting the proof of Lemma 
\ref{lem: line-bdl-from-direct-prod} 
one can see that certain irreducible fibers of ${\rm pr}_1\colon
\bar B\times\PP^1\to\bar B$ are  
replaced by chains of rational curves with sequences of weights of type $[[-1,-2,\ldots,-2,-1]]$. 
This yields a completion $\hat X_m\to \bar B$ of $X_m\to B$ whose boundary 
$\hat X_m\setminus X_m$ is a simple normal crossings divisor. 
The 
section at infinity $\bar B\times\{\infty\}$ of ${\rm pr}_1\colon\bar B\times\PP^1\to\bar B$ 
gives rise to 
a section at infinity  $S$ of $\hat X=\hat X_m\to\bar B$ with $S^2=0$. 
If the line bundle $\bar X_0\to\bar B$ 
in (\ref{eq:  seq-aff-modif}) is nontrivial then the completion $(\hat X, \hat D)$ of $X$
 is not pseudominimal. 

2. It is easily seen that $F$ has level $l$ if and only if it appears for the first time on the surface 
$X_l$ ($l\le m$) in \eqref{eq: seq-aff-modif-final}, see Definition \ref{def: fiber tree}. Thus, any fiber 
component $F'$ of $\pi_l\colon X_l\to B$ has  level  $\le l$. If $\pi\colon X\to B$ is a marked GDF $\mu_d$-surface then the level function 
is $\mu_d$-invariant. Notice that the center of the blowup $\rho_{l+1}\colon X_{l+1}\to X_l$ in  (\ref{eq: seq-aff-modif-final}) is situated on the union of the top level $l$ fiber components in $X_l$.
\erems

\bdefi[\emph{Trivializing completions}]\la{def: special-ext} The completion $(\hat X, \hat D)$ of a GDF surface $X$ fitting in (\ref{eq: seq-contractions}) and the corresponding graph divisor $\cD(\hat\pi)$  
will be called \emph{trivializing}. 
\edefi

\section{Vector fields and natural coordinates}
\subsection{Locally nilpotent vertical  vector fields} 
\label{ss: invariant-vertical-vector-fields}

\blem\label{lem: vert-vect-field} Let $\pi\colon X\to B$ be 
a marked GDF $\mu_d$-surface with a marking $z\in\mathcal{O}_B(B)\setminus\{0\}$.
Then for any $l=0,\ldots,m$ the surface $X_l$ in {\rm (\ref{eq: seq-aff-modif-final})} admits a locally nilpotent  regular
$\mu_d$-quasi-invariant vertical 
vector field $\partial_l$  of weight $l$ 
 non-vanishing on the fiber 
components of the top level $l$ and vanishing on the fiber components of lower levels.\elem

\bproof 
 Consider the locally nilpotent vertical vector field $\partial_0=\partial/\partial u$ on $X_0=B\times\A^1$ where $\A^1={\rm spec}\,\mathbb{k}[u]$. 
Clearly, $\partial_0$ is invariant under the $\mu_d$-action on $B\times\A^1$
 identical on the second factor. 
The $\mu_d$-equivariant fibered modification $\rho_1\colon X_1\to B\times\A^1$ 
over $B$ 
transforms $\partial_0$ into a  $\mu_d$-invariant rational vertical
vector field on $X_1$
with pole of order 1 along the fiber components of level 1. By induction,
$\partial_0$ lifts to a $\mu_d$-invariant rational vertical vector field on $X_l$ 
with pole of order $s$
on any fiber component of level $s$ where $1\le s\le l$ and no other pole. 

Since the  marking $z\in\mathcal{O}_B(B)\setminus \{0\}$ is  $\mu_d$-quasi-invariant of weight 1
then $\partial_l:=z^l\partial/\partial u$ generates a regular locally nilpotent $\mu_d$-quasi-invariant 
vertical vector field on $X_l$ of weight $l$ 
non-vanishing on the fiber components of level $l$  and vanishing on the fiber components of smaller levels.
 \eproof

\subsection{Standard affine charts}
\bnota\la{not: Bi}
Let $\pi\colon X\to B$ be a marked GDF $\mu_d$-surface with a marking $z\in\mathcal{O}_B(B)$ where $z^*(0)=b_1+\ldots+b_n$. For any $i=1,\ldots,n$ consider in $B$ the affine chart $B_i=B\setminus \{b_1,\ldots,b_{i-1},b_{i+1},\ldots,b_n\}$ about the point $b_i$. So, $\div\,(z|_{B_i})=b_i$. 

Given a surface $X_l$ from \eqref{eq: seq-aff-modif-final} we let $F_{i,1},\ldots, F_{i,n_i}$ be the components of the fiber $\pi_l^{-1}(b_i)$. Consider the $\mathbb{G}_a$-action $H_l$ on $X_l$ along the fibers of $\pi_l$ generated by the locally nilpotent vector field $\partial_l$ as in Lemma \ref{lem: vert-vect-field}. 
\enota

\bprop\label{prop: aff-chart} In the notation as above the following hold.
\begin{itemize} \item[$\bullet$] For any $i\in\{1,\ldots,n\}$ and $j\in\{1,\ldots,n_i\}$ there is a unique {\rm standard} affine chart $U_{i,j}\supset F_{i,j}$ in $X_l$ such that $U_{i,j}\cong_{B_i} B_i\times\A^1$. These standard affine charts $(U_{i,j})_{i,j}$ form a covering of $X_l$;
 \item[$\bullet$]  one has $U_{i,j}\cap U_{i,j'}=U_{i,j}\setminus F_{i,j}=U_{i,j'}\setminus F_{i,j'}$ for any $1\le j,j'\le n_i$;  
\item[$\bullet$]  the $\mu_d$-action on $X_l$ induces a $\mu_d$-action by permutations on the collection $(U_{i,j})$;
 \item[$\bullet$]  $U_{i,j}$ is invariant under any action of a connected algebraic group on $X_l\to B$ identical on $B$;
 \item[$\bullet$] for any $t\le l$ and any fiber component $F_{i,j}$ on level $t$ the $H_t$-action is well defined and  free on $U_{i,j}$; 
 \item[$\bullet$] for any $l,t\in\Z$ with $0\le t<l\le m$ the composition $\rho_{l,t}=\rho_{t+1}\circ\ldots\circ\rho_l\colon X_l\to X_t$ sends a standard affine chart $U_{i,j}\subset X_l$ on level $t$ isomorphically over $B_i$ to a standard affine chart in $X_t$.
\end{itemize}
\eprop

\bproof The assertions are evidently true for the product $X_0=B\times\A^1$ in (\ref{eq: seq-aff-modif-final}) with $n_i=1$ $\forall i$ and $U^{(0)}_{i,1}=\pi_0^{-1}(B_i)=B_i\times\A^1$. Suppose by recursion that they hold for a surface $X_{l-1}$ in (\ref{eq: seq-aff-modif-final}) and the collection $(U^{(l-1)}_{i,j})$  of standard affine charts on $X_{l-1}$. The $\mu_d$-equivariant  fibered modification $\rho_{l}\colon X_{l}\to X_{l-1}$ in (\ref{eq: seq-aff-modif-final}) consists in blowing up with center at a union of $\mu_d$-orbits situated on special fiber components of the top level $l-1$ in $X_{l-1}$ and deleting the proper transforms of these fiber components, see Remark \ref{rem: extended-to-completions}.2. Let $F=F_{i,j}$ be one of these components, and let $U_F=U^{(l-1)}_{i,j}$ be the corresponding standard affine chart in $X_{l-1}$. Then the modification $\rho_{l}$ replaces $F$ with new components, say, $F_1,\ldots,F_M$ of level $l$ on $X_l$. The induced fibered modification of $U_F\cong_{B_i} B_i\times\A^1$ results in a GDF surface over $B_i$ with the only degenerate fiber $\pi_l^{-1}(b_i)=\sum_{j=1}^M F_j$. Blowing up just one point, say, $x_j$ one replaces $F$ by $F_j$.  Choosing local coordinates $(z,u)$ in $U_F\cong_{B_i} B_i\times\A^1$ so that $u(x_j)=0$, $z(x_j)=0$ the latter affine modification consists in passing from $\mathcal{O}_{B_i}(B_i)[u]$ to $\mathcal{O}_{B_i}(B_i)[u/z]=\mathcal{O}_{B_i}(B_i)[u']$ where $u'=u/z$.  This results in a standard  affine chart $U^{(l)}_{F,j}\cong_{B_i} B_i\times\A^1$ in $X_l$. In total one obtains $M$ such affine charts on $X_l$ over $U_F$ with the intersections as needed.
 For a fiber component $F$ on $X_{l-1}$ which does not contain any center of the modification $\rho_{l}$ we let $U^{(l)}_{F}=\rho_{l}^{-1}(U^{(l-1)}_{F})$. It is easily seen  that the desired conclusions hold for the resulting collection $(U^{(l)}_{F,j})$ of standard affine charts on $X_l$. We leave the details to the reader.
\eproof

\subsection{Natural coordinates}\label{ss: Natural coordinates} 
Let  $\pi\colon X\to B$ be again a marked GDF $\mu_d$-surface with a marking $z\in\mathcal{O}_B(B)$ where $z^*(0)=b_1+\ldots+b_n$. 

\bdefi[\emph{Natural local coordinates}]\label{def: local coordinates} 
Fix a component $F$ of a fiber $\pi_l^{-1}(b_i)$ on $X_l$, and let $U_F$ be the standard affine chart in $X_l$ about $F$. 
An isomorphism $U_F\cong_{B_i} B_i\times \A^1$ provides sections of $\pi_l|_{U_F}\colon U_F\to B_i$. Fixing  such a section and using the  vertical free $\mathbb{G}_a$-action  on $U_F$ one obtains a $\mathbb{G}_a$-equivariant isomorphism $U_F\cong_{B_i} B_i\times \A^1$ where $\mathbb{G}_a$ acts on the direct product via translations along the second factor. Fixing a coordinate $u$ in $\A^1$ one gets a coordinate, say, $u=u_F$ in $U_F$. 

The restriction $z|_{U_F}$ vanishes to order 1 along $F$ and has no further zero. Hence $(z,u_F)$ yields local coordinates  in $U_F$ near $F$.
We call these {\em natural coordinates}. The local coordinates $(z,u_F,v)$  in the standard affine chart 
$U_{F}\times\A^1$ in the cylinder  $\mathcal{X}_l$  about the affine plane $\cF=F\times\A^1=\Spec\mathbb{k}[u_F,v]$ are also called \emph{natural}. \edefi

\blem\label{lem: nat-coord} Let $\pi\colon X\to B$ be a marked GDF $\mu_d$-surface, and let $X_l$ be one of the surfaces in  {\rm (\ref{eq: seq-aff-modif-final})}. Let $\fF_l$ be the collection of the components of $z^*(0)$ in $X_l$ on the top level $l$. Then there exists a collection $(z,u_F)_{F\in \fF_l}$ of natural local coordinates such that $(u_F)_{F\in \fF_l}$ is a quasi-invariant of $\mu_d$ of weight $-l$. 
\elem

\bproof
For  $F\in\fF_l$ let
$\mu_e\subset\mu_d$ (where $e|d$) be the isotropy subgroup of $F$. The $\mu_e$-action on $U_F$ induces 
a $\mu_e$-action on $B_i\times\A^1\cong_{B_i} U_F$ (see Notation \ref{not: Bi}). The latter isomorphism yields a bijection between the sections of
$\pi_{l}|_{U_F}\colon U_F\to B_i$ and the  functions in $\mathcal{O}_{B_i}(B_i)$. Choosing an arbitrary such section  and averaging over its $\mu_e$-orbit one obtains a $\mu_e$-invariant section, say, $Z_F$. Then $Z_F$ can be taken for the zero  locus of a coordinate 
function $u_{F}$ in $U_F$. Let $\partial_{l}$ be the vertical $\mu_d$-quasi-invariant vector field on $X_{l}$ of weight $l$  constructed in Lemma \ref{lem: vert-vect-field}. The coordinate function $u_F\in\cO_{U_F}(U_F)$ with $u_F|_{Z_F}=0$ and $\partial_{l}(u_F)=1$ is  unique. For $\zeta\in\mu_e$ the ratio $\zeta.u_F/u_F$ does not vanish, hence it is constant along any $\pi_{l}$-fiber.
Thus, one has $\zeta.u_F=(\pi_{l}^*f)\cdot u_F$ for some $f\in\mathcal{O}_{B_i}^\times (B_i)$. From the relations  $\partial_{l}\circ\zeta=\zeta^l\partial_{l}$ and $\partial_{l}(\pi_{l}^*f)=0$ one deduces that $f=\zeta^{-l}$ is a constant, and so, $u_F$ is a $\mu_e$-quasi-invariant of weight $-l$. 

For any component $F'$ in the $\mu_d$-orbit of $F$ define natural coordinates $(z,u_{F'})$ in $U_{F'}$ in such a way that the collection of functions $(u_{F'})$ becomes $\mu_d$-quasi-invariant of weight $-l$. 
Choosing a representative of any $\mu_d$-orbit on $\fF_l$ and repeating the same procedure gives the desired collection of local coordinates. 
\eproof

\brems\label{rem: vvf-in-nat-coord}   1. If $F$ is a component of $z^*(0)$ on $X_l$ on level $l'<l$ then the $\mu_d$-equivariant morphism $\rho_{l}\colon X_{l}\to X_{l-1}$ restricts to an isomorphism on $U_F$. Hence one can define  local coordinates $(z,u_F)$ in $U_F$ where $F$ runs over all the components of $z^*(0)$ in $X_l$ in such a way that for a given $l'\le l$ the collection $(u_F)_{F\in\fF_{l'}}$ is a $\mu_e$-quasi-invariant of weight $-l'$.

In the local coordinates $(z,u_F)$ in $U_F\subset X_l$ of level $l'\le l$ the vertical vector field $\p_l$ on  $X_l$ constructed in Lemma \ref{lem: vert-vect-field} coincides with $z^{l-l'}\p/\p u_F$. In particular, in a top level chart $U_F$ one has $\p_l|_U=\p/\p u_F$.

2. In the case $e=1$ our choice of a $\mu_e$-invariant section is arbitrary, and the coordinate $u_F$  in the standard affine chart $U_{F}$ is defined up to  a factor which is an invertible function lifted from the base and a shift along the $u$-axis. Hence for $F$ of the top level one may consider that $u_F$ does not vanish in any center of the  blowup $\rho_{l+1}\colon X_{l+1}\to X_l$ contained in $F$.
\erems

\subsection{Special $\mu_d$-quasi-invariants} \label{ss: quasi-invariant} In the sequel we need $\mu_d$-invariant locally nilpotent derivations on the cylinders over GDF $\mu_d$-surfaces. To this end we construct  on such surfaces quasi-invariant functions of prescribed weights, see Corollary \ref{cor: q-inv} below. 
Let us start with the following  fact (cf.\
 \cite[Lem.\ 2.12]{KK}).
 
\blem\label{bpp2}
Suppose we are given  a finite group $G$, a character $\lambda \in G^\vee$, an affine $G$-variety $Y$, and a $G$-invariant closed subscheme $Z$ of $Y$ which is not necessarily reduced. Let $f\in\cO_Z(Z)$ belongs to $\lambda$, that is, $f\circ g= \lambda(g)\cdot f$ $\forall g\in G$. Then 
$f$ admits a regular $G$-quasi-invariant extension to $Y$ which belongs to $\lambda$.
\elem

\bproof Letting $A=\cO_Y(Y)$ and $B=\cO_Z(Z)$ the $G$-action yields graded decompositions $A=\bigoplus_{\chi \in G^\vee} A_\chi$
and  $B=\bigoplus_{\chi \in G^\vee} B_\chi$. The piece $A_\chi$ ($B_\chi$, respectively) consists of the $G$-quasi-invariants in $A$ (in $B$, respectively) which belong to the character $\chi$. The closed embedding $Z \hookrightarrow Y$ induces a surjection $\varphi : A \to B$ (\cite[Thm. III.3.7]{Hartshorne}). We claim that $\varphi$ restricts to a surjection $\varphi|_{A_\lambda}\colon A_\lambda\to B_\lambda$ for any $\lambda\in G^\vee$. Indeed, any $f\in  B_{\lambda}$ admits an extension  to a regular function $\tilde f\in A$ such that $\varphi (\tilde f)=f$.
There is a unique decomposition $\tilde f = \sum_{\chi \in G^\vee} \tilde f_\chi$. Hence $f = \sum_{\chi \in G^\vee} \varphi (\tilde f_\chi)$. Since $f \in B_{\lambda}$ the summands $\varphi (\tilde f_\chi)$ with $ \chi\neq\lambda$ vanish, and so, $f =\varphi (\tilde f_\lambda)$. Hence $\tilde f_{\lambda}\in A_\lambda$ is a desired $G$-quasi-invariant
extension of $f$ which belongs to $\lambda$.
\eproof

 \bcor\label{cor: q-inv} Let $\pi\colon X\to B$ be a marked GDF $\mu_d$-surface, let $X_l$ be one of the surfaces in  {\rm (\ref{eq: seq-aff-modif-final})}, and let $\fF_l'\subset\fF_l$ be a $\mu_d$-invariant set of top level fiber components in $X_l$.
Consider a $\mu_d$-quasi-invariant collection of natural local coordinates $(z,u_F)_{F\in\fF_l}$  as in Lemma {\rm \ref{lem: nat-coord}}. 
Then for any $s\gg 1$ 
one can find a $\mu_d$-quasi-invariant function $\tilde u \in \mathcal{O}_{X_l} (X_l)$ of weight $-l$ such that
 \begin{itemize}
\item[(i)]  
$\tilde u\equiv u_{F'} \, {\rm mod} \, z^s$ near $F'$ if $F'\in\fF_l'$ and
\item[(ii)]  $\tilde u\equiv 0 \, {\rm mod} \, z^s$  near $F'$ otherwise. 
\end{itemize}
 \ecor
 
 \bproof It suffices to apply
Lemma \ref{bpp2} with $Y=X_l$, $G=\mu_d$, $Z=({z^s})^*(0)$ being the $s$th 
infinitesimal neighborhood of the union of the special fiber components 
in $X_l$, 
$\lambda (\zeta)=\zeta^{-l}$ for $\zeta\in\mu_d$, and the function $f\in\cO_Z(Z)$ defined in the affine charts $Z\cap U_{F'}$ via $f|_{Z\cap U_{F'}}=u_{F'}|_{Z\cap U_{F'}}$ for $F'\in\fF_l'$ and
$f|_{Z\cap U_{F'}}=0$ otherwise. 
\eproof

\subsection{Examples of GDF surfaces of Danielewski type}\la{sec:examples}
We start with the classical Danielewski example.

\bexa[\emph{Danielewski surfaces}]\label{exa: dani} The Danielewski surface $X_1$ results from the affine modification $\rho_1\colon X_1\to X_0$ of the affine plane $X_0=\A^2=\Spec \mathbb{k}[z,u]$ with the divisor $z=0$ and the center $I=(z,u^2-1)$. This consists in blowing up the points $x_1=(0,1)$ and $x_{-1}=(0,-1)$ in $\A^2$ and deleting the proper transform of the affine line $z=0$. Letting $A_0=\cO_{X_0}(X_0)=\mathbb{k}[z,u]$ and $A_1=\cO_{X_1}(X_1)$ one has $$A_1=A_0[(u^2-1)/z]=\mathbb{k}[z,u,t_1]/(zt_1-u^2+1)\,.$$ The projections $\pi_0\colon X_0\to B=\Spec \mathbb{k}[z]$ and $\pi_1\colon X_1\to B$ are induced by the inclusions $\mathbb{k}[z]\hookrightarrow \mathbb{k}[z,u]\hookrightarrow \mathbb{k}[z,u,(u^2-1)/z]$. Thus, $X_1$ is given in $\A^3$ with coordinates $(z,u,t_1)$ by equation  
$$zt_1-u^2+1=0\,.$$ The unique reducible fiber $\pi_1^*(0)$ of the GDF surface $\pi_1=z|_{X_1}\colon X_1\to B=\A^1$ consists of two disjoint affine lines (components of level one) $$F_1=\{z=0,u=1\}\quad\mbox{and}\quad F_{-1}=\{z=0,u=-1\}\,.$$  The complement $X_1\setminus F_j$ for $j\neq i$ gives a standard affine chart $U_i\cong\A^2$ about $F_i$. The chart $U_1$ can be obtained via the affine modification of $X_1$ along the divisor $z^*(0)=F_1+F_{-1}$ with center the ideal $\mathbb{V}(F_1)=(z,u-1)$. Thus, $$\mathcal{O}_{U_1}(U_1)=A_1\left[(u-1)/z\right] 
=\mathbb{k}[z,u_1]\quad\mbox{where}\quad u_1=(u-1)/z=t_1/(u+1)\,.$$
Similarly, the standard affine chart on $X_1$ about $F_{-1}$ is
$$U_{-1}=X_1\setminus F_1=\Spec A_1[(u+1)/z]=\Spec \mathbb{k}[z,u_{-1}]\cong\A^2\,$$ where $z$ and $u_{-1}=(u+1)/z=t_1/(u-1)$ are natural coordinates in $U_{-1}$.
The locally nilpotent vertical vector field $$\p_1= z\p/\p u+2u\p/\p t_1$$ on $X_1$ restricts to $\p/\p u_i$ in $U_i$, $i=1,-1$. The phase flow of $\p_1$ yields a free $\mathbb{G}_a$-action  on $X_1$. It is sent under $\rho_1$ to the field $d\rho_1(\p_1)=z\p_0=z\p/\p u$ on $X_0$.

The second Danielewski surface $X_2$ is obtained via the affine modification $\rho_2\colon X_2\to X_1$ with the divisor $z^*(0)$ on $X_1$ and the center $I=(z,t_1)\subset A_1$.  Thus, $\rho_2$ consists in blowing up $X_1$ at the points $x_{1.i}=(0, i,0)\in F_i$,  $i=1,-1$ (the origins of the affine planes $U_i\cong\A^2$, $i=1,-1$) and deleting the proper transforms of the fiber components $F_1$ and $F_{-1}$. Letting
$t_2=t_1/z$ one obtains $$A_2:=\cO_{X_2}(X_2)=A_1[t_1/z]=\mathbb{k}[z,u,t_2]/(z^2t_2-u^2+1)\,.$$ 
Once  again, $X_2$ is a GDF surface with a unique reducible fiber $z^*(0)$ consisting of two components of level 2.
Iterating this procedure we arrive at a sequence of Danielewski surfaces  $$X_m=\Spec \mathbb{k}[z,u,t_m]/(z^mt_m-u^2+1) ,\quad m=1,2,\ldots$$ 
along with a sequence of affine modifications fitting in (\ref{eq: seq-aff-modif-final}) $$\rho_m\colon X_m\to X_{m-1},\quad
(z,u,t_m)\mapsto (z,u,t_{m-1}) \quad\mbox{where}\quad t_{m-1}=zt_m\,.$$  The only special fiber $z^*(0)$ in $X_m$ is reduced and consists of two components of level $m$. The vector field $z^m\p/\p u$ on $X_0$ lifts to the locally nilpotent vertical vector field  on $X_m$, $$\p_m=z^m\p/\p u+2u\p/\p t_m\,.$$ Its phase flow defines a free $\mathbb{G}_a$-action on $X_m$.  The latter action restricts in a standard affine chart on $X_m$ to the standard $\mathbb{G}_a$-action via shifts in the vertical direction.

The extended divisor $D_{\rm ext, m}$ of a minimal completion $\bar\pi\colon\bar X_m\to\PP^1$ has dual graph 
\vspace{6mm}
\be\label{dia: Dani-graph}
\Gamma_{\rm ext, 0}:
\qquad \cou{0}{F_\infty}\lin\cou{0}{\bar S}\lin\cou{0}{\bar F_0}\qquad
\qquad\text{resp.,}\qquad\Gamma_{\rm ext, m}:
\quad \cou{0}{F_\infty}\lin\cou{0}{\bar S}\lin
\cou{\!\!\!\!\!\!\!\!\!\!\!\!\!-2}{\bar F_0}\nlin
\xbshiftup{}{\cF_1}
\llin\xbshiftright{}{\cF_{-1}}
\ee
\vspace{6mm}

\noindent  where $m\ge 1$ and a box stands for the chain $ [[-2,\ldots,-2,-1]]$ of length $m$ so that $\cF_i$ ends with the $(-1)$-feather $\bar F_i$ of level $m$, $i=1,-1$ (see Example \ref{exa: u-modif-Dan}). 
\eexa

\bexa\label{ex: DS} As an immediate generalization of the preceding example consider the surface $X_m$ in $\A^3$ with equation $z^mt_m-q(u)=0$ where $q\in \mathbb{k}[u]$ is a polynomial of degree $d\ge 2$ with simple roots. This is a GDF surface with projection $\pi=z|_{X_m}\colon X_m\to\A^1$. Letting $X_0=\A^2$ one has a sequence of affine modifications (\ref{eq: seq-aff-modif-final}) where $\rho_i\colon X_i\to X_{i-1}$, $(z,u,t_i)\mapsto (z,u,t_{i-1}=zt_i)$.  The vector field $z^m\p/\p u$ on $X_0=\A^2$ lifts to the locally nilpotent vertical vector field  on $X_m$, 
$$\p_m= z^m\p/\p u+q'(u)\p/\p t_m\,$$ which generates a free vertical $\mathbb{G}_a$-action on $X_m$. The dual graph $\Gamma_{\rm ext, m}$ of the pseudominimal completion $\bar\pi\colon\bar X_m\to\PP^1$ differs from the graph in diagram (\ref{dia: Dani-graph}) in one aspect: instead of two contractible chains $\mathcal{F}_1$ and $\mathcal{F}_{-1}$, it has $d$ such chains $\mathcal{F}_j$, $j=1,\ldots,d$ of the same length $m$ which are the branches of the fiber tree $\Gamma_0(\bar\pi)$ in the root $\bar F_0$.
\eexa

\bexa[\emph{GDF surfaces given by equations}]\label{ex: GDFs} In $\A^3$ with coordinates $(z,u,t_1)$ consider the surface $X_1=\{zt_1-q_1(u)=0\}$ where  $q=q_1\in \mathbb{k}[u]$ is a polynomial of degree $d\ge 1$ with simple roots $\alpha_1,\ldots,\alpha_d$. The projection $\pi_1=z|_{X_1}\colon X_1\to \A^1$ makes of $X_1$ a GDF surface with a unique reducible fiber $\pi_1^{-1}(0)$ consisting of $d$  components $F_1,\ldots, F_{d}$ on level 1, see Example \ref{ex: DS} with $m=1$. The projection $ (z,u,t_1)\mapsto (z,u)$ represents $X_1$ as a result of  the fibered modification $\rho_1\colon X_1\to X_0=\A^2$ which contracts  $F_{i}$ to the point $P_i=(0,\alpha_i)\in X_0$, $i=1,\ldots,d$. 

Let further $q_2\in \mathbb{k}[u,t_1]$ be such that, for each $i=1,\ldots,d$, either $q_2(\alpha_i,t_1)\in \mathbb{k}[t_1]$ has $m_i=\deg q_2(\alpha_i,t_1)>0$ simple roots $\beta_{i,1}, \ldots,\beta_{i,m_i}$, or $q_2(\alpha_i,t_1)=0$; in the latter case we let $m_i=0$.  Consider the complete intersection $V_2\subset \A^4$ given in coordinates $(z,u,t_1,t_2)$ by 
$$zt_1-q_1(u)=0,\quad zt_2-q_2(u,t_1)=0\,.$$ There is a unique irreducible component $X_2$ of $V_2$ which dominates the $z$-axis, while the other components are  disjoint affine planes contained in the hyperplane $z=0$. Let $P_{i,j}=(0,\alpha_i,\beta_{i,j})\in F_i$. 
The projection $ (z,u,t_1,t_2)\mapsto (z,u,t_1)$ defines a fibered modification $\sigma_2\colon X_2\to X_1$ along the divisor $z^*(0)=\sum_{i=1}^{d} F_i$ with a reduced center
$$\bigcup_{m_i=0} F_i \cup \bigcup_{m_i>0} (P_{i,1}+\ldots+P_{i,m_i})\,.$$
The projection $\pi_2=z|_{X_2}\colon X_2\to\A^1$ makes of $X_2$ a GDF surface with a unique reduced fiber over $z=0$. One has $X_2\setminus\pi_2^{-1}(0)\cong_{\A^1_*}\A^1_*\times\A^1$, that is, $\pi_2\colon X_2\to\A^1$ is a Danielewski-Fieseler surface as defined in \cite{Du0}.  The fiber $\pi_2^{-1}(0)\subset X_2$ has $d-c_2$ components $F_i$ of level 1 and $c_2=m_1+...+m_d$ components $F_{i,j}$ of level 2.

The graph $\Gamma_0(\pi_2)$ is a rooted tree with a root $\bar F_0$ of level 0, $d$ vertices $\bar F_1,\ldots,\bar F_d$ on level 1, and $c_2$ vertices $\bar F_{i,j}$, $i=1,\ldots,d, j=1,\ldots,m_i>0$ on level 2 where $\bar F_{i,j}$ is a neighbor of $\bar F_i$.  Clearly, any rooted tree $\Gamma$ of height 3 can be realized as $\Gamma_0(\pi_2)$. Moreover, any Danielewski-Fieseler surface $\pi\colon X\to \A^1$ with $\Gamma_0(\pi)\cong\Gamma$ can be obtained in this way. 

The vector field $z^2\p/\p u$ on $X_0=\A^2$ lifted to $X_2$ extends in $\A^4$ to a locally nilpotent vector field 
$$\p_2=z^2\frac{\p}{\p u}+zq_1'(u)\frac{\p}{\p t_1}+\left(z\frac{\p q_2}{\p u}(u,t_1)+\frac{\p q_2}{\p t_1}(u,t_1)\right)\frac{\p}{\p t_2} \,.$$  The associated vertical $\mathbb{G}_a$-action on $X_2$ is identical on the components $F_i$ of level 1 (which correspond to $m_i=0$) and is free on the components $F_{i,j}$ of level 2 and in the complement $X_2\setminus\pi_2^{-1}(0)$. 

By  a recursive procedure one can realize 
 in this way  any Danielewski-Fieseler surface (cf.\ \cite{Du1}). For instance, in the case of the Danielewski surface $X_m$ from Example \ref{ex: DS}
one arrives at a system
$$zt_1-p(u)=0,\quad zt_2-t_1=0,\quad\ldots,\quad zt_m-t_{m-1}=0\,,$$
which reduces to a single equation $z^mt_m-p(z)=0$  defining the original proper embedding $X_m\hookrightarrow\A^3$.
\eexa

\section{Relative flexibility}
\subsection{Definitions and the main theorem}
\bnota\la{not: rel-SAut} Let $\pi\colon X\to B$ be a GDF surface, and let $\cX=X\times\A^1$ be the cylinder over $X$.
We let $${\SAut}_B(\cX)=\langle \exp(\partial)\,|\,\partial\in\LND(\mathcal{O}_\cX(\cX)),\,\,\partial (z)=0\rangle\,$$ be the subgroup of the group $\Aut (\cX)$ generated by  the exponentials of locally nilpotent derivations in $\LND(\cX)$ which are automorphisms of $\cX$ over $B$, cf.\ Section \ref{ss:ML}.
 For a component $\cF=F\times\A^1$ of $z^*(0)$ in $\cX$ any $\varphi\in\SAut_B(\cX)$ 
stabilizes the standard affine chart $U_{F}\times\A^1\subset\mathcal{X}$  about $\cF$ with natural coordinates $(z,u_{F},v)$, see Proposition \ref{prop: aff-chart} and Definition \ref{def: local coordinates}. 
Furthermore, for any such $\cF$ the restriction $\varphi|_{U_F\times\A^1}$ preserves the volume form ${\rm d} z \wedge {\rm d} u_F \wedge {\rm d} v $ on $U_F\times\A^1$, that is, the Jacobian determinant of $\phi|_{U_F\times\A^1}$ equals 1, see \cite[Lem.\ 4.10]{AFKKZ}.
\enota

\bdefi[\emph{Relative flexibility}]\label{bpd1} 

We say that the cylinder $\cX=X\times\A^1$  is  {\em relatively flexible} (RF, for short)
if for any natural $s\geq 1$,
any collection $\fF$ of top level components $\cF=F\times\A^1$ of $z^*(0)$ in $\cX$,
and any  collection of pairs of ordered finite subsets $\Sigma_{\cF}=\{ x_1, \ldots , x_M \}$ and $\Sigma'_{\cF}=\{ x_1', \ldots , x_M' \}$ in $\cF$  of the same cardinality $M=M(\cF)>0$ where 
$\cF$ runs over $\fF$,  there exists an automorphism $\varphi\in\SAut_B \cX$
which 
satisfies the conditions
\begin{itemize}\item[($\alpha$)] $\varphi (x_\nu)=x_\nu'$ with prescribed 
volume preserving $s$-jets at $x_\nu$, $\nu=1, \ldots ,M(\cF)$ provided these jets preserve locally the fibers of $\cX\to B$; 
\item[($\beta$)] $\varphi|_{U_F\times\A^1}\equiv\id\mod\, z^s$ near $\cF=F\times\A^1$ for any $\cF\notin\fF$.
\end{itemize}  We say that {\em the condition RF$(l,s)$ holds for $X$} if the relative flexibility holds for the cylinder over the surface $X_l$ in \eqref{eq: seq-aff-modif-final} for a given $s\ge 1$. 
\edefi

\bdefi[\emph{Equivariant relative flexibility}]\label{bpd2} 
Let $\pi\colon X\to B$ be a marked GDF $\mu_d$-surface. 
We let $\mathcal{X}(k)$ 
denote the cylinder $\cX=X\times\A^1$ equipped with a product $\mu_d$-action where $\mu_d$ 
acts on the second factor 
via $(\zeta, v)\mapsto\zeta^kv$ for all $v\in\A^1$ and $\zeta\in\mu_d$.
Assume that  the collection $\fF$ of fiber components as in Definition \ref{bpd1} along with the finite sets $\Sigma=\bigcup_{F \in \fF} \Sigma_{\cF}$ and
 $\Sigma'=\bigcup_{F \in \fF} \Sigma'_{\cF}$ are $\mu_d$-invariant, and the correspondence $\Sigma_{\cF}\mapsto \Sigma_{\cF}'$ is $\mu_d$-equivariant.
 
We say that $\cX (k)$ is {\em $\mu_d$-relatively flexible} if one can choose a $\mu_d$-equivariant automorphism $\varphi\in\SAut_{\mu_d,B} \cX$ as in Definition \ref{bpd1}  provided that 
\begin{itemize}\item[($\alpha_1$)] the collection of prescribed $s$-jets  is $\mu_d$-invariant;
\item[($\alpha_2$)] if the stabilizer $\mu_e\subset\mu_d$ of $\cF$ is nontrivial and $x_\nu\in \Sigma_{\cF}$  is an isolated fixed point  of the $\mu_e$-action on $\cF$ then $x_\nu'=x_\nu$ and the prescribed $s$-jet at $x_\nu$  is the one of the identity.
\end{itemize} 

 If the cylinder $\cX_l(k)$ satisfies the above conditions for a given $s> 1$ then we say that {\em the $\mu_d$-equivariant condition RF$(l,k,s)$ holds for $X$.}
\edefi

The main result of this section is the following

\bthm\label{thm: RF} Consider a marked GDF $\mu_d$-surface $\pi\colon X\to B$ along with a trivializing sequence \eqref{eq: seq-aff-modif-final}. Then the $\mu_d$-equivariant condition RF$(l,-l,s)$ holds for $X$ with arbitrary  $s\ge 1$ and $l\in\{0,\ldots,m\}$. \ethm
The proof is done in Section \ref{ss: rtr}.

\subsection{Transitive group actions on Veronese cones}

Let us recall the notion of a saturated set of locally nilpotent derivations (see \cite[Def.\ 2.1]{AFKKZ}). For a vector field $\p$ on a variety $X$ and an automorphism $g\in\Aut X$ we let 
${\rm Ad}(g)(\p)=dg(\p)\circ g^{-1}$.

\bdefi[\emph{Saturation}]\label{def: saturation}  Consider an affine variety $X=\Spec A$ over $\mathbb{k}$. A set $\cN$ of locally nilpotent regular vector fields on $X$ (that is, of locally nilpotent derivations of the affine $\mathbb{k}$-algebra $A=\cO_X(X)$) is called {\em saturated} if 
\begin{itemize}\item[(i)] for any $\p\in\cN$ and $a\in\ker\p$ the  {\em replica} $a\p$ belongs to $\cN$, and \item[(ii)] ${\rm Ad}(g)(\p)\in\cN$ $\forall g\in G$, $\forall \p\in\cN$ where $G=\langle \exp\p\,\vert\,\p\in\cN\rangle\subset\Aut A$.
\end{itemize} \edefi

\blem\label{lem: saturation} Given  a set  $\cN\subset \Der A$ of locally nilpotent derivations satisfying {\rm (i)} consider the group $G\subset\Aut A$ as in {\rm (ii)} generated by $\exp\, (\cN)$. Then the set of locally nilpotent derivations $$\cN_1=\{{\rm Ad}(g)(\p)\,\vert\, g\in G,\,\,\p\in\cN\}\,$$ is saturated and generates the same group $G$. \elem

\bproof It is not difficult to see that $\cN_1$ satisfies (i). Let $G_1=\langle \exp\p\,\vert\,\p\in\cN_1\rangle$
be the group generated by $\cN_1$. We claim that $G_1=G$, and so,   (ii) follows by the chain rule. Indeed, an automorphism  $g\in \Aut X$ sends a  vector field $\p$ on $X$ into the vector field $\p'$ on $X$ such that $\p'(g(x))=dg(\p(x))$ $\forall x\in X$. Hence $\p'={\rm Ad}(g)(\p)$. 
On the other hand, if $\p$ is locally nilpotent with the phase flow 
$\exp(t\p)\in \Aut X$, $t\in k$,
then for the phase flow $\exp(t\p')\in \Aut X$, $t\in k$, one has
$\exp(t\p')=g\circ \exp(t\p)\circ g^{-1}$. Since $ \exp(t\p)\in G$ it follows that
$\exp(t\p')\in G$ for any $g\in G$ and $\p\in\cN$. Thus $\exp(t\p')\in G$ for any $\p'\in\cN_1$, and so, $G_1=G$, as claimed.
\eproof

\bsit\label{fln1}  Given $c,d\in\N$ consider the affine plane $\A^2=\Spec\mathbb{k}[u,v]$ 
equipped with the  $\mu_d$-action 
\be\la{eq:diag-action} \zeta . (u,v)=(\zeta^{-c} u , \zeta^{-c} v)\quad\forall\zeta\in\mu_d\,.\ee
This action is not effective, in general. However, it restricts to an effective action of the subgroup $\mu_e\subset\mu_d$ where $e=d/\gcd(c,d)$. The quotient $V_e=\A^2/\mu_e=\A^2/\mu_d$ is a Veronese cone. 

Consider also the locally nilpotent vector fields $\sigma_1=v\frac{\partial}{\partial u}$ 
and $\sigma_2=u\frac{\partial}{\partial v}$ on  $\A^2$. The associated one-parameter groups 
$(u,v)\mapsto (u+tv,v)$ and $(u,v)\mapsto (u,v+tu)$, $t\in k$, generate the standard $\SL_2$-action on $\A^2$.
Notice that $\sigma_1$ and $\sigma_2$ are $\mu_d$-invariant and the $\mu_d$-action  on $\A^2$ commutes with the $\SL_2$-action. Hence the $\SL_2$-action descends to the Veronese cone $V_e$. 
\esit

\bnota\label{not: G} Given $s\ge 2$ consider the $\mu_d$-invariant replicas $$\sigma_{1,f}=v^{ds}f(v^d)\sigma_1\,\,\,\mbox{of}\,\,\, \sigma_1\quad\mbox{and}\quad\sigma_{2,g}=u^{ds}g(u^d) \sigma_2\,\,\,\mbox{of}\,\,\, \sigma_2\quad\mbox{where}\quad f,g\in \mathbb{k}[t]\,.$$
Any vector field $\sigma_{1,f}$ vanishes modulo $v^s$
on the axis $v=0$. Hence $\phi_f:=\exp(\sigma_{1,f})$ fixes this axis pointwise along with its infinitesimal neighborhood of order $s$. Let $\psi_g=\exp(\sigma_{2,g})$. The subgroup
\be\label{eq: G} G=\langle \phi_f,\,\,\psi_g\,\vert \, f,g\in \mathbb{k}[t]\rangle \subset\SAut (\A^2)\,\ee 
 acts identically on the infinitesimal neighborhood of order $s$ of the origin,
is transitive in $\A^2\setminus\{\bar 0\}$, and commutes with the $\mu_d$-action \eqref{eq:diag-action} on $\A^2$.

Consider the normal subgroup $H\lhd G$ of all the automorphisms $\alpha\in G$ of the from 
\be\label{eq: alpha0} \alpha = \varphi_1 \cdot \psi_1 \cdot \ldots \cdot \varphi_\nu \cdot \psi_\nu\,\ee
where $\varphi_i=\phi_{f_i}$  and $\psi_i=\psi_{g_i}$, $i=1,\ldots,\nu$
verifying the condition
\be\label{eq: A} \phi_1 \cdot \phi_2 \cdot \ldots \cdot\phi_\nu ={\rm Id}\,. \ee 
\enota

\bprop\label{prop: Veronese} With the notation as before, let
$(O_1, \ldots , O_M)$ and $(O'_1, \ldots , O'_M)$ be two collections of distinct $\mu_d$-orbits  in $\A^2$ with $\card O_i=\card O_i'$ for $i=1,\ldots,M$. 
For every $i=1,\ldots,M$ choose a representative $x_i\in O_i$. In the case $e=1$ suppose in addition that the singletons $O_i$ and $O_i'$ are different from the origin. Then
there exists an automorphism $\alpha\in H$ 
 such that 
\begin{itemize}\item[(i)] $\alpha (O_i)=O'_i$ for $i=1,\ldots,M$ and
\item[(ii)] 
$\alpha$ has 
prescribed values of volume-preserving 
$s$-jets at the points $x_i$, $i=1,\ldots,M$ provided that for $e> 1$ and $O_i=\{\bar 0\}$ for some $i\in\{1,\ldots,M\}$ this prescribed $s$-jet at the origin is the $s$-jet of the identity. \footnote{Instead of prescribing the value of an $s$-jet in a single point of a $\mu_e$-orbit one might prescribe a $\mu_e$-invariant collection of $s$-jets at the points of the orbit.}
\end{itemize}
 \eprop

\bproof  
Consider the Veronese cone $V_e=\A^2/\mu_d=\Spec \mathbb{k}[u,v]^{\mu_d}$, and let $\rho\colon \A^2\to V_e$ be the quotient morphism. The cone $V_e$ is smooth outside the vertex $\bar 0\in V_e$. Since the $G$-action on $\A^2$ is transitive in $\A^2\setminus\{\bar 0\}$ and commutes with the $\mu_d$-action it descends to a $G$-action on $V_e$ transitive in $V_e\setminus\{\bar 0\}$. The $\mu_d$-invariant locally nilpotent vector fields $\sigma_{1,f}$ and $\sigma_{2,g}$ also descends to $V_e$. The set $\cN$ of all these vector fields on $V_e$ satisfies condition (i) of Definition \ref{def: saturation}. By Lemma \ref{lem: saturation} the group $G\subset\Aut V_e$ is generated as well by a saturated set $\cN_1$ of locally nilpotent vector fields  on the cone $V_e$. Therefore one can apply Theorems 2.2 and 4.14 from \cite{AFKKZ}.

Suppose first that $\{\bar 0\}$ is not among the $O_i$'s. By \cite[Thm.\ 2.2]{AFKKZ}, $G$ acts infinitely transitively in $V_e\setminus\{\bar 0\}$. It follows that there exists $\alpha\in G$ which sends the points $y_i=\rho(O_i)\in V_e$ into the points $y_i'=\rho(O'_i)$, $i=1,\ldots,M$. Acting in $\A^2$ this $\alpha$ transforms the orbit $O_i$ into $O'_i$ for every $i=1,\ldots,M$. Thus $\alpha$ verifies (i).

By \cite[Thm.\ 4.14]{AFKKZ} one can find $\alpha\in G$ verifying (i) with a prescribed 
volume-preserving $s$-jet at each point $y_i=\rho(O_i)\in V_e$, $i=1,\ldots,M$. Since $\rho$ is a local isomorphism near  a chosen point $x_i\in O_i$ over $y_i$ and near its image $\alpha(x_i)\in O'_i$  one may prescribe a
volume-preserving $s$-jet of $\alpha$ at $x_i$ with the given zero term 
$\alpha(x_i)$. 

If $e\ge 2$ and, say, $O_1=\{\bar 0\}$ then also $O'_1=\{\bar 0\}$. Indeed, any $\mu_e$-orbit different from $\{\bar 0\}$ contains $e>1$ points. Since $\sigma_{1,f},\,\sigma_{2,g}\equiv 0\mod\, (u,v)^s$ for any $f,g\in \mathbb{k}[t]$, see Notation \ref{not: G}, one has $\alpha\equiv\id\mod\, (u,v)^s$ for any  $\alpha\in G$. Thus automatically $\alpha(\bar 0)=\bar 0$ and, moreover, the $s$-jet at the origin of any $\alpha\in G$ is the one of the identity map.

In the case  $e=1$ one has $V_e=\A^2$ and every orbit $O_i$ and $O'_i$ is a 
singleton different from $\{\bar 0\}$ by assumption.
Then the argument in the proof 
works without change. 

It remains to find such an automorphism in the subgroup $H$. Due to the infinite transitivity of $G$ in $V_e\setminus\{\bar 0\}$ one can find $\beta\in G$  such that for every 
$i=1,\ldots,M$ the image $\beta(\rho(O_i'))$ is located in the line $v=0$  in $V_e$.
By the preceding there exists $\alpha\in G$ such that $\alpha\circ \beta(\rho(O_i))
=\beta(\rho(O_i'))$  for all $i=1,\ldots,M$ where $\alpha$ has prescribed volume preserving $s$-jets at these points. Since $\beta$ is volume-preserving (see \cite[Lemma 4.10]{AFKKZ}) one can find $\alpha_1=\beta^{-1}\circ\alpha\circ \beta\colon \rho(O_i)\mapsto \rho(O_i')$ with prescribed volume preserving $r$-jets at the points $y_i=\rho(O_i)$,  $i=1,\ldots,M$. 

Decomposing $\alpha$ as in (\ref{eq: alpha0})
consider the automorphism $\varphi_0=(\varphi_1 \cdot \ldots \cdot\varphi_\nu)^{-1}\in G$.
Since $ \varphi_0 \cdot \varphi_1 \cdot \ldots \cdot \varphi_\nu=\id$ replacing $\varphi_1$ by
 $\varphi_0 \circ \varphi_1$ one obtains an automorphism $\alpha'=\varphi_0\cdot\alpha\in H$. 
The $s$-jet of $\alpha'$ at each point $\beta(\rho(O_i))$ is the same as the one of $\alpha$. Indeed,
$\varphi_0 (\beta(\rho(O_i')))=\beta(\rho(O_i'))$ since $\beta(\rho(O_i')) \subset \{ v=0 \}$, and $\varphi_0$ is identical on the $s$th infinitesimal neighborhood of this line.  Since the subgroup $H\lhd G$ is normal  one has $\alpha_1'=\beta^{-1}\circ\alpha\circ \beta\in H$ where $\alpha_1'\colon \rho(O_i)\mapsto \rho(O_i')$ and $\alpha_1'$ has prescribed volume preserving $s$-jets at the points $\rho(O_i)$,  $i=1,\ldots,M$. 
Thus, $\alpha'_1\in H$ satisfies both (i) and (ii). 
\eproof

\subsection{Relatively transitive group actions on cylinders}\la{ss: rtr}

 \bnota\label{nota: q-inv-grps} Let $\pi\colon X\to B$ be a marked GDF $\mu_d$-surface, and let $X_l$ be one of the surfaces in  {\rm (\ref{eq: seq-aff-modif-final})}. We fix  quasi-invariant natural coordinates in the standard affine charts $U_{F}$ in $X_l$ so that the conventions of Lemma \ref{lem: nat-coord} and Remark \ref{rem: vvf-in-nat-coord}.2 are fulfilled.
 
 Fix also a $\mu_d$-invariant collection $\fF$ of top level fiber components in $X_l$. For $F\in \fF$ let
$\cF_0$ be the $\mu_d$-orbit of $F$ in $X_l$. For $s\ge 2$ let $\tilde u=\tilde u(\cF_0) \in \mathcal{O}_{X_l} (X_l)$ be a
$\mu_d$-quasi-invariant function of weight $-l$ which verifies conditions (i) and (ii) of Corollary \ref{cor: q-inv}.
Let $\p_l$ be the $\mu_d$-quasi-invariant vertical vector field  of weight $l$ on $X_l$ as in Lemma \ref{lem: vert-vect-field}. Given 
$f,g\in \mathbb{k}[t]$ consider the $\mu_d$-invariant locally nilpotent derivations of the algebra $\cO_{\cX_l}(\cX_l)$, 
\be\label{eq: sigmas} \tilde\sigma_{1,f}=
v^{ds+1}f(v^d)\p_l\quad\mbox{and}\quad \tilde\sigma_{2,g}=\tilde u^{ds+1}g(\tilde u^d) \p/\p v\,.\ee   Letting $F$ run over $\fF$ the corresponding automorphisms $$\tilde\phi_f=\exp(\tilde\sigma_{1,f})\quad\mbox{and}\quad \tilde\psi_g=\exp(\tilde\sigma_{2,g})$$ in $\SAut_{\mu_d,B}\, \cX_l(-l)$ generate a subgroup
\be\label{eq: tilde G} G_{\fF}=\langle\tilde\phi_f,\,\,\tilde\psi_g\,\vert \, f,g\in \mathbb{k}[t]\rangle \subset{\SAut}_{\mu_d,B}\,\cX_l(-l)\,\ee
contained in the centralizer of the cyclic group induced by the $\mu_d$-action on $\cX_l(-l)$.

Consider further the normal subgroup $ H_{\fF}\lhd G_{\fF}$,
\be\label{eq: tilde H} H_{\fF}=\{\tilde\alpha=\tilde\varphi_1 \cdot \tilde\psi_1 \cdot \ldots \cdot \tilde\varphi_\nu \cdot \tilde\psi_\nu\in G_{\fF}\,\vert\,\tilde\varphi_1 \cdot \ldots \cdot \tilde\varphi_\nu =\id\}\,.\ee 
For $\tilde u=\tilde u(\cF_0)$
one has $\tilde u\equiv 0\mod\, z^s$ in $U_{F'}\times\A^1$ near $\cF'=F'\times\A^1$ for any $F'\notin\fF$, see condition (ii) in Corollary \ref{cor: q-inv}. 
Hence $\tilde\psi_g\equiv \id\mod\, z^s$ in $U_{F'}\times\A^1$ near $\cF'$ for any $F'\notin\fF$ and any $g\in \mathbb{k}[t]$. Due to (\ref{eq: tilde H}) for any $\tilde\alpha\in H_{\fF}$ one has
\be\label{eq: id} \tilde\alpha|_{U_{F'}\times\A^1}= (\tilde\varphi_1 \cdot \tilde\psi_1 \cdot \ldots \cdot \tilde\varphi_\nu \cdot \tilde\psi_\nu)|_{U_{F'}\times\A^1} \equiv\id\mod\, z^s\quad \forall F'\notin \fF\,.\ee
\enota

\bdefi[\emph{s-reduced automorphism}]\label{def: s-reduced} Let $F\subset X_l$ be a special fiber component and $U_F$ be the standard affine chart about $F$ in $X_l$. Consider the affine chart $U_F\times\A^1$ about the affine plane $\cF=F\times \A^1 \simeq \A^2$ in the cylinder $\cX_l(-l)$. The subgroup $G_{\fF}\subset\SAut_B\,\cX_l(-l)$ preserves every fiber of the $\A^2$-fibration $\cX_l(-l)\to B$ and, moreover, every fiber component. Hence any $\alpha\in G_{\fF}$ preserves the affine chart $U_F\times\A^1$ (cf.\ Proposition \ref{prop: aff-chart}). Choosing natural coordinates $(z,u,v)$ in $U_F\times\A^1$ the restriction $\alpha|_{U_F\times\A^1}$ can be written as
$$\alpha|_{U_F\times\A^1}\colon (z,u,v)\mapsto \left(z, \sum_{i=0}^\infty z^if_i(u,v), 
\sum_{i=0}^\infty z^ig_i(u,v)\right).$$ 
We say that $\alpha$ is {\em $s$-reduced} if $f_1=\ldots =f_s = g_1 = \ldots = g_s=0$,
that is,
\be\la{eq:reduced} \alpha (z,u,v) \equiv (z, f_0 (u,v), g_0 (u,v)) \,\, {\rm mod} \, z^{s}\ee
in any such affine chart $U_F\times\A^1$ in $\cX_l(-l)$.
\edefi

 \blem\label{lem: composing s-reduced}  
\begin{itemize}\item[(a)] A composition of $s$-reduced automorphisms  is again $s$-reduced.
\item[(b)] Any automorphism $\tilde\alpha\in H_{\fF}$ is  $s$-reduced.
\end{itemize} 
\elem

\bproof The proof of (a) is straightforward. To prove (b), due to (a) it suffices to show that  $\tilde\varphi=\tilde\varphi_f$ and $\tilde\psi=\tilde\psi_g$ are $s$-reduced for any $f,g\in \mathbb{k}[t]$. However, the latter is true only in the top level affine charts. Indeed, in  a standard affine chart $U_F$ on level $l'\le l$ in $X_l$ one has $\p_l|_{U_F}=z^{l-l'}\p/\p u$ 
(see Remark \ref{rem: vvf-in-nat-coord}.1). Hence 
\be\label{eq: tilde-phi} \tilde\varphi|_{U_F\times\A^1} = \exp{\left(v^{ds+1}f(v^d)\p/\p u\right)} \colon (z,u,v)\mapsto \left(z,u+z^{l-l'}v^{ds+1}f(v^d),v\right)\,
\ee 
is $s$-reduced if $l'=l$, that is, in any top level affine chart. 

Since $\tilde u|_{U_F}\equiv u\mod\, z^s$ if $l'=l$ and $\tilde u|_{U_F}\equiv 0\mod\, z^s$ otherwise, near the affine plane $\cF\subset U_F\times\A^1$ one has
\be\label{eq: tilde-psi} \tilde\psi|_{U_F\times\A^1} = \exp{\left(\tilde u^{ds+1}f(\tilde u^d)\p/\p v\right)} \colon (z,u,v)\mapsto \left(z,u, v+u^{ds+1}g(u^d)\right)\mod\, z^s\,\ee  if $l'=l$ and
$\tilde\psi|_{U_F\times\A^1} \equiv \id\mod\, z^s$  otherwise. In particular, any $\tilde\psi \in G_{\fF}$ is $s$-reduced. It follows that any automorphism $$\tilde\alpha= \tilde\varphi_1 \cdot \tilde\psi_1 \cdot \ldots \cdot \tilde\varphi_\nu \cdot \tilde\psi_\nu\in G_{\fF}\,$$ is $s$-reduced in every top level affine chart $U_F\times\A^1$. If $F$ has level $l'<l$ then $\tilde\psi_i|_{U_F\times\A^1}\equiv\id\mod\, z^s$ $\forall i=1,\ldots,\nu$. Hence for any $\tilde\alpha\in H_{\fF}$ one has by \eqref{eq: id} $$\tilde\alpha|_{U_F\times\A^1}= (\tilde\varphi_1 \cdot \ldots \cdot \tilde\varphi_\nu)|_{U_F\times\A^1} \equiv\id\mod\, z^s\,,$$ that is, $\tilde\alpha$ is $s$-reduced.
\eproof

\bprop\label{prop: rel-flexibility}  Let $\fF$ be a $\mu_d$-invariant collection of top level components $\cF=F\times\A^1\subset \cX_l(-l)$ of $z^*(0)$, and let $\Sigma,\,\Sigma'\subset\bigcup_{\cF\in\fF} \cF$ be two $\mu_d$-invariant finite sets which meet every special fiber component  $\cF\in\fF$ with the same positive cardinality.  Assume that for some  $s>0$ at each point $x\in\Sigma\cap\cF$ we are given a volume preserving two-dimensional
$s$-jet $j_x$ of an automorphism $\cF\to\cF$
such that 
\begin{itemize} \item the zero term of $j_x$ runs over $\Sigma'$ when $x$ runs over $\Sigma$; 
\item the collection $(j_x)_{x\in\Sigma}$ commutes with the $\mu_d$-action on $\cX_l(-l)$; 
\item if $e=e(\cF)=d/\gcd(d,l)> 1$, see {\rm \ref{fln1}}, and $x_0\in\cF\cap\Sigma$ is a fixed point of the stabilizer $\mu_e$ of $\cF$ in $\mu_d$ then $j_{x_0}$ is the $s$-jet of the identity.
\end{itemize} 
Then there exists a {\rm (}$\mu_d$-equivariant{\rm )} 
 automorphism $\tilde\alpha\in H_{\fF}$ 
such that
\begin{itemize}\item[(i)] $\tilde\alpha (\Sigma)=\Sigma'$;
\item[(ii)] 
$\tilde\alpha$ has the
prescribed two-dimensional $s$-jets at the points of $\Sigma$; 
\item[(iii)] $\tilde\alpha|_{U_F\times\A^1}\equiv\id\mod\, z^s\quad \forall \cF\notin \fF$.
\end{itemize}
\eprop

\bproof Let $\cF\in\fF$, and let $\mu_d(\cF)$ be the $\mu_d$-orbit of $\cF$ in $\cX_l(-l)$. It suffices to construct such an automorphism $\tilde\alpha\in H_{\fF}$ assuming that $\fF$ consists of the components from the orbit $\mu_d(\cF)$. Indeed, then $\tilde\alpha\in H_{\fF}$ coincides with the identity modulo $z^s$ near any special fiber component $\cF'\notin\fF$. Composing such automorphisms $\tilde\alpha$ for different top level orbits 
one obtains a desired automorphism in the general case. 

Furthermore, if (i) and (ii) hold for a special fiber component $\cF$
then they automatically hold for any special fiber component $\cF'\in\mu_d(\cF)$ due to the $\mu_d$-invariance of the conditions and the $\mu_d$-equivariance of the automorphisms $\tilde\alpha\in H_{\fF}$. Hence it suffices to take care of a particular $\cF\in\fF$ equipped with two collections of orbits $\{O_i\cap \cF\}_{i=1,\ldots,\nu}$ and $\{O'_i\cap \cF\}_{i=1,\ldots,\nu}$ of the stabilizer $\mu_e$ of $\cF$ in $\mu_d$, see Proposition \ref{prop: Veronese}.  Let  $U_F\times\A^1$ be the standard affine chart about $\cF$ equipped with $\mu_e$-quasi-invariant natural local coordinates $(z,u_F,v)$.  Due to Remark  \ref{rem: vvf-in-nat-coord}.2 for $e=1$ one may assume $O_i\neq\{\bar 0\}$ $\forall i$ as needed in  Proposition \ref{prop: Veronese}, see Notation \ref{nota: q-inv-grps}.

Comparing \eqref{eq: tilde-phi} and \eqref{eq: tilde-psi} with \eqref{eq: G} in Notation \ref{not: G} one can see that the automorphisms $\tilde\varphi_f, \tilde\psi_g\in H_{\fF}$ restrict to
$$\tilde\varphi_f|_{\cF} = \varphi_f\quad\mbox{and}\quad  \tilde\psi_g|_{\cF}=\psi_g\,,$$ respectively, where $\varphi_f$ and $\psi_g$ run over the generators of the subgroup $G\subset\SAut_{\mu_d} (\cF)$ when $f,g$ run over $\mathbb{k}[t]$. Let $H\lhd G$ be as in \ref{not: G}. Applying Proposition \ref{prop: Veronese} one can find an automorphism $\alpha = \varphi_1 \cdot\psi_1 \cdot \ldots \cdot \varphi_\nu \cdot \psi_\nu\in H$
satisfying  in the affine plane $\cF\cong\A^2$ conditions (i) and (ii) of this proposition. Extending every $\varphi_i$ to $\tilde\varphi_i\in H_{\fF}$ and $\psi_i$ to $\tilde\psi_i\in H_{\fF}$ one obtains an $s$-reduced automorphism $\tilde\alpha = \tilde\varphi_1 \cdot \tilde\psi_1 \cdot \ldots \cdot \tilde\varphi_\nu \cdot \tilde\psi_\nu\in H_{\fF}$, see Lemma \ref{lem: composing s-reduced}(b). Since $\tilde\alpha$ also satisfies (\ref{eq: id}) in Notation \ref{nota: q-inv-grps} then (iii) holds, and so, $\tilde\alpha$ is a  desired automorphism.  
\eproof

\noindent {\em Proof of Theorem {\rm \ref{thm: RF}}.} Let $\pi\colon X\to B$ be a marked GDF $\mu_d$-surface. We have to show that the $\mu_d$-equivariant
condition $\RF(l,-l,s)$ holds for $X$ whatever are $s\ge 1$ and $l\in\{0,\ldots,m\}$. It suffices to show that, given any $\mu_d$-invariant collection $\fF$ of top level special fiber components in $X_l$ and any two finite sets $\Sigma, \Sigma'\subset \bigcup_{F\in\fF}\cF$ with the same $\mu_d$-orbit structure and
with $\card \Sigma_F=\card\Sigma'_F>0$ $\forall F\in \fF$ where $\Sigma_F=\Sigma\cap \cF$, there exists $\varphi\in\SAut_B\,\cX_l(-l)$ such that the $\mu_d$-equivariant versions of conditions ($\alpha$) and ($\beta$) in Definition \ref{bpd1} are fulfilled. 

By Proposition \ref{prop: rel-flexibility} one can find $\varphi\in H_{\fF}\subset \SAut_{\mu_d,B}\,\cX_l(-l)$ verifying (i) and (ii) of Proposition \ref{prop: Veronese} and condition  (\ref{eq: id}). That is, $\varphi$ is $\mu_d$-equivariant, $s$-reduced, verifies (\ref{eq: id}), sends $\Sigma$ to $\Sigma'$, and has prescribed 2-dimensional $r$-jets (in the vertical planes) in the (chosen) points on each $\mu_d$-orbit in $\Sigma$. Since $\varphi$ is $s$-reduced it has as well the prescribed volume preserving three-dimensional $s$-jets in the given points. Hence $\varphi$ satisfies conditions ($\alpha$), ($\alpha_1$), and ($\alpha_2$)  of Definitions \ref{bpd1} and \ref{bpd2}.
Due to  \eqref{eq: id}, $\varphi$ satisfies also condition ($\beta$)
of Definition \ref{bpd1}.
\qed

\subsection{A relative Abhyankar-Moh-Suzuki Theorem} \label{ss: AMS}

We need in the sequel the following version of the Abhyankar-Moh-Suzuki Epimorphism Theorem. 

\bprop\la{prop: autm} Let $\pi\colon X\to\ B$ be a GDF surface, let $\{F_1,\ldots, F_t\}$ be a collection of top level special fiber components in $X$, and let $\cF_i=F_i\times\A^1\cong\A^2$, $i=1,\ldots,t$, be the corresponding components of the divisor $z^*(0)$ in the cylinder $\cX$. For every $i=1,\ldots,t$ we fix in $\cF_i$ a curve $C_i\cong\A^1$. Then there exists an automorphism $\alpha\in\SAut_B(\cX)$ such that $\alpha(C_i)=F_i\times\{0\}$, $i=1,\ldots,t$.
\eprop

\bproof  Choose $i\in\{1,\ldots,t\}$, and let $F=F_i,\,\cF=\cF_i$, and $C=C_i\subset\cF$. Our assertion follows  by induction on $i$ from the next claim.

\smallskip

\noindent {\bf Claim.} \emph{There exists an automorphism $\beta=\beta_i\in\SAut_B(\cX)$ such that $\beta(C)=F\times\{0\}$ and $\beta(F'\times\{0\})=F'\times\{0\}$ for any special fiber component $F'\neq F$.}

\smallskip

\noindent Indeed, to deduce the assertion it suffices to apply this claim successively for $i=1,\ldots,t$. 

\smallskip

\noindent\emph{Proof of the claim.} By Corollary \ref{cor: q-inv} one can find $\tilde u\in\cO(X)$  such that 
\begin{itemize} \item[(i)] $\tilde u|_{F}=u_F$ where $u_F$ is an affine coordinate on $F$; \item[(ii)]  
$\tilde u|_{F'}=0$ for any $F'\neq F$. \end{itemize}
Consider the locally nilpotent derivations on $\cO_\cX(\cX)$, $$\sigma_1=\p_l\quad\mbox{and}\quad  \sigma_2=\tilde u\frac{\p}{\p v}\,$$
where $l$ is the highest level of the  special fiber components of $X$ and $\p_l$ is a vertical locally nilpotent vector field on $X$ as in Lemma \ref{lem: vert-vect-field} so that $\p_l(z)=0$ and $\p_l|_F=\p/\p u_F$.
Consider  the replicas 
$$\sigma_{1,f}=f(v)\sigma_1\quad\mbox{and}\quad  \sigma_{2,g}=g(\tilde u)\sigma_2 \quad\mbox{where}\quad f,g\in \mathbb{k}[t]\,.$$ Their exponentials $$\varphi_f=\exp(\sigma_{1,f}),\quad \psi_g=\exp(\sigma_{2,g})\in{\SAut}_{B}\cX$$  generate 
a subgroup $\cH\subset\SAut_B\cX$.  In the affine plane $\cF\cong\Spec\mathbb{k}[u_F,v]$ one has
$$\varphi_f|_{\cF} \colon (u_F,v)\mapsto (u_F+f(v),v) \quad\mbox{and}\quad \psi_g|_{\cF} \colon (u_F,v)\mapsto (u_F,v+u_Fg(u_F))\,.$$ In particular, $\cH|_{\cF}$ contains all the transvections, hence also the group $\SL(2,k)$. For $F'\neq F$ by virtue of (ii) the group $\cH|_{\cF'}$ is generated by the shears $\varphi_f|_{\cF'}$. It follows that 
\begin{itemize} \item $\cH|_{\cF}=\SAut \cF\cong \SAut\A^2$ and \item  the coordinate line $F'\times\{0\}\subset\cF'$ stays $\cH$-invariant  for any  $F'\neq F$.  \end{itemize}
 Now the claim follows by the Abhyankar-Moh-Suzuki Theorem. 
\eproof

The next lemma allows to interchange the $u$- and $v$-axes in the top level special fiber components of $\cX\to B$. 

\blem\label{lem: tau} Let $\pi\colon X\to B$ be a marked GDF $\mu_d$-surface, let $X_l$ be one of the surfaces in  {\rm (\ref{eq: seq-aff-modif-final})}, and let $\{(z,u_F)\}$ be a quasi-invariant system of natural local coordinates  in the standard local charts $U_F$ about the special fiber components $F$ in $X_l$. Given $s>1$ 
there exists a $\mu_d$-equivariant  automorphism $\tau\in\SAut_B\,\cX_l(-l)$ such that \begin{itemize}\item $\tau |_{U_F\times\A^1}\colon (z,u_F,v)\mapsto (z,v,-u_F) \mod\, z^s$ for any top level $F$; \item
$\tau|_{U_{F}\times\A^1}=\id \mod\, z^s$  for any  $F$ of lower level. \end{itemize}
\elem

\bproof 

Likewise in (\ref{eq: sigmas}) we let
\be\label{eq: sigmas-bis} \tilde\sigma_{1}= v\p_l\quad\mbox{and}\quad \tilde\sigma_{2}=-\tilde u\p/\p v\,\ee where $\p_l$ is the vertical vector field on $X_l$ as in Lemma \ref{lem: vert-vect-field} and $\tilde u\in\cO_{\cX_l(-l)}(\cX_l(-l))$ is a $\mu_d$-quasi-invariant  of weight $-l$ verifying conditions (i) and (ii) of Corollary \ref{cor: q-inv}. Letting $\tilde\varphi=\exp{\left(\tilde\sigma_{1}\right)}\quad\mbox{and}\quad \tilde\psi=\exp{\left(\tilde\sigma_{2}\right)}$ by virtue of (i) and (ii) one obtains
$$\tilde\phi |_{U_{F}\times\A^1}\colon (z,u_F,v)\mapsto \left(z,u_F+v, v\right)\mod\, z^s$$
and
$$\tilde\psi |_{U_{F}\times\A^1}\colon (z,u_F,v)\mapsto \left(z,u_F, v-u_F\right)\mod\, z^s\,$$  if $F$ is of top level and
$\tilde\psi|_{U_{F}\times\A^1} \equiv \id\mod\, z^s$  otherwise, cf.\ (\ref{eq: tilde-phi}) and (\ref{eq: tilde-psi}). Letting $\tau=\tilde\phi\tilde\psi\tilde\phi$ one gets
$$\tau |_{U_{F}\times\A^1} \colon (z,u_F,v)\mapsto \left(z,v, -u_F\right)\mod\, z^s$$
if $F$  is of top level and $\tau |_{U_{F}\times\A^1} \equiv \id\mod\, z^s$ otherwise.
\eproof

We need as well the following versions of Lemma \ref{lem: tau}.

\blem\la{lem: tau-1} Under  the assumptions of  Lemma {\rm \ref{lem: tau}} consider a $\mu_d$-invariant subset $\Upsilon\subset\{b_1,\ldots,b_n\}=z^{-1}(0)$. Given  a collection $\fF_\Upsilon(l)$ of  top level special fiber components in $\pi_l^{-1}(\Upsilon)\subset X_l$ there exists a $\mu_d$-equivariant  automorphism $\tau\in\SAut_B\,\cX_l(-l)$ such that
\be\la{eq: tau-11} \tau |_{U_{F}\times\A^1} \colon (z,u_F,v)\mapsto \left(z,v, -u_F\right)\mod\, z^s\ee
if $F\in \fF_\Upsilon(l)$ and $\tau |_{U_{F}\times\A^1} \equiv \id\mod\, z^s$ otherwise.
\elem

\bproof Choose a $\mu_d$-invariant $h\in\cO_B(B)$ such that $h-1\equiv 0\mod\,z^s$ near each point $b_i\in\Upsilon$ and $h\equiv 0\mod z^s$ near each point $b_i\notin\Upsilon$. Denote by the same letter the lift of $h$ to $\cX_l$. For the regular $\mu_d$-quasi-invariant vector field $\p_{l}$  of weight $l$ on  $\cX_l(-l)$ 
one has $\p_{l}(h)=0$ and $\p h/\p v=0$. Let $\tilde u\in\cO(X)$ be as in Corollary \ref{cor: q-inv} and $\tilde\sigma_{i}$, $i=1,2$ be as in \eqref{eq: sigmas-bis}. Then the locally nilpotent vector fields
\be\label{eq: sigmas-bis-bis} \tilde\sigma_{1,h}= h\tilde\sigma_{1}\quad\mbox{and}\quad \tilde\sigma_{2,h}=h\tilde\sigma_{2}\,\ee  
 on $\cX_l(-l)$
are $\mu_d$-invariant. Using these derivations instead of $\tilde\sigma_{i}$, $i=1,2$ and proceeding as in the proof of Lemma \ref{lem: tau} yields the result. \eproof

\section{Rigidity of cylinders upon deformation of surfaces}
\subsection{Equivariant Asanuma modification}
In the next lemma we introduce an equivariant version of the Asanuma modification. For the reader's convenience we repeat in (a) the statement of Lemma \ref{lem: as-trick0}. 

\begin{lem}\label{lem: as-trick} Let $\pi\colon X\to B$ be a GDF surface, and let $\rho\colon X'\to X$
 be a  
fibered modification along a reduced principal divisor $\div f$ where $f\in\pi^*\cO_B(B)\setminus\{0\}$ with center a reduced zero-dimensional subscheme $\mathbb{V}(I)$ where $I\subset \cO_X(X)$ is an ideal, see Definition {\rm \ref{def: simple-aff-modif}}. 
Consider the principal divisor $\cD=\mathbb{V}(f)\times\A^1$ on the cylinder  $\cX=X\times\A^1$
and the ideal $J=(I,v)\subset \mathcal{O}_{\cX}(\cX)$ with support $\mathbb{V}(I)\times\{0\}\subset\mathbb{V}(f)\times\{0\}$. 
Then the following holds.
\begin{itemize}\item[(a)]
The cylinder $\cX'=X'\times\A^1$  is isomorphic to the affine modification $Z$ of 
$\cX$ along the divisor $\cD$ with the center $J$.
This isomorphism fits in the commutative diagram
\be\label{diagr: Asanuma-modif}
 \bdi
 \cX'&\rTo<{\cong} & Z & \rTo<{}  & \cX\\
 & \rdTo & \dTo>{} &\ldTo & \\
 & & B & & 
\edi\ee
where the vertical arrows are $\A^2$-fibrations over $B$. 
\item[(b)] Assume that  the modification $\rho\colon X'\to X$ is equivariant with respect to actions of 
a finite group $G$ on $X,X'$, 
and $B$ and, moreover, the ideal $I$ is $G$-invariant and the function $f$ is $G$-quasi-invariant and 
belongs to a character 
$\chi\in G^\vee$. Define the $G$-action on the factor $\A^1$ of the cylinder $\cX=X\times\A^1$ via 
the multiplication by 
a character $\lambda\in G^\vee$. Then the morphisms  in {\rm (\ref{diagr: Asanuma-modif})}
 are $G$-equivariant
where $G$ acts on the factors $\A^1$ of the cylinder $\cX'=X'\times\A^1$ 
via the multiplication by $\lambda/\chi$.  
In particular, if $G=\mu_d$, $\chi\colon\zeta\mapsto \zeta^t$, and $\lambda\colon\zeta\mapsto 
\zeta^k$ then 
$\lambda/\chi\colon\zeta\mapsto \zeta^{k-t}$ $\forall\zeta\in\mu_d$.
\end{itemize}
\end{lem}

\begin{proof} For the proof of (a) see Lemma \ref{lem: as-trick0}.
Statement (b) follows since under its assumptions the variable $v'$ in the proof of Lemma \ref{lem: as-trick0} belongs to $\lambda$, hence $v=v'/f$ belongs to $\lambda/\chi$. 
\end{proof}

\bdefi[\emph{Asanuma modifications}]\label{defi: asanuma} The upper line in (\ref{diagr: Asanuma-modif}) yields an 
affine modification $\cX'\to \cX$
called an 
{\em Asanuma modification of the first kind}.
Its center is a reduced zero-dimensional subscheme of $\cX$.
 
We call an  {\em Asanuma modification of the second kind}  an affine modification $\cX''\to\cX$ of the cylinder $\cX=X\times\A^1$ over a marked GDF surface $\pi\colon X\to B$ along the divisor $\cD =(f\circ\pi)^*(0)$  on $\cX$ where $f\in\cO_B(B)\setminus\{0\}$ with a one-dimensional center $\mathbb{V}(I)\subset X\times\{0\}$ where $I=(f,v)\subset\cO_{\cX}(\cX)$. Due to the next lemma the latter modification  results in a cylinder isomorphic to $\cX$ over $B$.
\edefi

\bsit
\label{nota: cylinder} Let $\pi\colon X\to B$ be  a marked GDF $\mu_d$-surface with a marking $z\in\mathcal{O}_B(B)\setminus\{0\}$. Recall (see Definition \ref{bpd2}) that
$\mathcal{X}(k)$ stands for the cylinder $\cX=X\times\A^1$ equipped with a product $\mu_d$-action where $\mu_d$ 
acts on the second factor 
via $(\zeta, v)\mapsto\zeta^kv$ for $v\in\A^1$ and $\zeta\in\mu_d$. By abuse of notation 
we still denote by  
$\pi$ the $\mu_d$-equivariant projection of the induced  $\A^2$-fibration $\mathcal{X}(k)\to B$.
\esit

\blem\label{lem: As-2d-kind} In the notation of {\rm \ref{defi: asanuma}}--{\rm \ref{nota: cylinder}} consider  an Asanuma modification of the second kind $\cX''\to\cX$. Then the following hold. \begin{itemize}\item[(a)]  There is an isomorphism $\cX''\cong_B\cX$. 
\item[(b)] If $\pi\colon X\to B$ is a marked GDF $\mu_d$-surface with a marking $z\in\cO_{B}(B)\setminus\{0\}$, $f=z$, and $\cX=\cX(k)$ then $\cX''=\cX''(k-1)$.
\item[(c)] Let things be as in {\rm (b)}. Consider a second marked GDF $\mu_d$-surface $\pi'\colon X'\to B$  with the same marking $z\in\cO_B(B)\setminus\{0\}$, and let $\cX'=X'\times\A^1$ where $\A^1=\Spec \mathbb{k}[v']$. Assume that for some natural $r$ there is an equivariant isomorphism $\varphi_r\colon\cX(r)\stackrel{\cong_{\mu_d,B}}{\longrightarrow}\cX'(r)$ such that $\varphi_r^*(v')\equiv v \mod\, z^s$ where $s>d$. Then for any  $k\in\ZZ$ there is an equivariant isomorphism $\varphi_{k}\colon\cX(k)\stackrel{\cong_{\mu_d, B}}{\longrightarrow}\cX'(k)$ such that $\varphi_k^*(v')\equiv v \mod\, z^{s-d}$.
\end{itemize}
\elem

\bproof (a) Indeed,
the affine modification $\cX''\to\cX$ amounts to 
\be\la{eq:modif-II} \cO_{\cX}(\cX)\hookrightarrow\cO_{\cX''}(\cX'')=\cO_{\cX}(\cX)[v/f]=
\cO_X(X)[v'']
\,\,\,\mbox{where}\,\,\,v''=v/f\,.\ee 
(b) Under the assumptions of (b) one has
$\zeta . v''=\zeta^{k-1} v''$ for any $\zeta\in\mu_d$. 

(c) Consider first the case $k=r-1$. Let $I=(z,v)\subset\cO_{\cX}(\cX)$ and $I'=(z,v')\subset\cO_{\cX'}(\cX')$. Under our assumptions one has $\varphi_r^*(I')=I$. By virtue of Lemma \ref{lem: preserving-isomorphisms} the isomorphism $\varphi_r$ lifts to an equivariant isomorphism, say, $\varphi_{r-1}$ of the affine modifications of the cylinders $\cX$ and $\cX'$ along the divisors $z^*(0)$ with the ideals $I$ and $I'$, respectively. By (b) this leads to a commutative diagram 
$$\bdi
\cX(r-1) &\rTo<{\varphi_{r-1}}>{\cong_{\mu_d,B}} & \cX'(r-1) \\
\dTo<{} &&\dTo>{} \\
\cX(r) &\rTo>{\varphi_r}<{\cong_{\mu_d,B}} & \cX'(r)\, 
\edi$$
where the vertical arrows are the corresponding Asanuma modifications of the second kind and $\varphi_{r-1}^*(v'/z)\equiv v/z \mod\, z^{s-1}$. Since the sequences $(\cX(k))_{k\in\ZZ}$ and $(\cX'(k))_{k\in\ZZ}$ are both periodic with period $d$ the recursion on $k$ ends the proof. 
\eproof

\brem\la{rem: as-mod} Let $\rho\colon X'\to X$ be a fibered modification as in Lemma \ref{lem: as-trick}. Consider the  product  modification of cylinders $\sigma=\rho\times\id\colon \cX'\to\cX$ followed by the Asanuma modification of the second kind $\cX''\to\cX'$ with $f=z$. This yields an affine modification $\cX''\to\cX$ 
factorized as in 
 Remark \ref{rem: aff-modif}.2. Identifying $\cX'$ and $\cX''$ via an  isomorphism as in Lemma \ref{lem: As-2d-kind} gives an Asanuma modification of the first kind $\tilde\rho\colon\cX'\to\cX$ 
such that $\rho=\tilde\rho|_{X'\times\{0\}}$. Under this correspondence the conclusions of Lemmas \ref{lem: as-trick}(b) and \ref{lem: As-2d-kind}(c) agree  in the $\mu_d$-equivariant setting. 
\erem

\subsection{Rigidity of cylinders under deformations of GDF surfaces}\label{ss: deformation-equivalent cylinders}
Form Lemma \ref{lem: as-trick} we deduce the following corollary. 

\bcor\label{cor: asan-sequence} \begin{itemize}\item[(a)] Consider a marked GDF $\mu_d$-surface $\pi\colon X\to B$ along with a trivializing $\mu_d$-equivariant sequence {\rm (\ref{eq: seq-aff-modif-final})} of  fibered modifications, see Corollary {\rm \ref{cor: final decomposition}(b)}. 
Given $l\in\{1,\ldots,m\}$ and $k\in\ZZ$ the fibered modification $\rho_l\colon X_l\to X_{l-1}$ as in {\rm (\ref{eq: seq-aff-modif-final})} along the divisor $z^*(0)$ with a center, say, $I_l$ induces a 
$\mu_d$-equivariant Asanuma modification of the first kind
$\tilde\rho_l\colon \mathcal{X}_l(k) \to \mathcal{X}_{l-1}(k+1)$
over $B$ along the divisor $z^*(0)$ on $\mathcal{X}_{l-1}$ with the center $J_l=(I_l,v)$, cf.\  Lemma {\rm \ref{lem: as-trick}}.
\item[(b)] Consequently,
{\rm (\ref{eq: seq-aff-modif-final})} yields a sequence of $\mu_d$-equivariant affine modifications
\be\label{eq: sec-Asanuma-modif} 
\mathcal{X}_{m}(-m)\stackrel{\tilde\rho_m}{\longrightarrow}\mathcal{X}_{m-1}(-m+1)\stackrel{}{\longrightarrow}\ldots\stackrel{}{\longrightarrow}  
\mathcal{X}_{1}(-1)\stackrel{\tilde\rho_1}{\longrightarrow} \mathcal{X}_{0}(0) = (B \times \A^2)(0)\,.\ee 
\end{itemize}
\ecor

\bproof The statement of (a) follows  by Lemma  
\ref{lem: as-trick} and (b) is immediate from (a).
\eproof

The next theorem is the
main result of this subsection.

\bthm\label{thm: GDF-cancellation-fixed-graph} 
Let $\pi\colon X\to B$ and $\pi'\colon X'\to B$ be two 
marked GDF $\mu_d$-surfaces with the same $\mu_d$-quasi-invariant marking  $z\in\cO_B(B)\setminus\{0\}$ of weight $1$.  
Assume that for some trivializing $\mu_d$-equivariant completions
$(\hat X, \hat D)$ and $(\hat X', \hat D')$ of $X$ and $X'$, respectively, 
the graph divisors $\cD(\hat\pi)$ and $\cD(\hat\pi')$ 
are $\mu_d$-equivariantly isomorphic {\rm(}see Definition {\rm \ref{def: graph-divisor}}{\rm)}.
Then for any $k\in\ZZ$ there is a $\mu_d$-equivariant isomorphism  $\mathcal{X}(k)\cong_{\mu_d,B}\mathcal{X}'(k)$.
In particular, 
$\cX(0)\cong_{\mu_d,B}\cX'(0)$.
\ethm

\bproof The trivializing sequences (\ref{eq: sec-Asanuma-modif}) 
associated with the GDF surfaces $X$ and $X'$, respectively,
start both with the same product $\cX_0(0)=(B\times\A^2)(0)=\cX'_0(0)$. Using Proposition \ref{prop: steps} below one shows
by induction on $l$  
 that for any $l=0,\ldots,m$ there is a $\mu_d$-equivariant isomorphism   
$\varphi_l\colon\mathcal{X}_l(-l)\stackrel{\cong_{\mu_d, B}}{\longrightarrow}\mathcal{X}_l'(-l)$.
In particular, for $l=m$ one obtains  an isomorphism $\varphi_m\colon\cX(-m)\stackrel{\cong_{\mu_d, B}}{\longrightarrow}\cX'(-m)$.
Then by Lemma \ref{lem: As-2d-kind}(c) for any  $k\in\ZZ$ one gets a $\mu_d$-equivariant isomorphism $\cX(k)\cong_{\mu_d, B}\cX'(k)$. 
\eproof

The following proposition provides the inductive step in the proof of Theorem \ref{thm: GDF-cancellation-fixed-graph}.

\bprop\label{prop: steps} Under the assumptions of Theorem {\rm \ref{thm: GDF-cancellation-fixed-graph}}
suppose that for some $l\in\{0,\ldots,m-1\}$ there exists a $\mu_d$-equivariant isomorphism
 $\psi_l\colon \cX_l(-l) \stackrel{\cong_{\mu_d,B}}{\longrightarrow}\cX'_l(-l)$  such that
\begin{itemize}\item[(i$_l$)] 
 the induced correspondence between the special fiber components of $\pi_l$ and $\pi'_l$ is the restriction of the  isomorphism $\cD(\hat\pi)\stackrel{\cong}{\longrightarrow}\cD(\hat\pi')$;
 \item[(ii$_l$)]
 $\psi_l^*(v_l')\equiv v_l\, {\rm mod}\, z^s$
where $s>0$ and $v_l$ {\rm (}$v_l'$, respectively{\rm)} is an affine coordinate in the $\A^1$-factor of the cylinder $\cX_l(-l)$ {\rm (}$\cX'_l(-l)$, respectively{\rm)}.
\end{itemize}
Then there exists a $\mu_d$-equivariant isomorphism $\psi_{l+1}\colon\cX_{l+1} (-l-1)\stackrel{\cong_{\mu_d,B}}{\longrightarrow} \cX'_{l+1} (-l-1)$  
such that 
\begin{itemize}\item[(i$_{l+1}$)] 
the induced correspondence between the special fiber components of $\pi_{l+1}$ and $\pi'_{l+1}$ is the restriction of the  isomorphism $\cD(\hat\pi)\stackrel{\cong}{\longrightarrow}\cD(\hat\pi')$;
\item[(ii$_{l+1}$)] $\psi_{l+1}^*(v'_{l+1})\equiv v_{l+1}\, {\rm mod} (z^{s-1})$.
\end{itemize} 
\eprop

\bproof 
The morphism $\rho_{l+1}\colon X_{l+1} \to X_l$ in (\ref{eq: seq-aff-modif-final}) is a $\mu_d$-equivariant affine
modification along the reduced principal divisor $\mathbb{V}(z)=z^*(0)$ on $X_l$ with a reduced center $I$ where $\V(I)$ is the union of a finite set $\Sigma\subset X_l$ and the components of $\mathbb{V}(z)$ disjoint from $\Sigma$, cf.\ Remark \ref{rem: non-complete-intersection}. Notice that $\Sigma$ is contained in the union of  the top level components $F$ of $\mathbb{V}(z)$. Let $\fF_\Sigma$ be the set of the top-level components $\cF=F\times\A^1$ which meet
$\Sigma\times\{0\}$.  By Lemma \ref{lem: as-trick},
$\rho_{l+1}$ induces a $\mu_d$-equivariant Asanuma modification of the first kind $\tilde\rho_{l+1}\colon \cX_{l+1}(-l-1) \to \cX_l(-l)$ with the principal divisor $\mathbb{V}(z) \times \A^1$ and
center $\V(I)\times \{0\} \subset
\cX_l(-l)$ consisting of a $\mu_d$-invariant finite set $\Sigma \times \{0\}$ and a $\mu_d$-invariant union $C$ of curves isomorphic to $\A^1$ such that $C_\cF=C\cap \cF$ is given by equation $v_l=0$ in
each component $\cF=F\times \A^1\notin\fF_\Sigma$. Thus, $C\subset\{z=v_l=0\}$ in $\cX_l(-l)$. For any $\cF \in\fF_\Sigma$ we let 
\be\la{eq: centers} \Sigma_\cF =\cF \cap  (\Sigma \times \{0\}) = \{ x_1, \ldots , x_{M(\cF)} \}\,.
\ee
There is a similar collection of objects related with $X'$ instead of $X$. 
In particular,
one has a modification $\tilde\rho_{l+1}' : \cX_{l+1}'(-l-1) \to \cX_l'(-l)$ with the divisor $\mathbb{V}(z) \times \A^1$ and the
center $\V(I')\times \{0\}$ consisting of a $\mu_d$-invariant finite set $\Sigma' \times \{0\}$ and a $\mu_d$-invariant union $C'$ of curves $C'_{\cF'}\cong \A^1$. 

By virtue of (i$_l$) the $\mu_d$-equivariant isomorphism $\cD(\pi)\stackrel{\cong}{\longrightarrow}\cD(\pi')$  of graph divisors yields a one-to-one correspondence $\cF\rightsquigarrow \cF'$ between the components in $\fF_\Sigma$ and in $\fF'_{\Sigma'}$ so that (see \eqref{eq: centers}) $$M(\cF)=\card\Sigma_\cF=\card\Sigma_{\cF'}=M(\cF')\qquad\forall \cF\in\fF_\Sigma\,.$$

By virtue of (ii$_l$), $\psi_l$ sends the pair $(\cX_l(-l), \mathbb{V}(z) \times \A^1)$ to the pair $(\cX_l'(-l), \mathbb{V}(z) \times \A^1)$ and $C$ to $C'$, but not in general $\Sigma \times \{0\}$ to $\Sigma' \times \{0\}$. 
To get a bijection between the centers $\Sigma$ and $\Sigma'$ of modifications we will replace $\psi_l$ by a composition $\phi_l\circ\psi_l$ with a suitable $\mu_d$-equivariant automorphism $\varphi_l\in\SAut_B\cX'_l(-l)$.

Let $(z,u,v)=(z,u_l,v_l)$ be $\mu_d$-quasi-invariant natural coordinates  in the standard affine chart $U_F\times\A^1$ about $\cF$ where $ \cF\in\fF_\Sigma$, see Definition \ref{def: local coordinates}.  For a point $x_\nu$ in (\ref{eq: centers}) one has $x_\nu=(0,u(x_\nu), 0)$.
Similarly, for $\cF'=F'\times \A^1=\psi_l(\cF)$ consider  the standard affine chart $U_{F'}\times\A^1$  about $\cF'$ with  natural coordinates $(z,u',v')=(z,u_l',v_l')$. Let $$\cF' \cap (\V(I')\times \{0\}) =\cF' \cap  (\Sigma' \times \{0\}) = \{ x'_1, \ldots , x'_{M(\cF)} \}\,$$ where $x'_\nu=(0,u'(x'_\nu), 0)$. 

Let $\mu_e$ with $e=e(\cF)>1$ be the stabilizer of $\cF\in\fF_\Sigma$ in $\mu_d$. Then $M(\cF)\equiv 1\mod\,e$ if $\bar 0\in\Sigma_{\cF}$ and $M(\cF)\equiv 0\mod\, e$ otherwise. Since $\psi_l\colon\cF\to\cF'$ is $\mu_d$-equivariant one has $e(\cF')=e(\cF)$. Since also $M(\cF')=M(\cF)$ it follows that $\bar 0\in\Sigma_{\cF}$ if and only if $\bar 0\in\Sigma_{\cF'}$.

 By  (ii$_l$) one obtains  $$\psi_{l}(x_\nu)=:x_\nu''=(0,u'(x_\nu''),0)\in \cF',\quad\nu=1,\ldots,M(\cF)\,.$$ 
Suppose that $e(\cF)>1$ and $x_\nu=\bar 0\in\Sigma_{\cF}$. Since $\psi_l\colon\cF\to\cF'$ is $\mu_d$-equivariant is sends the orbits to the orbits. It follows that $x_\nu''=\psi_{l}(x_\nu)=\bar 0\in \Sigma_{\cF'}$. Up to renumbering one may assume in this case that $x_\nu'=x_\nu''=\bar 0$. 
 
\smallskip

\noindent {\bf Claim 1.} \emph{There exists a $\mu_d$-equivariant automorphism
$\varphi_l\in\SAut_{\mu_d,B}\cX'_l(-l)$ as in Definition {\rm \ref{bpd2}} with prescribed $\mu_d$-equivariant  $s$-jets in the points $x''_\nu$ chosen so that 
 \begin{itemize} 
 \item[(j)] $\varphi_l(\cF')=\cF'$ for every component $\cF'=F'\times\A^1$ of the divisor $z^*(0)$ on $\cX_l'(-l)$;
 \item[(jj)]   $\varphi_l^*(v')\equiv v'\mod\, z^s$ near $\cF'$  $\forall \cF' \notin\fF_{\Sigma'}$;
  \item[(jjj)] up to  reordering, $\varphi_l(x''_\nu) = x'_\nu$, $\nu=1,\ldots,M(\cF')$  $\forall \cF' \in\fF_{\Sigma'}$;
 \item[(jv)] the $s$-jets of $\varphi_l^*(v')$ and $v'$ at $x_\nu''$ coincide for any $\nu=1,\ldots,M(\cF')$, $\forall \cF' \in\fF_{\Sigma'}$.
\end{itemize}}

\smallskip

\noindent\emph{Proof of Claim $1$}.  Condition (j) holds for any $\varphi\in \SAut_B(\cX'_l)$, cf.\ Proposition \ref{prop: aff-chart}. Due to Theorem \ref{thm: RF} the surface $X'$ verifies the $\mu_d$-equivariant condition ${\rm RF}(l,-l,s)$. Therefore, one can choose $\varphi_l\in\SAut_{\mu_d,B}\cX'_l(-l)$ verifying conditions ($\alpha_1$),  ($\alpha_2$), and ($\beta$) of Definitions \ref{bpd1} and  \ref{bpd2} with a suitable data. This yields (jj), and as well (jjj) and (jv) in the case where either $e(\cF')=1$ or $\bar 0\notin  \Sigma_{\cF'}$.  

In the remaining case one has $e(\cF')>1$ and $\bar 0=x_\nu'\in  \Sigma_{\cF'}$. By the observations preceding the claim one may assume that $x_\nu'=x_\nu''=\bar 0$ and the $s$-jet of $\varphi_l$ at $\bar 0$ is the $s$-jet of the identity. This yields  (jjj) and (jv) also in the remaining case. Now the claim follows. 

\smallskip

Due to (j)--(jjj) the composition $\tilde\psi_l:=\varphi_l\circ\psi_l$ sends the center and the divisor of $\tilde\rho_{l+1}$ to  the center  and the divisor of $\tilde\rho'_{l+1}$.
 By Lemma \ref{lem: preserving-isomorphisms}, $\tilde\psi_l$ lifts to  a $\mu_d$-equivariant isomorphism $\psi_{l+1}\colon \cX_{l+1} (-l-1) \stackrel{\cong_{\mu_d,B}}{\longrightarrow}  \cX'_{l+1} (-l-1)$. The proof ends due to the following 

\smallskip

\noindent {\bf Claim 2.} \emph{$\psi_{l+1}$ satisfies conditions {\rm (i$_{l+1}$)} and  {\rm (ii$_{l+1}$)}.}

\smallskip

\noindent\emph{Proof of Claim $2$}.  Due to conditions (i$_l$) for $\psi_l$ and (j) for $\varphi_l$ the isomorphism $\cD(\hat\pi)_{\le l}\stackrel{\cong}{\longrightarrow}\cD(\hat\pi')_{\le l}$  induced by $\psi_{l+1}$ coincides with the restriction of the given isomorphism $\cD(\hat\pi)\stackrel{\cong}{\longrightarrow}\cD(\hat\pi')$. The same holds for  the induced isomorphism $\cD(\hat\pi)_{\le l+1}\stackrel{\cong}{\longrightarrow}\cD(\hat\pi')_{\le l+1}$ after a suitable renumbering of the points $x_1',\ldots,x'_{M(\cF')}$ on each component $\cF'\in\fF'$. This gives  (i$_{l+1}$).

For any special fiber component $\cF$ in $\cX_{l+1}(-l-1)$ of level $\le l$ condition (ii$_{l+1}$)  holds due to (ii$_l$), (jj), and the equalities $v_{l+1}=v_l/z$, $v'_{l+1}=v'_l/z$. It holds as well for $\cF$ of the top level $l+1$ due to (ii$_l$), (jv), and the same equalities. 
\eproof

\subsection{Rigidity of cylinders under deformations of  $\A^1$-fibered surfaces}\label{ss: def-A1-fibered-surfaces}

Using Theorem \ref{thm: GDF-cancellation-fixed-graph} we obtain our second main result.

\bthm\label{thm: cancellation-fixed-graph} 
Let $\pi\colon Y\to C$ and $\pi'\colon Y'\to C$ be two 
$\A^1$-fibered normal affine surfaces over a smooth 
affine curve $C$. Let $\hat Y\to\hat C$ be an SNC completion of the minimal resolution of singularities of $Y$,
and let $\hat D_{\rm ext}$ be the extended divisor of this completion. Let a pair
(${\hat Y}', \hat D_{\rm ext}')$ plays the same role for $Y'$.
Suppose that \begin{itemize}\item[$\bullet$] the degenerate fibers 
of $\pi$ and $\pi'$ are situated over the same points $p_1,\ldots,p_t\in C$; \item  for  $i=1,\ldots,t$
 the corresponding fiber trees 
$\Gamma_{p_i}(\pi)$ and $\Gamma_{p_i}(\pi')$ 
are isomorphic\footnote{This condition ensures that the corresponding fiber 
components of $\pi$ and 
$\pi'$ have the same multiplicities. Indeed, under our assumptions the isomorphism respects 
the feathers along with 
their bridges.}; \item making similar contractions in $\hat D_{\rm ext}$ and $\hat D_{\rm ext}'$ one can
reduce both $\hat Y$ and ${\hat Y}'$ to the product $\hat C\times \PP^1$
with the same distinguished "section at infinity" $\hat C \times \{\infty\}$.\end{itemize} 
Then the cylinders 
$Y\times\A^1$ and $Y'\times\A^1$
are isomorphic over $C$.
\ethm

\bproof Applying a suitable cyclic Galois base change 
$B\to C$ of order $d$ ramified over the  points $p_1,\ldots,p_t\in C$ one can replace the $\A^1$-fibered 
surfaces $\pi\colon Y\to C$ and $\pi'\colon Y'\to C$ by two marked GDF $\mu_d$-surfaces $X\to B$ and $X'\to B$, 
respectively, with the same $\mu_d$-quasi-invariant marking $z\in\cO_B(B)\setminus\{0\}$, see Lemma \ref{lem: br-covering} and Remark \ref{rem: GDF-cyl}. 
Due to our assumptions the extended graphs and the fiber trees of the special fibers of suitable $\mu_d$-equivariant  pseudominimal completions $\bar X$ and $\bar X'$ of the GDF surfaces $X$ and $X'$ are 
isomorphic under a $\mu_d$-equivariant isomorphism.\footnote{Such completions are not unique. However,  the deformation parameters are irrelevant for the combinatorial invariants such as the extended graph and the fiber trees.  } Moreover, these surfaces admit trivializing $\mu_d$-equivariant completions $\hat X$ and $\hat X'$, respectively, verifying the assumptions of Theorem \ref{thm: GDF-cancellation-fixed-graph}.
Due to this theorem  there is a $\mu_d$-equivariant isomorphism 
$\mathcal{X}(0)\cong_{\mu_d,B}\mathcal{X}'(0)$.
Passing to the quotients $\mathcal{X}(0)/\mu_d=
Y\times\A^1$ and $\mathcal{X}'(0)/\mu_d=Y'\times\A^1$ 
yields a desired $C$-isomorphism
$Y\times\A^1\cong_C Y'\times\A^1$. 
\eproof 

\bcor\la{cor: moduli-cyl} Let $C$ be a smooth affine curve with marked points $p_1,\ldots,p_t\in C$. Consider the collection $\fH=\fH(C,p_1,\ldots,p_t)$ of all the $\A^1$-fibered normal affine surfaces $\pi\colon X\to C$ such that $\pi$ restricted over $C\setminus\{p_1,\ldots,p_t\}$ is the projection of a trivial line bundle. 
Then the set of $C$-isomorphism classes of cylinders $\cX=X\times\A^1$, where $X$ runs over $\fH$, is at most countable. 
\ecor

\bproof Indeed, the set of the isomorphism classes of finite trees is countable. The same is true for the set of all ordered $t$-tuples of such trees as in Theorem \ref{thm: cancellation-fixed-graph}. Now the assertion follows from this theorem. \eproof

\brem\la{rem: non-A1-base} Let $\pi\colon X\to C$ and $\pi'\colon X'\to C'$ be $\A^1$-fibered normal affine surfaces.  If $C\not\cong\A^1$ then any isomorphism  of cylinders $\varphi\colon\cX\stackrel{\cong}{\longrightarrow}\cX'$ fits in a commutative diagram 
$$\bdi
\cX &\rTo<{\varphi}>{\cong} & \cX' \\
\dTo<{} &&\dTo>{} \\
C &\rTo>{\psi}<{\cong} & C'\, 
\edi$$ where $\psi$ is an isomorphism
(cf.\ also Lemma \ref{lem: mult-fiber}).
For instance, the isomorphism type of the cylinder $\cX$ over the surface $X=(\A^1\setminus \{k\,\,\mbox{points}\})\times\A^1$ depends essentially on the isomorphism type of the factor $\A^1\setminus \{k\,\, \mbox{points}\}$ (see \cite[9.10.1]{Fu1}). 
 \erem

\subsection{Rigidity of line bundles over affine surfaces} In unpublished notes \cite{BML5} kindly provided to us by the authors the study of cylinders over affine surfaces is extended to the total spaces of line bundles over affine surfaces. 
Theorem \ref{thm: GDF-cancellation-fixed-graph-T} below is an analog of Theorem \ref{thm: GDF-cancellation-fixed-graph}  in this wider context. We do not use this extension in the sequel, so, we just indicate the necessary modifications in the proof of Theorem \ref{thm: GDF-cancellation-fixed-graph}. 

\bnota\label{gan1}
Let  $X$ be
an affine algebraic variety. For a Cartier divisor $T\in\CDiv X$ 
we let  $\pi^T\colon \cX^T\to X$ be the associated line bundle on $X$ with a zero section $Z^T\subset \cX^T$.
\enota

\bdefi\la{def: as-mod-1-kind-T} Let $D\in\CDiv X$ be a reduced effective  Cartier divisor on $X$. By an {\em Asanuma modification of the second kind} of $\cX^T$ we mean an  affine modification $\sigma^D\colon \cX^{T,D}\to \cX^T$ along the principal divisor $\cD^T=(\pi^T)^*(D)$ on $\cX^T$
with the center $\cD^T\cdot Z^T$.
\edefi

We have the following analogue of Lemma \ref{lem: as-trick} (corresponding to the case $T\sim 0$).

\blem\label{lem: graph-change0}
In Notation {\rm \ref{gan1}}, $\pi^{T,D}=
\pi^T\circ\sigma^D \colon \cX^{T,D}\to X$ admits a structure of a line bundle such that $\cX^{T,D}\cong_X \cX^{T-D}$.
\elem

\bproof Choose an open covering $X= \bigcup_i U_i$ such that 
\begin{itemize}\item $D\cap U_i=f_i^*(0)$ and $T\cap U_i=\div h_i$ where $f_i \in \cO_{U_i}(U_i)$ and $h_i\in \Frac \cO_{U_i} (U_i)$.  \end{itemize} Then \begin{itemize} \item $\alpha_{i,j}=f_j/f_i, \,\beta_{i,j}=h_j/h_i\in \cO^\times_{U_{i,j}}(U_{i,j})$ where $U_{i,j}=U_i \cap U_j$, are \v{C}ech $1$-cocycles on $X$ associated with the line bundles $\cX^{D}\to X$ and $\cX^{T}\to X$, respectively.
\end{itemize}
Letting $V_i =(\pi^T)^{-1}(U_i)$ there are local trivializations $V_i \cong_{U_i}U_i\times \A^1$ of $\pi^T\colon \cX^T\to X$ where $\A^1=\Spec \mathbb{k}[v_i]$ with  $v_j=\beta_{i,j}v_i$ over $U_{i,j}$. 
Consider the restriction $V_i' \to V_i$ of the morphism $\sigma\colon \cX^{T,D} \to X^T$ over $V_i$ 
induced by the natural inclusion
$$\cO_{V_i}(V_i)\hookrightarrow \cO_{V'_i}(V'_i)=\cO_{V_i}(V_i)[v_i/f_i]\,.$$ 
One has $V_i'\cong_{U_i} U_i'\times \A^1$ where $\A^1=\Spec \mathbb{k}[v'_i]$ with $v'_i=v_i/f_i$. This defines local trivializations of the projection $\pi^{T,D}\colon \cX^{T,D}\to X$, hence a structure of a line bundle on $\cX^{T,D}$ over $X$.
Note that $v'_j =(\alpha_{i,j}^{-1}\beta_{i,j}) v'_i$ where $\{ \alpha_{i,j}^{-1}\beta_{i,j} \}$ is a \v{C}ech $1$-cocycle on $X$
associated with the line bundle $\pi^{T-D}\colon\cX^{T-D}\to X$.
\eproof

\bnota\la{not: character} Let $X$ be an affine variety acted upon by a  finite group G, and let $T,D\in\Div (X)$ be $G$-invariant divisors where $D$ is reduced. Then the line bundle $\cX^T\to X$ admits a $G$-linearization, that is, a structure of a $G$-equivariant line bundle. 
This structure is not unique, in general. It is defined modulo the multiplication by a character, see, e.g., \cite{MFK}. Choosing a $G$-linearization, say, $\cX^T(1)\to X$ with the corresponding   linear equivariant $G$-action $\phi\colon (g,v)\mapsto g.v$ on $\cX^T$, for
a  character $\chi\in G^\vee$ consider a new such action $\phi^\chi\colon (g,v)\mapsto \chi(g)\cdot g.v$. This yields a new $G$-linearization denoted by $\cX^T(\chi)\to X$. 

 In the case of a cyclic group $G=\mu_d$,
fixing a primitive
character $\chi$ of $\mu_d$ we write $\cX^T(k)\to X$ for the $\mu_d$-linearization on $\cX^T\to X$ associated with the character $\chi^k$. With this notation, $\cX^T(0)\to X$ corresponds to the given $G$-linearization. Clearly, the sequence $(\cX^T(k))_{k\in\Z}$ is periodic with period $d$. For any $G$-invariant divisors $T_1,T_2\in\Div (B)$ and any characters $\chi,\lambda\in G^\vee$ there is a $G$-equivariant isomorphism $\cX^{T_1}(\chi)\otimes \cX^{T_2}(\lambda)\cong_X \cX^{T_1+T_2}(\chi\lambda)$. 
\enota

In the sequel we need the following simple lemma. 

\blem\la{lem: lb-over-curve} Let $B$ be a smooth affine curve acted upon by a finite group $G$, and let $\xi: L\to B$ be a line bundle over $B$ which admits a $G$-linearization. 
Then for any $b_1,\ldots,b_n\in B$ there are a $G$-invariant open set $U$ 
containing these points and a $G$-equivariant trivialization of $\xi|_U$. 
\elem

\bproof It suffices to find a nonzero $G$-stable (that is, $G$-equivariant) rational section $s\colon B\to L$ of $\xi$ which has neither pole nor zero in $b_1,\ldots,b_n$ and to set $U=B\setminus \supp(\div s)$. Given any nonzero $G$-stable rational section $s_0\colon B\to L$ of $\xi$ one can find a $G$-invariant rational function $f\neq 0$ on $B$ such that $\div f$ restricts to $\div s$ on $b_1,\ldots,b_n$. Then $s=s_0/f$ is a desired $G$-stable section  of $\xi$.\eproof

\bnota\label{gan-1} Let $\pi\colon X\to B$ be a marked GDF $\mu_d$-surface over a smooth affine curve $B$ with a $\mu_d$-quasi-invariant marking $z\in\cO_B(B)\setminus\{0\}$ of weight 1. Then the principal divisor $D=z^*(0)\in\Div (B)$ is $\mu_d$-invariant. Given a $\mu_d$-invariant divisor $T\in\Div (B)$  consider the line bundle $\cX^{T^*}\to X$ where $T^*=\pi^*(T)\in\Div (X)$. By abuse of notation we let  $\cX^{T}=\cX^{T^*}$. If $\xi\colon L\to B$ is the line bundle associated with $T$ then $\cX^{T}\to X$ is induced by $\xi$ via the morphism $\pi\colon X\to B$. Hence both $\xi$ and $\cX^T\to X$ admit $\mu_d$-linearizations such that the natural morphism $\cX^T\to L$ is $\mu_d$-equivariant. Choosing such a $\mu_d$-linearization of $\xi$ 
and the one of $\cX^T\to X$ we observe that $L(k)$ naturally corresponds to $\cX^T(k)$. \enota

There is the following equivariant version of Lemma \ref{lem: graph-change0}.

\blem\label{gal1} Let things be as in {\rm \ref{gan-1}}. Then for any $k\in\Z$ there exists a $\mu_d$-action on 
$\cX^{T,D}$ and a $\mu_d$-equivariant 
isomorphism of line bundles $\cX^{T,D}\cong_{\mu_d,B} \cX^{T-D}(k-1)$ such that the induced morphism $\sigma^D\colon\cX^{T-D}(k-1)\to\cX^T(k)$ is $\mu_d$-equivariant.
\elem

\bproof The $\mu_d$-action on $\cX^T(k)$ stabilizes the divisor $\cD^T=(\pi^T)^*(D)\in\Div (\cX^T)$ and the
center $\cD^T\cdot Z^T$ of the affine modification $\sigma^D\colon\cX^{T,D}\to\cX^T(k)$. By \cite[Cor.\ 2.2]{KZ} (see Lemma \ref{lem: preserving-isomorphisms}) it lifts to a $\mu_d$-action on $\cX^{T,D}$ making $\sigma^D$ equivariant. 

Choose a trivializing open set $U\subset B$ as in Lemma \ref{lem: lb-over-curve}, and let $V=\pi_X^{-1}(U)\subset X$. Then $\cX^T(k)\to X$ admits over $V$ a $\mu_d$-equivariant trivialization $\cX^T(k)|_V\cong_{\mu_d,V} (V\times\A^1)(k+m)$ where $\A^1=\Spec \mathbb{k}[v]$ and $m$ is the weight of an equivariant trivialization of $\cX^T(0)|_V$.  Recall that $D=\div z$ where $z$ has weight 1 and there is a natural isomorphism $\cX^{T,D}|_V\cong_V V\times\A^1$ where $\A^1=\Spec \mathbb{k}[v/z]$ compatible with an isomorphism $\cX^{T,D}\cong_{B} \cX^{T-D}$ of Lemma \ref{lem: graph-change0} (see the proof of this lemma). This gives an  equivariant trivialization $\cX^{T-D}|_V\cong_{\mu_d,V} (V\times\A^1)(k+m-1)$ and shows that the induced $\mu_d$-action
on $\cX^{T-D}$ has weight $k-1$. Now the assertions follow. \eproof

The following result is an analog of Theorem \ref{thm: GDF-cancellation-fixed-graph} in our more general setting. 

\bthm\label{thm: GDF-cancellation-fixed-graph-T} 
Let $\pi_X\colon X\to B$ and $\pi_Y\colon Y\to B$ be two 
marked GDF $\mu_d$-surfaces over $B$ with the same $\mu_d$-quasi-invariant marking  $z\in\cO_B(B)\setminus\{0\}$ of weight $1$.  
Assume that for some trivializing $\mu_d$-equivariant completions
$(\hat X, \hat D_X)$ and $(\hat Y, \hat D_Y)$
the graph divisors $\cD(\hat\pi_X)$ and $\cD(\hat\pi_Y)$ 
are $\mu_d$-equivariantly isomorphic. Let $T\in\Div (B)$ be a $\mu_d$-invariant divisor. 
Then for any $k\in\ZZ$ there is a $\mu_d$-equivariant isomorphism
$\cX^{T}(k)\cong_{\mu_d,B}\cY^{T}(k)$.
\ethm

In the proof we use an analog of the Asanuma modification of the first kind for line bundles over surfaces (see Definition \ref{def: as-mod-T} below). Let us introduce the following notation.

\bnota\label{gan-2} Let $z^{-1}(0)=\{b_1,\ldots,b_n\}\subset B$. Consider a trivializing sequence (\ref{eq: seq-aff-modif-final}) of fibered modifications $\rho_{l+1}\colon X_{l+1}\to X_l$, $l=0,\ldots,m$. Let $T$ be a $\mu_d$-invariant divisor on $B$, and let $f$ be a rational $\mu_d$-quasi-invariant function on $B$ such that $(\div f)(b_i)=- T(b_i)$, $ i=1,\ldots,n$. Then $T$ and $T+\div f$ represent the same class in $\Pic B$.
Replacing $T$ by $T+\div f$ we may assume that $b_i\notin\supp T$ $\forall i=1,\ldots,n$. For every $l=0,\ldots,m$ we let $T_l=\pi_l^*(T)\in\Div (X_l)$. Since $T_{l+1}=\rho_{l+1}^*(T_l)$
the modification $\rho_{l+1}\colon X_{l+1}\to X_l$ induces an affine 
modification $\rho^T_{l+1}\colon\cX_{l+1}^{T_{l+1}}\to \cX_l^{T_l}$ which fits in the commutative
diagram  

\smallskip

\be\la{diag: cub}\begin{diagram}[notextflow]
  \cX_{l+1}^{T_{l+1}-D_{l+1}}  &    &\rTo^{\rho^{T-D}_{l+1}} &      &    \cX_{l}^{T_{l}-D_{l}}   \\
      & \rdTo_{\pi^{T-D}_{l+1}} &      &      & \vLine^{}& \rdTo^{\pi^{T-D}_{l}}  \\
\dTo^{\sigma_{l+1}^D} &    &   X_{l+1}   & \rTo^{\rho_{l+1}}  & \HonV   &    &  X_{l}  \\
      &    & \dTo^{\rm id}  &      & \dTo_{\sigma_{l}^D}   \\
   \cX_{l+1}^{T_{l+1}}  & \hLine & \VonH   & \rTo_{\rho^T_{l+1}} &  \cX_{l}^{T_{l}}   &    & \dTo_{\rm id} \\
      & \rdTo_{\pi^{T}_{l+1}} &      &      &      & \rdTo^{\pi^{T}_{l}}  \\
      &    &   X_{l+1}   &      & \rTo_{\rho_{l+1}}  &    &  X_{l}  \\
\end{diagram}\ee
\medskip

\noindent For a fiber component $F_i\subset D_l$ we let
$C_i$  be the intersection  of $F_i$ with the center of the modification $\rho_{l+1}\colon X_{l+1}\to X_l$. 
Then $\rho^T_{l+1}\colon \cX_{l+1}^{T_{l+1}} \to   \cX_l^{T_l}$ is an affine modification along the divisor $\cD_l=\bigcup_i \cF_i$ with the center $\cC_l=\bigcup_i\cC_i$ where
$\cF_i=(\pi^{T_l})^{-1}(F_i)\cong F_i \times \A^1\cong\A^2$ and  
$\cC_i\cong C_i \times \A^1\subset F_i \times \A^1$. There is an alternative:  either
\begin{itemize}\item[(i)] $C_i$ is finite, or
 \item[(ii)] $C_i=F_i$.\end{itemize} In case (i), $F_i$ is a top level  component. In case (ii), $X_{l+1} \to   X_l$
($\cX_{l+1}^{T_{l+1}} \to   \cX_l^{T_l}$, respectively) is an isomorphism near  $F_i$ (near $\cF_i$, respectively).

\enota

\bdefi\la{def: as-mod-T} We call an \emph{Asanuma modification of the first kind} the birational morphism 
$$\kappa_{l+1} \colon \cX_{l+1}^{T_{l+1}-D_{l+1}}\to \cX_l^{T_l}\,$$
where $\kappa_{l+1}=\rho_{l+1}^{T}\circ\sigma^D_{l+1}=\sigma_{l}^D\circ\rho_{l+1}^{T-D}$ is the diagonal composition
of morphisms in the back square of (\ref{diag: cub}). 
Then $\kappa_{l+1}$ is an affine modification along the divisor $\cD_l=\bigcup_i \cF_i$ on $\cX_l^{T_l}$
with the center $\bigcup_i (C_i \times \{0\})$.
In case (i), $C_i \times \{0\}\subset \cC_i\cong \A^2$ is zero-dimensional, while in case (ii)
this is just the coordinate axis $v=0$ in $\cC_i\cong \A^2$. Due to Lemma \ref{gal1} and by analogy with  (\ref{eq: sec-Asanuma-modif}) 
one has the following sequence of equivariant Asanuma modifications of the first kind:
\be\label{eq: sec-Asanuma-modif-1} 
\mathcal{X}^{T_m-mD_m}_{m}(-m)\stackrel{\tilde\rho_m}{\longrightarrow}
\ldots\stackrel{}{\longrightarrow} \mathcal{X}^{T_2-2D_2}_{2}(-2) \stackrel{}{\longrightarrow}  
\mathcal{X}^{T_1-D_1}_{1}(-1)\stackrel{\tilde\rho_1}{\longrightarrow} \mathcal{X}^{T_0}_{0}(0)\,.\ee 
\end{defi}

\medskip

\noindent \emph{Proof of Theorem} \ref{thm: GDF-cancellation-fixed-graph-T}. For the given GDF surfaces $\pi_X\colon X\to B$ and $\pi_Y\colon Y\to B$, consider the corresponding sequences (\ref{eq: sec-Asanuma-modif-1}) starting with the same line bundle $\cX^{T_0}_0(0)=(B\times\A^1)^{T_0}(0)=\cY^{T_0}_0(0)$. One may suppose that $\mu_d$ acts trivially on the factor $\A^1$. 
Using Proposition \ref{gal2} below with 
$s>m$
it follows by induction that for $l=0,\ldots,m$ there is a (non-linear, in general) $\mu_d$-equivariant isomorphism
$$\mathcal{X}^{T_l-lD_l}_l(-l)\cong_{\mu_d,B}\mathcal{Y}^{T_l-lD_l}_l(-l)\,$$ which sends the zero section $Z(\mathcal{X}^{T_l-lD_l}_l(-l))$ of the first line bundle to such a section of the second one. Replacing $T$ by $T+mD$ one obtains  for $l=m$,
$$\cX^{T}(-m)=\cX^{T_m}_{m}(-m)\stackrel{\phi}{\longrightarrow}\cY^{T_m}_{m}(-m)=\cY^{T}(-m)\,,$$ where $\phi$ is a $(\mu_d,B)$-isomorphism  respecting the zero sections $Z(\cX^{T}(-m))$ and $Z(\cY^{T}(-m))$ and the divisors $\cD^T(\cX^{T})$ and $\cD^T(\cY^{T})$. Hence $\phi$ respects also the centers $\cD^T(\cX^{T})\cdot Z(\cX^{T}(-m))$ and $\cD^T(\cY^{T})\cdot Z(\cY^{T}(-m))$ of the Asanuma modifications of the second kind. Applying these modifications on both sides, by Lemma \ref{gal1} we decrease by 1 the weights of the $\mu_d$-actions. Due to Lemma \ref{lem: preserving-isomorphisms}, $\phi$ admits a lift to a $(\mu_d,B)$-isomorphism $\tilde\phi$ fitting in the commutative diagram  
\be\la{diag: square}\begin{diagram}[notextflow]
  \cX^{T-D}(-m-1)  &  \rTo^{\tilde\phi}_{\cong_{\mu_d,B}} &    \cY^{T-D}  (-m-1) \\
\dTo^{\sigma^D} &                           &     \dTo^{\sigma^D} \\  \cX^T(-m)            &  \rTo_{\phi}^{\cong_{\mu_d,B}}        &  \cY^T(-m)  \\
\end{diagram}\ee
and respecting the zero sections. Choose $n\ge 1$ such that $-(m+n)\equiv k\mod\, d$. For $s\gg 1$ after $n$ iterations one arrives at an isomorphism $\cX^{T-mD}(k)\cong_{\mu_d, B}\cY^{T-nD}(k)$. This holds for an arbitrary $\mu_d$-stable divisor $T\in\Div (B)$. Replacing the initial $T$ by $T+nD$ one gets an isomorphism
$\cX^{T}(k)\cong_{\mu_d, B}\cY^{T}(k)$, as required.\qed 

\medskip

In the proof we have used
 the following analog of Proposition \ref{prop: steps}. By abuse of notation we let $v_i$ and $\tilde v_i$ be the local fiber coordinates of the line bundles $\cX_l^T\to X$ and $\cY_l^T\to Y$, respectively.

\bprop\label{gal2} Under the assumptions of Theorem {\rm\ref{thm: GDF-cancellation-fixed-graph-T}} let $$\psi_l \colon \cX^T_l (-l) \stackrel{\cong_{\mu_d, B}}{\longrightarrow} {\cY^T_l} (-l)$$ be a $\mu_d$-equivariant isomorphism
such that
\begin{itemize}\item[{\rm (i$_l$)}] $\psi_l^*(\tilde v_i)\equiv v_i\, \mod\, z^s$ $\forall i$.\end{itemize}
Then there exists a $\mu_d$-equivariant isomorphism $$\psi_{l+1} \colon \cX^{T-D}_{l+1} (-l-1) \stackrel{\cong_{\mu_d,B}}{\longrightarrow} 
{\cY^{T-D}_{l+1}} (-l-1)$$
such that 
\begin{itemize}
\item[{\rm (i$_{l+1}$)}] $\psi_{l+1}^*(\tilde v_i)\equiv v_i\, \mod\, z^{s-1}$ $\forall i$.\end{itemize}
\eprop

\noindent \emph{Hint}. The proof of  Proposition \ref{prop: steps} goes verbatim modulo the existence of an automorphism $\phi$ which is guaranteed by Theorem \ref{thm: RF}.  Thus, it suffices to prove the following  analog of Theorem \ref{thm: RF}. 

\bthm\la{thm: RF-T} Let a GDF $\mu_d$-surface $\pi_X\colon X\to B$, $z\in\cO_B(B)$, and $T\in\Div (B)$ be as in Theorem {\rm\ref{thm: GDF-cancellation-fixed-graph-T}}. Then $\cX^T$ satisfies  an analog of the $\mu_d$-equivariant condition RF$(l,-l,s)$.
\ethm

\bproof It suffices to reproduce mutatis  the proof of Theorem \ref{thm: RF} (see Section \ref{ss: rtr}). The  modifications are as follows. 

The coordinate $v$ used when working with cylinders might do not exist on the total space of the line bundle $\pi^T\colon\cX^T\to X$. Hence one cannot consider on $\cX^T$ the locally nilpotent derivations $\tilde\sigma_{1,f}$ and $\tilde\sigma_{2,g}$ as in (\ref{eq: sigmas}). 
However, one can employ instead their analogs which coincide with these up to a given order on any special fiber component $\cF_i=(\pi^T)^{-1}(F_i)$ in $\cX^T$. 

Indeed,  let $\xi: L\to B$ be the line bundle associated with $T$, and let $U\subset B$ be a $\mu_d$-stable dense open subset  as in Lemma \ref{lem: lb-over-curve} which contains $z^{-1}(0)=\{b_1,\ldots,b_n\}$ and such that $\xi|_U$ is trivial as a $\mu_d$-line bundle. Then also the induced line bundle $\pi^T\colon\cX^T\to X$ is trivial over $V=\pi^{-1}(U)\subset X$. 
Thus, $\cX^T|_V\cong_{\mu_d,V} V\times\A^1$ where $\A^1=\Spec \mathbb{k}[v]$. Via this isomorphism, $v$ yields a rational $\mu_d$-quasi-invariant function on $\cX^T$ which we denote  by the same letter. 

Choose a regular $\mu_d$-quasi-invariant function $h\in\cO_B(B)$ such that $h-1\equiv 0 \mod\, z^s$ and $h|_{B\setminus U}= 0$. Consider  the lift $\tilde h\in\cO_{\cX^T}(\cX^T)$ of $h$. For $s\gg 1$ the product $\tilde  v=\tilde h^sv\in\cO_{\cX^T}(\cX^T)$ is a regular $\mu_d$-quasi-invariant which coincides with $v$ to order $s$ on any special fiber component $\cF_i=(\pi^T)^{-1}(F_i)$ in $\cX^T$. Letting  $\p^*_l=(\pi^T)^*(\p_l)$ and $\hat\sigma_{1,f}=
f(\tilde v^d)\p^*_l$ for $f\in \mathbb{k}[t]$ yields a $\mu_d$-invariant locally nilpotent derivation on $\cO_{\cX^T}(\cX^T)$ which coincides to order $s$ with $\tilde\sigma_{1,f}$  on any  fiber component $\cF_i$. 

Furthermore, for $s\gg 1$ the product $\tilde h^{d+1}\p/\p v$ is a $\mu_d$-invariant locally nilpotent derivation on $\cO_{\cX^T}(\cX^T)$. Letting $\hat\sigma_{2,g}=\tilde u^{ds}g(\tilde u^d) \tilde h^{d}\p/\p v$ for $g\in \mathbb{k}[t]$ where $\tilde u$ is as defined in \ref{nota: q-inv-grps} yields a $\mu_d$-invariant locally nilpotent derivation on $\cO_{\cX^T}(\cX^T)$ which coincides  to order $s$ with $\tilde\sigma_{2,g}$  on any  fiber component $\cF_i$.

Using the locally nilpotent derivations $\hat\sigma_{1,f}$ and $\hat\sigma_{2,g}$ instead of $\tilde\sigma_{1,f}$ and $\tilde\sigma_{2,g}$, respectively, the rest of the proof of Theorem  \ref{thm: RF} applies and gives the desired $\mu_d$-equivariant relative flexibility.
\eproof

\section{Basic examples of Zariski factors} \label{sec: basic-examples}
\subsection{Line bundles  over affine curves}\label{ss: Line bundles}

\bprop\label{prop: line-bundle} Let $\pi\colon X\to B$ be a line bundle over a smooth affine curve $B$. Then the surface $X$ is a Zariski factor. 
\eprop

\bproof If $B\cong\A^1$ then $\pi\colon X\to B$ is a trivial line bundle, and so, $X\cong \A^2$ is a Zariski factor by the Miyanishi-Sugie-Fujita Theorem (\cite{Fu, MS}; see also \cite[Ch.\ 3, Thm.\ 2.3.1]{Mi}). 

Suppose further that $B\not\cong\A^1$, and so, any morphism $\A^1\to B$ is constant. 
Consider a second smooth affine surface $X'$ and the cylinder $\cX'=X'\times\A^n$.  Assume that there is an isomorphism $\varphi\colon \cX'\stackrel{\cong}{\longrightarrow}\cX$. The  structure of a vector bundle of $\tilde\pi:={\rm pr}_1\circ\pi\colon\cX\to B$ is transferred by $\varphi^{-1}$ to such a structure on $\cX'$ with the projection $\tilde\pi'=\tilde\pi\circ\varphi\colon\cX'\to B$. This yields the commutative diagram
\be\label{diagr: case-02}
 \bdi
\cX'&\rTo<{\phi}>{\cong_{B}} &  \cX \\
\dTo<{\tilde\pi'} &  & \dTo>{\tilde\pi} \\
B &\rTo>{\id} & B \, 
\edi
\ee
 Since any morphism $\A^1\to B$ is constant, $\tilde\pi'$ admits a factorization 
\be\label{eq: factorization}
\tilde\pi'\colon \cX'\stackrel{{\rm pr}_1}{\longrightarrow} X'\stackrel{\pi'}{\longrightarrow} B\,.
\ee
Letting $F_b=\pi^{-1}(b)$ and $F_b'={\pi'}^{-1}(b)\subset X'$ where $b\in B$, $\phi$ restrits to an isomorphism $$\phi|_{F_b'\times\A^n}\colon F_b'\times\A^n\stackrel{\cong}{\longrightarrow} F_b\times\A^n\cong\A^{n+1}\,.$$ Since any curve is a Zariski factor (\cite[Thm.\ 6.5]{AHE}) one deduces that $F_b'\cong F_b\cong\A^1$ $\forall b\in B$.
Thus, $\pi'\colon X'\stackrel{}{\longrightarrow} B$ is an $\A^1$-fibration with all the fibers being reduced and irreducible, because the fibers of $\tilde\pi'\colon\cX'\to B$ are. Therefore, $\pi'\colon X'\stackrel{}{\longrightarrow} B$ admits a structure of a line bundle. Now the existence of an isomorphism $X'\cong X$ follows from the next lemma where we let $d=1$.
\eproof

We work in an equivariant setup that we will use below. 

\blem\label{lem: line-bundle-equiv}  Consider two GDF $\mu_d$-surfaces $\pi\colon X\to B$ and $\pi'\colon X'\to B$ with only irreducible fibers. 
Extend the $\mu_d$-actions to the cylinders $\cX=X\times\A^n$ and $\cX'=X'\times\A^n$ by the identity on the second factor. Suppose that there is a $\mu_d$-equivariant isomorphism $\Phi\colon\cX\stackrel{\cong_{\mu_d,B}}{\longrightarrow} \cX'$ over $B$.
Then there exist $\mu_d$-equivariant line bundle structures $\xi$ and $\xi'$ on $X$ and $X'$ with projections $\pi$ and $\pi'$, respectively, and a $\mu_d$-equivariant isomorphism of line bundles $\xi\cong_{\mu_d,B}\xi'$ identical on $B$. 
\elem 

\bproof The average of an arbitrary section  of $\pi$ upon the $\mu_d$-action on $X$ yields a $\mu_d$-invariant section  of $\pi$ and, respectively, a $\mu_d$-equivariant structure of a line bundle $\xi$ on $X$ with projection $\pi$. Similarly, $X'$ admits a $\mu_d$-equivariant line bundle structure $\xi'$ with projection $\pi'$. Let us show that there exists an isomorphism of line bundles $\xi\cong_{\mu_d,B}\xi'$.

The cylinder $\cX=X\times\A^n$ inherits a structure of a vector bundle of rank $n+1$ isomorphic to the Whitney sum $\xi\oplus \bf {1_n}$ with projection $\tilde\pi\colon\cX\to B$ where $\bf {1_n}$ stands for the trivial vector bundle of rank $n$ over $B$. Similarly, $\tilde\pi'\colon\cX'\to B$ represents the vector bundle $\xi'\oplus \bf {1_n}$.
The $\mu_d$-equivariant isomorphism of the total spaces 
$$\Phi\colon \cX={\rm tot} (\xi\oplus \bf {1_n})\stackrel{\cong_{\mu_d,B}}{\longrightarrow} \cX'={\rm tot} (\xi'\oplus \bf {1_n})$$ 
sends the zero section $Z$ of $\xi\oplus \bf {1_n}$ to a section, say, $Z''$ of  $\xi\oplus \bf {1_n}$. Both $Z$ and $Z''$ are $\mu_d$-invariant. It is easily seen that the translation $t_{-Z''}$ on $-Z''$ in $\xi'\oplus \bf {1_n}$ is $\mu_d$-equivariant. Hence the composition  $\Psi=t_{-Z''}\circ\Phi$ is as well, and it sends $Z$ to the zero section $Z'$ of 
$\xi'\oplus \bf {1_n}$. The differential $d\Psi|_{Z}$ yields a $\mu_d$-equivariant isomorphism of the normal bundles $\cN_{Z/\cX}\cong_{\mu_d,B}\cN_{Z'/\cX'}$. Furthermore,  
$\cN_{Z/\cX}\cong_{\mu_d, B}  \xi\oplus {\bf 1_n}$ and $\cN_{Z'/\cX'}\cong_{\mu_d,B} \xi\oplus {\bf 1_n}$ as $\mu_d$-vector bundles. Therefore, one has $\xi\oplus {\bf 1_n}\cong_{\mu_d,B} \xi'\oplus {\bf 1_n}$, that is, the  line bundles $\xi$ and $\xi'$ are stably $\mu_d$-equivariantly equivalent. In fact, they are $\mu_d$-equivariantly equivalent. Indeed, one has (see \cite[\S 8, Corollary]{Se1}) $$\xi\cong_{\mu_d, B}\det(\xi\oplus {\bf 1_n})\cong_{\mu_d,B} \det(\xi'\oplus {\bf 1_n})\cong_{\mu_d,B}\xi'\,,$$   as stated. 
\eproof

\subsection{Parabolic $\G_m$-surfaces: an overview}\label{ss: parabolic-surfaces}
\bdefis[\emph{the DPD presentation for parabolic $\mathbb{G}_m$-surfaces}]\label{sit: toric-surfaces} (\cite{FZ}) A {\em parabolic $\mathbb{G}_m$-surface} is a normal affine  surface  $X$ endowed with an effective $\mathbb{G}_m$-action along the fibers of an $\A^1$-fibration $\pi\colon X\to C$ over a smooth affine curve $C$. The $\mathbb{G}_m$-action on $X$ defines a grading 
$$\cO_X(X)=\bigoplus_{n\ge 0} A_n\quad\mbox{where}\quad A_n=H^0(C,\cO_C(\lfloor nD_X\rfloor)\,\,\,\forall n\ge 0$$ for a $\Q$-divisor $D_X$ on $C$. 
This is called a {\em Dolgachev-Pinkham-Demazure  presentation}, or a {\em DPD presentation} for short, see \cite[Thm.\ 3.2]{FZ}.
The $\Q$-divisor $D_X$ on $C$ is uniquely defined by the class of isomorphism of $\pi\colon X\to C$ up to a linear equivalence. 
Any fiber $\pi^*(p)$, $p\in C$, is irreducible of multiplicity $m$ where
$D_X(p)=(e/m)[p]$ with coprime $e,m\in\Z$. Any reduced fiber $\pi^{-1}(p)$ is smooth and isomorphic to $\A^1$  (\cite[Rem.\ 3.13(iii)]{FZ-lnd}). The  projection $\pi\colon X\to C$ admits a section consisting of the fixed points of the $\mathbb{G}_m$-action on $X$. The singularities of $X$ 
are the fixed points of the $\mathbb{G}_m$-action  in the multiple fibers of $\pi$. 
More precisely, if $D_X(p)=(e/m)[p]$ where $m>1$ and $e,m$ are coprime then the unique fixed point $x_p$ over $p$ is a cyclic quotient singularity of type $(m,e')$ where $e'\in\{1,\ldots,m-1\}$ and $e'\equiv e\mod\, m$, see \cite[Prop.\ 3.8(b)]{FZ}. The following analog of Proposition 4.12 in \cite{FZ} deals with parabolic (instead of hyperbolic) $\mathbb{G}_m$-surfaces. 
\edefis

\blem\label{lem: covering divisor} 
Given a parabolic $\mathbb{G}_m$-surface $\pi\colon X\to C$ and a branched covering 
$\mu\colon B\to C$,  let $\pi'\colon X'\to B$ be obtained from the cross-product $B\times_C X$ via normalization. Then $\pi'\colon X'\to B$ is again a parabolic $\mathbb{G}_m$-surface. The $\Q$-divisors $D_X$ on $C$ and $D_{X'}$ on $B$ in the 
corresponding  DPD 
presentations are related via $D_{X'}=\mu^*D_X$.\elem

\bproof The projection $\pi\colon X\to C$ is the orbit morphism of a parabolic 
$\mathbb{G}_m$-action, say, $\Lambda$ on $X$ with all its fibers being smooth and irreducible. 
Hence the fibers of $\pi'\colon X'\to B$ are also irreducible and $\Lambda$ 
lifts to a parabolic $\mathbb{G}_m$-action $\Lambda'$ on the cross-product 
$B\times_C X$ where $\lambda\colon (b,x)\mapsto (b, \lambda. x )$ $\forall\lambda\in\mathbb{G}_m$. This lifted action survives in the normalization 
$X'\to B\times_C X$. Thus, $\Lambda$ lifts to a parabolic $\mathbb{G}_m$-action $\Lambda'$ on $X'$ 
such that $\pi'\colon X'\to B$ is the orbit morphism. The induced morphism $\mu'\colon X'\to X$ is $\mathbb{G}_m$-equivariant. 

On the other hand, consider the $\Q$-divisor $D_{X''}=\mu^*D_X$ on $B$ and the corresponding
parabolic $\mathbb{G}_m$-surface $\pi''\colon X''\to B$ with the DPD presentation related to the pair $(B,D_{X''})$. For any $n\ge 0$ there is a natural embedding $A_n=H^0(C,\cO_C(\lfloor nD_X\rfloor))\hookrightarrow \hat A_n=H^0(B,\cO_B(\lfloor nD_{X''}\rfloor))$. This yields a monomorphism of graded rings $$\cO_X(X)=\bigoplus_{n\ge 0} A_n\hookrightarrow\cO_{X''}(X'')=\bigoplus_{n\ge 0} A''_n\,$$ and the induced $\mathbb{G}_m$-equivariant surjection $\mu''\colon X''\to X$ that fits in the commutative diagram
$$
 \bdi
X''&\rTo<{\mu''} &  X\\
\dTo<{\pi''} &  & \dTo>{\pi} \\
B &\rTo>{\mu} & C \, 
\edi
$$
By the universal property of the cross-product, $\mu''$ can be factorized as $$\mu''\colon X''\to B\times_C X \stackrel{\pi}{\longrightarrow} X\,.$$ Since $X''$ is normal one has as well a factorization $\mu''\colon X''\stackrel{\psi}{\longrightarrow} X' \stackrel{\pi}{\longrightarrow} X$ where $\psi$ is a $\mathbb{G}_m$-equivariant surjection fitting in the commutative diagram
$$
 \bdi
X''&&\rTo<{\psi} &&   X'&\\
&\rdTo<{\pi''} &  & \ldTo>{\pi'}& \\
&&B&& \, .
\edi
$$
Since the fibers of a parabolic $\mathbb{G}_m$-surface are irreducible   (\cite[Rem.\ 3.13(iii)]{FZ-lnd}), $\psi$ is a bijection. Due to the normality of both $X''$ and $X'$, $\psi$ is an isomorphism. Now the desired conclusion follows. 
\eproof

\bprop\label{prop: Gm-surfaces} Consider an $\A^1$-fibration $\pi\colon X\to C$ on a normal affine surface $X$ over a smooth affine curve $C$. Let $\tilde\pi\colon \tilde X\to B$ be a marked GDF $\mu_d$-surface obtained from $\pi\colon X\to C$ via a cyclic base change $\delta\colon B\to C$ with the Galois group $\mu_d$ and  a subsequent normalization as in Lemma {\rm \ref{lem: br-covering}}.
Then the following are equivalent.
\begin{itemize}\item[(i)] $\tilde\pi\colon \tilde X\to B$ admits a structure of a line bundle;
\item[(ii)]  $\pi\colon X\to C$  admits a structure of  a parabolic $\mathbb{G}_m$-surface.
 \end{itemize}
\eprop

\bproof (i)$\Rightarrow$(ii). In case (i) the fibers of $\pi\colon X\to C$ are irreducible. If all of them are reduced then $\pi\colon X\to C$ admits a  structure of a line bundle, and so, (ii) holds. Otherwise, $\pi\colon X\to C$ has multiple fibers. If a fiber $F_c=\pi^{-1}(c)$ is multiple then the branched covering construction of  \ref{sit: construction} creates a unique point $b=\delta^{-1}(c)\in B$ over $c$. This point $b$ is a fixed point of the $\mu_d$-action on $B$.

The $\mu_d$-action on $ \tilde X$ preserves the fibration $\tilde\pi\colon \tilde X\to B$ and sends the sections of $\tilde\pi$ to sections. Taking the fiberwise barycenter of the $\mu_d$-orbit of the zero section yields a $\mu_d$-invariant section, say, $Z$ of $\tilde \pi$. 
There is a new line bundle structure on $\tilde X$ with projection $\tilde\pi$ and the associated parabolic $\mathbb{G}_m$-action $\tilde\Lambda$ on $\tilde X$ along the fibers of $\tilde\pi$ with $Z$ as the fixed point set. In fact, 
$\tilde\Lambda=t_Z\circ\Lambda\circ t_Z^{-1}$ where $\Lambda$ stands for the parabolic $\mathbb{G}_m$-action on $\tilde X$ associated with the original line bundle structure on $\tilde X$, and $t_Z\in\Aut_B\tilde X$ is the translation on $Z$ in the vertical direction. 

The $\mathbb{G}_m$-action $\tilde\Lambda$ commutes with the $\mu_d$-action on $\tilde X$. Indeed, the conjugation of $\tilde\Lambda$ by elements of the $\mu_d$-action yields a homomorphism $\mu_d\to\Aut\G_m\cong\ZZ/2\ZZ$. To show that this homomorphism is trivial we take a fixed point $b\in B$ of $\mu_d$. The fiber $\tilde F_b=\tilde\pi^{-1}(b)\cong\A^1$ is $\mu_d$-invariant and $Z\cap\tilde F_b$ is a common fixed point of $\mu_d$ and $\tilde\Lambda$. Hence $\mu_d|_{\tilde F_b}\subset\tilde\Lambda|_{\tilde F_b}$. Since  $\mu_d$ and $\tilde\Lambda$ commute when restricted to $\tilde F_b$ they commute on $\tilde X$.

Therefore, $\tilde\Lambda$ descends to a parabolic $\mathbb{G}_m$-action on the quotient $X=\tilde X/\mu_d$ along the fibers of $\pi\colon X\to C$ thus converting $X$ into a parabolic  $\mathbb{G}_m$-surface over $C$. 

(ii)$\Rightarrow$(i). Conversely, suppose that $\pi\colon X\to C$ is the orbit morphism of a parabolic $\mathbb{G}_m$-action $\Lambda$ on $X$. Then $\Lambda$ lifts to a parabolic $\mathbb{G}_m$-action $\tilde\Lambda$ on the cross-product $B\times_C X$ where $\lambda\colon (b,x)\mapsto (b, \lambda. x )$ $\forall\lambda\in\mathbb{G}_m$. This lifted action survives in the normalization $\tilde X\to B\times_C X$. Thus, $\Lambda$ lifts  through the branched covering $\tilde X\to X$ as in \ref{sit: construction}. In this way the GDF surface $\tilde\pi\colon \tilde X\to B$ acquires an effective parabolic $\mathbb{G}_m$-action $\tilde\Lambda$ along the fibers of  $\tilde\pi$. Hence all these fibers are reduced and irreducible, cf.\ \cite[Rem.\ 3.13(iii)]{FZ-lnd}. This allows to define a line bundle structure on $\tilde\pi\colon \tilde X\to B$. 
\eproof

\brems\la{rem: DPD} 1.
The integral divisor $-\cD_{\tilde X}=-\mu^* D_X\in\Pic B$ is associated with the line bundle $\tilde\pi\colon\tilde X\to B$. This divisor $\cD_{\tilde X}$ determines a DPD presentation of the $\G_m$-surface $\tilde\pi\colon\tilde X\to B$ (see Lemma \ref{lem: covering divisor}). 
 
 2. 
It is known that a Gizatullin $\mathbb{G}_m$-surface $X$ is toric if and only if the associated extended graph 
$\Gamma_{\rm ext}$ is linear, see \cite[Lem.\ 2.20]{FKZ-completions}. A similar criterion holds for the parabolic  $\mathbb{G}_m$-surfaces. Namely, one can show  that conditions (i) and (ii) of Proposition \ref{prop: Gm-surfaces} are equivalent to
 the following one (cf.\ \cite[Prop.\ 3.22]{FKZ-completions}):
 
 \smallskip
 
(iii) \emph{ $\Gamma_{\rm ext}$ is a bush, that is, any fiber tree $\Gamma_c(\bar\pi)$, $c\in C$, is a chain.}

\smallskip

\noindent  As usual, $\Gamma_{\rm ext}$ stands for the extended graph of a pseudominimal resolved  completion $\bar\pi\colon \bar X\to\bar C$ of the $\A^1$-fibered  surface $\pi\colon X\to C$ as in Proposition \ref{prop: Gm-surfaces}. This graph is viewed  as a rooted tree 
 with the section at infinity $S$  as the root vertex. 
\erems

\subsection{Parabolic $\mathbb{G}_m$-surfaces as Zariski factors}

The following theorem is the main result of Section \ref{sec: basic-examples}.

\bthm\label{thm: Gm-surfaces-are-zar-factors} Any parabolic $\mathbb{G}_m$-surface is a Zariski factor.
\ethm

The proof is done in Lemmas \ref{lem: proof-under-assumptions}--\ref{lem: A1-fibration}. It is preceded by several auxiliary facts. In the next elementary lemma we use the following terminology.

\bdefi\label{sit: uniruledness} Given a morphism $\pi\colon X\to C$ of a normal affine variety $X$ onto a smooth curve $C$, the fiber over a point $c\in C$ is called {\em multiple} if $d>1$ where $d$ is the greatest common divisor of the multiplicities of the components of $\pi^*(c)$. 
\edefi
 
\blem\label{lem: two-mult-fibers} 
A polynomial of one variable cannot have two or more multiple fibers. 
\elem

\bproof
Suppose that $f\in \mathbb{k}[t]$ (is nonconstant and) has at least two multiple fibers, say, $f^*(0)$ and $f^*(1)$. Then $f=p^r=1-q^s$ for some polynomials $p,q\in \mathbb{k}[t]$ such that $p^r+q^s=1$, $r,s\ge 2$, $\deg p=d/r$, and $\deg q=d/s$ where $d=\deg f$. The derivative $f'$ vanishes to order $r-1$ at any root of $p$ and to order $s-1$ at any root of $q$. More precisely, since $p$ and $q$ do not have any common root one has
$$\div f' \ge (r-1)\div p + (s-1)\div q\,.$$ Taking the degrees one gets the inequalities  $$(r-1)/r+(s-1)/s\le (d-1)/d\,.$$
Since $r,s\ge 2$ it follows that
$$1\le \left(1-\frac{1}{r}\right)+\left(1-\frac{1}{s}\right)\le \left(1-\frac{1}{d}\right)\,.$$
This gives a contradiction. Alternatively, letting $x=p(t)$, $y=q(t)$ yields a parametrization of the plane affine curve $E=\{x^r+y^s=1\}$. However, there is no nonconstant morphism $\A^1\to E$. 
\eproof

An affine variety is called {\em $\A^1$-uniruled} if a general point of $X$ belongs to the image of a nonconstant morphism $\A^1\to X$. 
One can find in the literature different versions of the following results, see, e.g., \cite{Ru} and \cite[Thm.\ 4.1]{GM}. 

\blem\label{lem: mult-fiber} Consider a dominant morphism $\pi\colon X\to C$  from an affine variety $X$ to  a smooth affine curve $C$. 
Assume that  one of the following conditions is fulfilled.
\begin{itemize} \item[(i)] $C\not\cong\A^1$; \item[(ii)] $C\cong\A^1$ and $\pi$ has at least two multiple fibers. \end{itemize}
Then the following hold.
\begin{itemize}\item[(a)] Any morphism  $\A^n\to X$ has image contained in a fiber of $\pi$. Consequently, there is no dominant morphism $\A^n\to X$.
\item[(b)] 
If the general fibers of $\pi$ are $\A^1$-uniruled then any automorphism $\alpha\in\Aut X$ preserves the fibration $\pi\colon X\to C$, that is, sends the fibers to fibers.
\end{itemize}
\elem

\bproof (a) In case (i) any morphism $\A^n\to C$ is constant, hence the assertion follows. Assuming (ii) suppose to the contrary that there exists $F\colon\A^n\to X$ such that $f:=\pi\circ F\colon\A^n\to\A^1$ is nonconstant. Then $f\in \mathbb{k}[t_1,\ldots,t_n]$ is a nonconstant polynomial with two distinct multiple fibers. The latter contradicts Lemma \ref{lem: two-mult-fibers}.

(b) If  $\alpha\in\Aut X$ does not preserve the fibration $\pi\colon X\to C$ then there is a morphism $\phi\colon\A^1\to X$ such that the composition $f=\pi\circ\phi$ is not constant. Then $C\cong\A^1$, and so, $f\in \mathbb{k}[t]$ is a polynomial with two multiple fibers. This leads to a contradiction  as before.
\eproof

In the proof of Theorem \ref{thm: Gm-surfaces-are-zar-factors} we use the following auxiliary Lemmas \ref{lem: rank-Pic} and \ref{lem: isom-sing}.

\blem\label{lem: rank-Pic} Let  $\pi\colon X\to \PP^1$ be an $\A^1$-fibration on  a normal affine surface $X$. Assume that the group $\Pic X$ is  finite and $\pi(X)\supset\A^1=\PP^1\setminus\{\infty\}$. Then $\pi(X)=\A^1$, all the fibers of $\pi$ are irreducible, and the divisor class group $\Cl\, (X)$ is generated by the classes of the multiple fibers of $\pi$. 
\elem

\bproof Let $\bar\pi\colon\bar X\to\PP^1$ be a resolved completion of $\pi\colon X\to\PP^1$ with extended graph $\Gamma_{\rm ext}$. Then  $\Gamma_{\rm ext}$ is a rooted tree with the section at infinity $S$ as the root. We let 
\begin{itemize}
\item 
$\cB_1,\ldots,\cB_n$ be the degenerate fibers of $\bar\pi$ over the points $b_i\in\PP^1$, $i=1,\ldots,n$;
\item 
$D=\bar X\setminus X_{\rm resolved}$ stand for the boundary divisor;
\item
$E$ stand for the exceptional divisor of the resolution of singularities $X_{\rm res}\to X$;
\item
$m_i\ge 0$ be the number of components of the fiber $\pi^{-1}(b_i)$;
\item 
$n_i\ge 0$ be the number of components of $\cB_i$ which are components of $D+E$.
\end{itemize}
 Thus, $\cB_i$ consists of $n_i+m_i$ components.  Contracting subsequently $(-1)$-fiber components one arrives finally at a Hirzebruch surface $\mathbb{F}_s$.  In this way one contracts $n_i+m_i-1$ components of $\cB_i$, $i=1,\ldots,n$. Let $\rho(V)$ be the Picard number of a variety $V$. Since $\rho(\mathbb{F}_s)=2$ one has $\rho(\bar X)=2+\sum_{i=1}^n (n_i+m_i-1)$. Letting $\natural \cD$ be the number of components of a divisor $\cD$ one gets
$$0=\rho(X)=\rho(\bar X)-\natural (D+E)$$
\be\la{eq:Picard-number} =\left(2+\sum_{i=1}^n (n_i+m_i-1)\right) - \left(1+\sum_{i=1}^n n_i\right)=1+\sum_{i=1}^n (m_i-1)\,.\ee 
It follows that 
\begin{itemize}
\item 
$m_i\le 1$ $\forall i=1,\ldots,n$, that is, the fibers of $\pi\colon X\to \pi(X)$ are irreducible;
\item
 $m_i=0$ for exactly one value of $i$, that is, $\pi(X)=\A^1$.
 \end{itemize}

Let $\omega\subset\A^1$ be a Zariski open dense subset such that $U=\pi^{-1}(\omega)$ is isomorphic over $\omega$ to the cylinder $\omega\times\A^1$, and so, $\Cl\, (U)=0$. For $\cD=X\setminus U$ one has the exact sequence $\Div (\cD)\to \Cl\, (X)\to\Cl\, (U)\to 0$ where $\Div (\cD)$ is the subgroup of Weil divisors on $X$ supported by $\cD$, see \cite[p.\ 206]{Mi}. Thus $\Cl\, (X)$ is generated by the fibers of $\pi$ contained in $\cD$. Any reduced fiber of $\pi$ represents the zero class in $\Cl\, (X)$. Hence $\Cl\, (X)$ is generated by the classes of the multiple fibers of $\pi$. 
\eproof

\blem\label{lem: isom-sing} Let $\pi\colon X\to C$ be a parabolic $\mathbb{G}_m$-surface with a singular point $x\in X$, and let    $X'$ be
 a normal affine surface.  Suppose that there is an isomorphism $\phi\colon\cX'\stackrel{\cong}{\longrightarrow}\cX$ of the $n$-cylinders $\cX=X\times\A^n$ and $\cX'=X'\times\A^n$. Let $\phi(\{x\}\times\A^n)=\{x'\}\times\A^n$ where $x'\in \Sing X'$. Then the germs  of  surface singularities $(X,x)$ and $(X',x')$ are isomorphic. \elem

\bproof Let $\sigma_1\colon X_1\to X$ be the blowup of the maximal ideal of the unique singular point $x\in X$ followed by a  normalization. The induced morphism of $n$-cylinders $\sigma_1\times\id\colon \cX_1\to\cX$ consists in the blowup of the ideal of the singular ruling ${\rm sing}\,\cX=\{x\}\times\A^n$ and a subsequent normalization. By a theorem of Zariski (\cite{Za}; see also \cite{Li}) a sequence of blowups in maximal ideals and subsequent normalizations 
$$X_N\stackrel{\sigma_N}{\longrightarrow} X_{N-1}\to\ldots\to X_1\stackrel{\sigma_1}{\longrightarrow} X_0=X\,$$ resolves the singularity $(X,x)$. It induces a similar sequence of blowups  in rulings of our $n$-cylinders  and subsequent normalizations 
$$\cX_N\stackrel{\sigma_N}{\longrightarrow} \cX_{N-1}\to\ldots\to \cX_1\stackrel{\sigma_1}{\longrightarrow} \cX_0=\cX\,$$ which results in a resolution of the corresponding singularities of $\cX\cong\cX'$. The exceptional divisor of the resolution $\cX_N\to\cX$ is $\cE=E \times\A^n$ where $E$ is the exceptional divisor of the resolution $X_N\to X$.  

Let further $\sigma'_1\colon X'_1\to X'$ be the blowup of the maximal ideal of the  singular point $x'\in X'$ followed by a normalization. Then $\sigma'_1\times\id\colon \cX'_1\to\cX'$ is the blowup of the ideal of the singular ruling $\{x'\}\times\A^n$  followed by a normalization. Under the isomorphism $\psi:=\phi^{-1}\colon\cX'\stackrel{\cong}{\longrightarrow}\cX$ this ruling goes to the ruling $\{x\}\times\A^n$. Hence $\phi$ lifts to an isomorphism $\psi_1\colon\cX'_1\stackrel{\cong}{\longrightarrow}\cX_1$. Continuing in this way one arrives finally at a resolution $\cX_N'\to\cX'$ where  $\cX_N'=X'_N\times\A^n\cong\cX_N$ with exceptional divisor $\cE'=E' \times\A^n\cong\cE=E \times\A^n$ where $E'$ is the exceptional divisor of the induced resolution of singularity $X'_N\to X'$.  Under this procedure the singularities of the embedded surfaces $X\times\{0\}\subset\cX$ and $X'\times\{0\}\subset\cX'$ are simultaneously resolved and there is an isomorphism $\psi_N\colon\cX'_N\stackrel{\cong}{\longrightarrow}\cX_N$ such that $\psi_N(\cE')=\cE$. The only irreducible complete curves in $\cE$ (in $\cE'$, respectively) are of the form $E_i\times\{v\}$ ($E_i'\times\{v'\}$, respectively) where $E_i$ and $E_i'$ are components of $E$ and $E'$, respectively, and $v,v'\in\A^n$. 
Given such a curve $E_i'\times\{v'\}$ there is a curve $E_{\sigma(i)}\times\{v\}$
such that $\psi(E'_i\times\{v'\})=E_{\sigma(i)}\times\{v\}$. It follows that $\psi(E'_i\times\A^n)=E_{\sigma(i)}\times\A^n$. The image $\psi(X'_N\times\{v'\})$ is a smooth surface in $\cX_N$ which meets the exceptional divisor $\cE\subset\cX_N$ transversely along the curve $\psi(E'\times\{v'\})=E\times\{v\}\subset X_N\times\{v\}$. The same is true for $X_N\times\{v\}$. Namely, the latter is a smooth surface in $\cX_N$ which meets $\cE$ transversely along the same curve $E\times\{v\}$. Projecting the both surfaces to $X_N$ via the canonical projection $\cX_N\to X_N$ yields a local isomorphism of  the surface germs $(\psi(X_N'\times\{v'\}), E\times\{v\})$ and $(X_N\times\{v\}, E\times\{v\})$ near the common exceptional divisor  $E\times\{v\}$. Contracting the divisor $E\times\{v\}$ yields an isomorphism between the singular germs $(X,x)$ and $(X',x')$.
\eproof

The next lemma gives a proof of Theorem \ref{thm: Gm-surfaces-are-zar-factors} under an additional assumption. 

\blem\label{lem: proof-under-assumptions} 
Let $\pi\colon X\to C$ be a parabolic $\mathbb{G}_m$-surface, and let    $X'$ be
 a normal affine surface.  Assume that there is an isomorphism $\phi\colon\cX'\stackrel{\cong}{\longrightarrow}\cX$ of the $n$-cylinders $\cX=X\times\A^n$ and $\cX'=X'\times\A^n$. Suppose also that for the induced $\A^{n+1}$-fibration $\hat\pi\colon\cX\to C$ one of the conditions {\rm (i)} and {\rm (ii)} of Lemma {\rm \ref{lem: mult-fiber}} is fulfilled. Then  $X'\cong X$.\elem

\bproof By Lemma \ref{lem: mult-fiber} one has $\hat\pi\circ\phi|_{\{x'\}\times\A^n}={\rm cts}(x')\in C$. This provides a surjection $\pi'\colon X'\to C$ which extends to a morphism $\hat\pi'=\pi'\circ {\rm pr}_1\colon \cX'\to C$ fitting in the commutative diagram 
\be\label{diagr: case-1}
 \bdi
 \cX'&\rTo<{\phi}>{\cong} &  \cX\\
\dTo<{\hat\pi'} &  & \dTo>{\hat\pi} \\
C &\rTo>{\id} & C \, 
\edi
\ee
For any point $c\in C$, $\phi$ restricts to an isomorphism $$\phi|_{{\pi'}^{-1}(c)\times\A^n}\colon {\pi'}^{-1}(c)\times\A^n\stackrel{\cong}{\longrightarrow} \pi^{-1}(c)\times\A^n\cong\A^{n+1}\,.$$ Since any curve is a Zariski factor (\cite[Thm.\ 6.5]{AHE}) one deduces that any fiber of $\pi'\colon X'\to C$ is isomorphic to $\A^1$. The multiple fibers of $\pi'$ (and $\hat\pi'$) are situated over the same points of $C$ as the ones of $\pi$ (and $\hat\pi$), have the same multiplicities, and each of them carries a unique singular point of $X'$, see \ref{sit: toric-surfaces}. 

Applying to the surfaces $\pi\colon X\to C$ and $\pi'\colon X'\to C$ a suitable branched covering with the same cyclic base change $\mu\colon B\to C$ as in Lemma \ref{lem: br-covering} one obtains GDF $\mu_d$-surfaces $\tilde \pi\colon\tilde X\to B$ and $\tilde \pi'\colon\tilde X'\to B$ and 
 two cyclic branched coverings $\tilde X\to X$ and $\tilde X'\to X'$ with the Galois group $\mu_d$. The same branched covering construction applied to the cylinders $\cX$ and $\cX'$ (that are isomorphic over $C$, see \eqref{diagr: case-1}) yields the cylinders $\tilde\cX=\tilde X\times\A^n$ and $\tilde\cX'=\tilde X'\times\A^n$ along with a $\mu_d$-equivariant commutative diagram
\be\label{diagr: case-2}
 \bdi
 \tilde \cX'&\rTo<{\tilde\phi}>{\cong_{\mu_d,B}} &  \cX \\
\dTo<{\tilde \pi'\times\id} &  & \dTo>{\tilde \pi\times\id} \\
B &\rTo>{\id} & B \, 
\edi
\ee
where $\tilde\phi$ 
is a lift of $\phi$ from \eqref{diagr: case-1}.

 By Proposition \ref{prop: Gm-surfaces},  $\tilde\pi\colon\tilde X\to B$ admits a line bundle structure.  
In particular, the fibers of $\tilde\pi\colon \tilde X\to B$ and the ones of $\tilde\pi\times\id\colon \tilde\cX\to B$ are reduced and irreducible. Hence 
the fibers of $\tilde\pi'\times\id\colon\tilde\cX'\to B$ and the ones of $\tilde\pi'\colon\tilde X'\to B$ are as well. Therefore,  $\tilde\pi'\colon\tilde X'\to B$ also admits a structure of a line bundle. Proceeding as in the proof of (i)$\Rightarrow$(ii) in Proposition \ref{prop: Gm-surfaces} one can choose $\mu_d$-equivariant line bundle structures, say, $\xi$ and $\xi'$ of $\tilde\pi\colon\tilde X\to B$ and $\tilde\pi'\colon\tilde X'\to B$, respectively. In particular, the zero sections are $\mu_d$-invariant. Taking the quotients by the $\mu_d$-actions yields a structure of parabolic $\G_m$-surfaces on $\pi\colon X\to C$ and $\pi'\colon X'\to C$ where the first one is the given $\G_m$-structure.  By Lemma \ref{lem: line-bundle-equiv} there is a $\mu_d$-equivariant isomorphism of line bundles $\xi\cong\xi'$. It induces a $\G_m$-equivariant isomorphism over $C$ of the quotient parabolic $\G_m$-surfaces $X\cong_{\G_m,C} X'$.
\eproof

\bsit\label{sit: further-conventions}
We will suppose in the sequel that in the setting of Theorem \ref{thm: Gm-surfaces-are-zar-factors} 
neither (i) nor (ii) of Lemma \ref{lem: mult-fiber} is fulfilled, that is, $C=\A^1$ and the fibration $\pi\colon X\to \A^1$ 
has at most one multiple fiber. By virtue of the Miyanishi-Sugie-Fujita Theorem
we may suppose as well that
the parabolic $\mathbb{G}_m$-surface $\pi\colon X\to \A^1$ has exactly one  multiple fiber $\pi^{-1}(0)$ of multiplicity $d>0$, and so, $X$ has a unique singular point, say, $x$ which is the unique fixed point of the $\mathbb{G}_m$-action on the multiple fiber $\pi^{-1}(0)$ and a cyclic quotient singularity. 
Let $X'$ be a normal affine surface such that the cylinders $\cX$ and $\cX'$ are isomorphic. Then $X'$ has as well a unique singular point, say, $x'$.  By Lemma \ref{lem: isom-sing} there is a local isomorphism of singularities $(X,x)\cong (X',x')$.
\esit

Remind that a toric variety $X$ is called {\em non-degenerate} if $\cO_X(X)^\times=\mathbb{k}^*$. Any non-degenerate affine toric surface is isomorphic to the quotient of $\A^2$ by a diagonal $\mu_d$-action 
\be\la{diag-action} \zeta.(x,y)=(\zeta x, \zeta^e y)\quad\mbox{where}\quad\zeta^d=1,\,\,\,1\le e<d,\,\,\,\gcd(e,d)=1\,\ee (see, e.g., \cite[Ex.\ 2.3]{FZ} or \cite[Ex.\ 2.8]{FZ-lnd}). One has $\Cl\, (\A^2/\mu_d)\cong\ZZ/d\ZZ$.

\blem\label{lem: toric-1} 
Under the assumptions of {\rm\ref{sit: further-conventions}}, $X$ is a non-degenerate affine toric surface.
\elem
\bproof According to Proposition \ref{prop: Gm-surfaces} the branched covering construction 
applied to $\pi\colon X\to\A^1$  with the cyclic base change $\A^1\to\A^1$, $z\mapsto z^d$, yields a a line bundle $\tilde\pi\colon\tilde X\to\A^1$ which is trivial since $\Pic\A^1=0$. Its zero section is $\mu_d$-invariant, hence the line bundle structure is $\mu_d$-equivariant. Via an isomorphism $\tilde X\cong\A^2$ one obtains an effective action of $\mu_d$ on $\A^2$ which can be linearized taking the form \eqref{diag-action} in appropriate coordinates on $\A^2$. Then $\tilde\pi$ becomes the standard projection $\A^2\to\A^1$, $(x,y)\mapsto x$. Thus $X\cong\A^2/\mu_d$ is an affine toric surface of type $(d,e)$. 
\eproof

\bcor\la{cor:cl-grp} Under the assumptions of {\rm\ref{sit: further-conventions}} one has $\Cl\, (X')\cong\Cl\, (X)\cong\ZZ/d\ZZ$.
\ecor

\bproof
Recall  (see \cite[Thm.\ 8.1]{Fo}; cf.\ \cite[(9.9.8)]{Fu}) that $\Cl\, (X)$ is a cancellation invariant. Since $\cX\cong\cX'$ there are isomorphisms
\be\label{eq: seq-isom} \Cl\, (X')\cong\Cl\,(\cX')\cong\Cl\,(\cX)\cong\Cl\, (X)\cong\Z/d\Z\,.\ee
\eproof

We use below the following  simple version of the Cox ring (see  \cite{ADHL}, \cite{Cox}). 

\bdefi[\emph{Cox ring}]\label{def: Cox} Let $X$ be a normal affine variety with $\cO_X(X)^\times=\mathbb{k}^*$. Suppose that the  divisor class group $\Cl\, (X)$ is a finite cyclic group of order $d$ generated by the class of a Weil divisor $F$ on $X$.
Consider the ($\Z/d\Z$)-graded Cox ring 
$${\rm Cox}\,\cO_X(X):=\bigoplus_{j=0}^{d-1} H^0(X,\cO_{X}(jF))\zeta^j\,$$
where $\zeta\in k^\times$ is a primitive $d$th root of unity. Then $\tilde X=\Spec {\rm Cox}\,\cO_X(X)$   is a normal affine variety equipped with a $\mu_d$-action defined via the ($\Z/d\Z$)-grading on  $\cO_{\widetilde X}(\tilde X)={\rm Cox}\,\cO_X(X)$, see \cite[Thm.\ 1.5.1.1]{ADHL}. The natural embedding $\cO_X(X)\hookrightarrow {\rm Cox}\,\cO_X(X)$ onto the  subalgebra of $\mu_d$-invariants yields the quotient morphism $\tilde X\to X=\tilde X/\mu_d$. We call this morphism a {\em  Cox   covering construction}. 

Consider the $n$-cylinder $\cX=X\times\A^n$ over $X$. The divisor class group $\Cl\,\cX\cong\Cl\, (X)\cong\mu_d$ is generated by the class of the Weil divisor $\cF=F\times \A^n$ on $\cX$, see \cite[Thm.\ 8.1]{Fo}. The  Cox covering construction applied to $(\cX,\cF)$ yields the $n$-cylinder $\tilde\cX=\tilde X\times\A^n$.  
\edefi

For the next lemma we provide two alternative proofs.

\blem\label{lem: cyclic-factor} In the setup of {\rm \ref{sit: further-conventions}}, $X'$ is a non-degenerate affine toric surface.
\elem

\bproof By Lemma \ref{lem: toric-1} one has $\cO_{X}(X)^\times=\mathbb{k}^*$. Since $\cO_{\cX'}(\cX')^\times\cong\cO_{\cX}(\cX)^\times=\mathbb{k}^*$ then also $\cO_{X'}(X')=\mathbb{k}^*$. Consider the Weil divisor $F_0=\pi^{-1}(0)$ on $X$. Its class generates the class group $\Cl\, (X)\cong\mu_d$. By \cite[Thm.\ 3.1]{AG} the branched covering $\A^2\to X=\A^2/\mu_d$ as in Lemma \ref{lem: toric-1} coincides with the Cox covering defined by the pair $(X, F_0)$, see Definition \ref{def: Cox}. 
Letting $\cF_0=F_0\times\A^n$ and applying the cyclic  Cox covering construction to the pair $(\cX,\cF_0)$ one obtains the $n$-cylinder $\tilde\cX=\A^2\times\A^n=\A^{n+2}$.

Choose a Weil divisor $F'_0$ on $X'$ whose class in $\Cl\, (X')$ is sent to the class of $F_0$ in $\Cl\, (X)$ via the isomorphisms  (\ref{eq: seq-isom}).   
Applying the Cox covering construction to the pair $(X', F'_0)$ leads to a cyclic $\mu_d$-covering $\tilde X'\to X'$. Letting $\cF_0'= F_0'\times\A^n$ and
applying the Cox covering construction to the pair $(\cX',\cF'_0)$ yields the $n$-cylinder $\tilde\cX'=\tilde X'\times\A^n$. We claim that $\tilde\cX'$ is isomorphic to $\tilde \cX\cong\A^{n+2}$. Indeed, let $\varphi\colon\cX\stackrel{\cong}{\longrightarrow}\cX'$ be an isomorphism. One has $\cF'_0 \sim \phi_*\cF_0$ on $\cX'$. The Cox covering construction does not depend, up to an isomorphism, on the choice of a divisor in the class generating the group $\Cl\, (\cX')\cong\Z/d\Z$, see  \cite[Prop.\ 1.4.2.2]{ADHL}.  Hence applying this construction to the pair $(\cX',  \phi_*\cF_0)$ yields a variety, say, $\hat\cX'$ isomorphic to $\tilde\cX'$. On the other hand, the isomorphism of pairs $\varphi\colon (\cX,\cF_0)\stackrel{\cong}{\longrightarrow} (\cX',\varphi_*\cF_0)$ leads to an isomorphism $\hat\cX'\cong\tilde\cX$.

It follows that  $\tilde X'\times\A^n\cong\A^{n+2}$. By the Miyanishi-Sugie-Fujita Theorem (\cite[Cor.\ 3.3]{Fu}, \cite[Ch.\ 3, Thm.\ 2.3.1]{Mi}) one has $\tilde X'\cong\A^2$. Since the $\mu_d$-action on $\A^2$ can be linearized (see, e.g., \cite[Thm.\ 2]{Fur}) the quotient $X'\cong\A^2/\mu_d$ is an affine toric surface. 
\eproof

\bsit[\emph{The second proof of Lemma} \ref{lem: cyclic-factor}]\label{sh-another proof}
Alternatively, one can argue as follows.  

\smallskip

\noindent {\bf Claim  1.} \emph{There exists an $\A^1$-fibration $\pi'\colon X'\to \A^1$.}

\smallskip

\noindent \emph{Proof of Claim $1$.}
By the Iitaka-Fujita Theorem (\cite{IF}) the log-Kodaira dimension is a cancellation invariant. Hence  $\bar k(X')=\bar k(X)=-\infty$.
Therefore, $X'$ admits an $\A^1$-fibration $\pi'\colon X'\to C$, see \cite[Ch.\ 2, Thm.\ 2.1.1]{Mi}. Composing the induced surjection $\cX'\to C$ with an isomorphism $\varphi\colon\cX\stackrel{\cong}{\longrightarrow}\cX'$  one concludes that either $C\cong\A^1$ or $C\cong\PP^1$. However, the latter is impossible. Indeed, by Corollary \ref{cor:cl-grp} one has $\Cl\, (X')\cong\Cl\, (X)\cong\ZZ/d\ZZ$ where $d$ is the multiplicity of the fiber $F_0=\pi^{-1}(0)$ through the singular point $x\in X$.  Let $F_0'$ be the fiber  of ${\pi'}$ through the singular point $x'$ of $X'$. Since $\Cl\, (X')\cong\ZZ/d\ZZ$ the  divisor $dF'_0$ on $X'$ is principal, that is, $dF'_0=\div f$ for some $f\in\cO_{X'}(X')$. Since $f$ is constant on any 
 $\A^1$-fiber of $\pi'$ one has $f={\pi'}^*(g)$ where $g\in\cO_C(C)$ is not a constant. Hence $C$ cannot be a complete curve. 
 
\smallskip

\noindent {\bf Claim  2.} \emph{Let $\pi'\colon X'\to \A^1$ be an $\A^1$-fibration. Then each fiber of $\pi'$ is irreducible, there exists exactly one multiple fiber of $\pi'$, this fiber has multiplicity $d$ and contains the singular point $x'$ of $X'$.}

\smallskip

\noindent \emph{Proof of Claim $2$.}
The irreducibility of the fibers of $\pi'$ follows from Lemma \ref{lem: rank-Pic}. Suppose to the contrary that $\pi'\colon X'\to\A^1$ has 
 two or more multiple fibers. Then by Lemma \ref{lem: mult-fiber} any automorphism of the cylinder $\cX'$ preserves the induced $\A^{n+1}$-fibration $\cX'\to\A^1$. The same must be true for $\cX\cong\cX'$ and the induced fibration $\cX\to\A^1$. However, the cylinder $\cX$ over the non-degenerate affine toric surface $X$ is flexible (\cite[Thm.\ 2.1]{AKZ}). This leads to a contradiction. 
 
Let $F_0'=(\pi')^{-1}(0)$ be the unique multiple fiber of $\pi'$. By Corollary \ref{cor:cl-grp} one has \be\la{eq:clgrp} \Cl\,\, (X')=\langle F_0'\rangle\cong\Cl\, (X)=\langle F_0\rangle\cong\ZZ/d\ZZ\,.\ee It follows that $(\pi')^*(0)=dF_0'$ and, in turn, $\pi'(x')=0\in \A^1$.  This proves Claim 2.

\smallskip

Consider further pseudominimal resolved completions $\bar\pi\colon\bar X\to\PP^1$ and $\bar\pi'\colon\bar X'\to\PP^1$ of $\pi\colon X\to\A^1$ and $\pi'\colon X'\to\A^1$, respectively. Let $\bar F_0\subset \bar X$ and $\bar F_0'\subset \bar X'$ be the fiber components which contain the proper transforms of $F_0$ and $F_0'$, respectively.  Since the fibers $F_0=\pi^{-1}(0)$ and $F'_0=(\pi')^{-1}(0)$ are irreducible and the completions are pseudominimal, $\bar F_0$ and $\bar F_0'$ are the unique $(-1)$-vertices of $\Gamma_0(\bar\pi)$ and $\Gamma_0(\bar\pi')$, respectively.
These are the bridges of the unique feathers $\cF\subset \Gamma_0(\bar\pi)$ and $\cF'\subset\Gamma_0(\bar\pi')$, respectively.
 The chains $\cF\ominus\bar F_0$ and $\cF' \ominus\bar F_0'$ correspond to the exceptional divisors of the minimal resolutions of the cyclic quotient singularities $(X,x)$ and $(X',x')$, respectively. By Lemma \ref{lem: isom-sing} one has $(X,x)\cong (X',x')$. Hence $\cF\cong\cF'$ as ordered chains. 
 
 For the non-degenerate affine toric surface $X$ the extended graph $\Gamma_{\rm ext}$ of $(\bar X, D)$ is a chain (\cite[Lem.\ 2.20]{FKZ-completions}). Hence also the fiber tree $\Gamma_0(\bar\pi)$ is. The tip, say, $R$ of  the chain $\Gamma_0(\bar\pi)$ which is not a tip of $\cF$ meets the section at infinity $S$. Hence $R$ has multiplicity 1 in $\bar\pi^*(0)$. It follows that the chain $\cB=\Gamma_0(\bar\pi)\ominus R$ can be contracted to a smooth point starting with the contraction of the unique $(-1)$-component $\bar F_0$, see Lemma \ref{lem: contracting-in-fiber}.
 
 \smallskip

\noindent {\bf Claim 3.} \emph{The extended graph $\Gamma_{\rm ext}'$ of $(\bar X', D')$  is a chain.}

\smallskip

\noindent \emph{Proof of Claim $3$.} It suffices to show that the fiber tree $\Gamma_0(\bar\pi')$ is a chain.  Suppose the contrary. Then the extremal linear branch $\cB'$ of $\Gamma_0(\bar\pi')$ which contains $\cF'$ is adjacent to a  branching vertex, say,  $R'$ of $\Gamma_0(\bar\pi')$. The contraction of $\cF'$
makes $R'$ a $(-1)$-vertex of degree 2 which contracts to a node of the resulting fiber. Hence $R'$ has multiplicity $m>1$ in $(\pi')^*(0)$. 

As follows from \cite[Lem.\ 4.7]{FKZ-completions} the isomorphism $\cF\cong\cF'$ can be extended to an isomorphism of the contractible chains $\cB\supset\cF$ and $\cB'\supset\cF'$ with the unique $(-1)$-vertices $\bar F_0$ and $\bar F'_0$, respectively. Since the multiplicity of $\bar F_0$ in $\pi^*(0)$ equals $d$  the one of $\bar F_0$ in  $(\pi')^*(0)$ equals $md>d$. However, the latter contradicts \eqref{eq:clgrp}. This proves the claim.

\smallskip

Since the extended graph $\Gamma_{\rm ext}'$ is a chain and $\cO_{X'}(X')^\times=\mathbb{k}^\times$, it follows that $X'$ is a Gizatullin surface.  Applying \cite[Lem.\ 2.20]{FKZ-completions} one concludes that $X'$ is a nondegenerate affine toric surface. \qed
\esit

The next lemma completes the proof of Theorem \ref{thm: Gm-surfaces-are-zar-factors}. 

\blem\label{lem: A1-fibration} 
Under the assumptions of {\rm \ref{sit: further-conventions}}
one has $X'\cong X$.
\elem

\bproof Two non-degenerate toric affine surfaces are isomorphic if and only if their singularities are. By Lemmas \ref{lem: toric-1} and \ref{lem: cyclic-factor}, $X$ and $X'$ are non-degenerate toric affine surfaces, and by Lemma \ref{lem: isom-sing} the singularities $(X,x)$ and $(X',x')$ are isomorphic. Hence $X\cong X'$.
\eproof

\section{
Zariski 1-factors}\la{sec:final}

\subsection{Stretching and rigidity of cylinders}

\bdefi[\emph{Combinatorial stretching}]\la{def: stretching} Let $B$ be a smooth affine curve.
Given an effective divisor $A=\sum_{i} a_ip_i\in\Div (B)$ where $a_i\in\Z_{\ge 0}$ and $p_i\in B$ we 
associate with $A$ a chain divisor $\cD(A)=\sum_{i} L(a_i)p_i$ where $L(a_i)$ is a chain with weights $[[-2,-2,\ldots,-2,-1]]$ of length $a_i$ if $a_i>0$ and  $L(0)=\emptyset$ otherwise. 

Let $\cD=\sum_{i} \Gamma_{i} p_i$ be a  graph divisor, see Definition \ref{def: graph-divisor}.
We let $(A.\cD)_{\bar m}=\cD'=\sum_{i} \Gamma'_{i} p_i$ where
 
\smallskip

\begin{itemize}
\item $\bar m=(m_1,\ldots)$ with $-1\le m_i\le {\rm ht}\,(\Gamma_{i})$;
\item 
$\Gamma'_{i}$ is obtained from $\Gamma_{i}$ by inserting  the chain $L_i$ above each vertex $v$ of $\Gamma_{i}$ on level $m_i$ if $m_i\ge 0$ and below the root $v_i$ if $m_i=-1$ so that the left end $l_i$ of $L_i$ becomes a vertex on level $m_i+1$  of $\Gamma'_{i}$ and its right end $r_i$
is joint with the vertices of $\Gamma_{i}$ on level $m_i+1$ over $v$  if $v$ is not a tip  of  $\Gamma_{i}$ and becomes a tip of  $\Gamma_{i}'$ otherwise. 
The weights change accordingly; the weight of $v$ decreases by 1
and the weight of $r_i$ becomes $-1-s(v)$ where $s(v)$ is the number of 
vertices on level $m_i+1$ in $\Gamma_{i}$ joint with $v$.
\end{itemize}

\smallskip

\noindent 
The transformation $\cD\mapsto (A.\cD)_{\bar m}$ will be called a \emph{combinatorial $(A,\bar m)$-stretching}. It is called a \emph{top-level stretching} if $m_i={\rm ht}\,(\Gamma_{b_i})$ $\forall i=1,\ldots,n$, and an \emph{$(A,\overline{-1})$-stretching} if $m_i=-1$ $\forall i$. 
A combinatorial  $(A,\bar m)$-stretching is called \emph{principal} if $A$ is a principal effective divisor, that is, $A=\div f$ where $f\in\cO_B(B)\setminus\{0\}$.
\edefi

\bdefi[\emph{Geometric stretching}]\label{def: u-modification} Let $\pi'\colon X'\to B$ and $\pi\colon X\to B$  be  two GDF surfaces over $B$.
An affine modification $\sigma: X'\to X$  over $B$ will be called  
a \emph{geometric  $(A,\bar m)$-stretching} where $m_i\ge 0$ $\forall i$ if its effect on the graph divisor $\cD(\pi)$ amounts to a combinatorial
$(A,\bar m)$-stretching $\cD(\pi')=(A.\cD(\pi))_{\bar m}$ as in Definition \ref{def: stretching}. One extends here $\cD(\pi)$ so that $\supp\cD(\pi)\supset\supp A$ adding the new terms $\Gamma_ip_i$ where $p_i\in\supp A\setminus\supp \cD(\pi)$ and $\Gamma_i=[[0]]$. 

A \emph{geometric  $(A,\overline{-1})$-stretching} inserts  the chain $[[-2,\ldots,-2,-1]]$  of length $a_i$ in the fiber tree $\Gamma_{p_i}(\pi)$ between the root and the section $S$ so that the $(-1)$-vertex of this chain becomes the root of the resulting fiber tree $\Gamma_{p_i}(\pi')$. A \emph{principal geometric  $(A,\overline{-1})$-stretching} with $A=\div f$ for $f\in\cO_B(B)\setminus\{0\}$ amounts to perform in (\ref{eq: seq-aff-modif-final}) an affine modification $X_0'\to X_0$ of $X_0=B\times\A^1$ with the divisor $(f\circ\pi)^*(0)$ and the center $f^*(0)\times\{0\}$. In other words, letting $\A^1=\Spec \mathbb{k}[u]$ one has  
\be\la{eq:-1-stretching} \cO_{X_0'}(X_0')=\cO_B(B)[u']\quad\mbox{where}\quad u'=u/f\,.\ee 
Therefore, $X_0'\cong_B B\times\A^1$. Performing the remaining fibered modifications in (\ref{eq: seq-aff-modif-final}) gives again the same surface $X=X_m$. However, this transformation affects the trivializing completion $(\hat X, \hat D)$ as in (\ref{eq: seq-contractions}). 
\edefi

\blem\la{lem: equal-hights} 
Consider a marked GDF surface  $\pi\colon X\to B$ along  with a marking $z\in\cO_B(B)$ where $z^*(0)=b_1+\ldots+b_n$, with a trivializing sequence  {\rm (\ref{eq: seq-aff-modif-final})},  and with the corresponding graph divisor $\cD(\pi)=\sum_{i=1}^n \Gamma_i b_i$. Then the following hold.
\begin{itemize}\item[(a)]
Performing a suitable principal geometric $(A,\overline{-1})$-stretching and extending {\rm \eqref{eq: seq-aff-modif-final}}  accordingly one may assume that 
\be\la{eq:rectifying-heights} \h(\Gamma_{i})=\h(\cD(\pi))\qquad\forall i=1,\ldots,n\,.\ee 
\item[(b)] 
Let, furthermore, $\pi\colon X\to B$ be a marked GDF $\mu_d$-surface. Then \eqref{eq:rectifying-heights} holds after performing a suitable principal $\mu_d$-equivariant $(A,\overline{-1})$-stretching where $A=\div f$ with a $\mu_d$-invariant function  $f\in\cO_B(B)\setminus \{0\}$.
\end{itemize}
\elem

\bproof (a)
Let $$D_\infty=\sum_{c_i\in\bar B\setminus B} c_i \quad\mbox{and}\quad D_0=\sum_{i=1}^n (m-m_i)b_i\quad\mbox{where}\quad m=\h\,(\cD(\pi))\quad\mbox{and}\quad m_i=\h(\Gamma_i)\,.$$ 
Let also $D(r)=rD_\infty-D_0$ where $r\gg 1$. The very ample linear system $|D(r)|$ defines an embedding 
\be\la{eq:emb} \Phi_{|D(r)|}\colon\bar B\hookrightarrow\PP^{N(r)}=\PP E\quad\mbox{where}\quad E=H^0(\bar B,\cO_{\bar B}(D(r)))^{\vee}\,.\ee
A general hyperplane section 
cuts out  on $\bar B$ a reduced effective divisor, say, $A'$ such that $\supp A'\subset B\setminus \supp z^*(0)$. Letting $A=A'+D_0\sim rD_\infty$ there exists a rational function $f$ on $\bar B$ with $\div f=A-rD_\infty\in\Div \bar B$. Then $A=\div\,(f|_B)$ is a principal effective divisor on $B$. It is easily seen that the $(A,\overline{-1})$-stretching over $X$ satisfies \eqref{eq:rectifying-heights}.

(b) Under the assumptions of (b) the divisors $D_0, D_\infty$, and $D(r)$ are $\mu_d$-invariant and the embedding \eqref{eq:emb} is equivariant with respect to a linear action of $\mu_d$ on $E$. Consider the character decomposition $E=\bigoplus_{\chi\in\mu_d^\vee} E_\chi$. Choose $\chi\in\mu_d^\vee$ such that the equivariant projection $E\to E_\chi$ yields a non-constant map $\bar B\to \PP E_\chi$. The reduced $\mu_d$-invariant effective divisor $A'$ cut out  on $\bar B$ by a general hyperplane section in $\PP E_\chi$  still satisfies $\supp A'\subset B\setminus \supp z^*(0)$. Define $f|_B\in\cO_B(B)$ as before.  The divisor $A=\div\,(f|_B)\in\Div (B)$ being $\mu_d$-invariant,  $f$ is a quasi-invariant of weight, say, $k\in\{0,\ldots,d-1\}$. Replacing $f$ by the $\mu_d$-invariant $\tilde f=z^{d-k}f$ one obtains a $\mu_d$-invariant principal divisor $$\tilde A=\div (\tilde f|_B)=A'+D'_0 +(d-k)\div z^*(0)\,.$$ Then $\tilde \cD(\pi)=(\tilde A.\cD(\pi))_{\overline{-1}}$ verifies \eqref{eq:rectifying-heights} with $\tilde \cD(\pi)$ instead of $\cD(\pi)$. 
\eproof

\brem\la{rem: extending-support} 
Extending a trivializing sequence \eqref{eq: seq-aff-modif-final} likewise this is done in the proof one introduces new special fibers. The resulting graph divisor $\tilde \cD(\pi)$ adopts a certain number of
new fiber graphs $\tilde\Gamma_{p_j}(\pi)=[[0]]p_j$ of hight 1 supported off $\supp (\div \,z)=\{b_1,\ldots,b_n\}$. So, the former marking $z$ cannot serve any longer as a marking.  
\erem

By virtue of the following lemma and the subsequent remarks, in certain cases a combinatorial stretching admits a simple geometric realization.

\blem\la{lem: geometrization} 
Let $\pi\colon X\to B$ be a marked GDF surface with a marking $z\in\cO_B(B)\setminus\{0\}$ and a graph divisor $\cD(\pi)=\sum_{i=1}^n \Gamma_i b_i$, and  let  $\fF$ be the set of leaves of $\cD(\pi)$. Given a subset $\fF_0$ 
of $\fF$ and an integer $s\gg 1$ choose a function $\tilde u\in\cO_X(X)$ such that
\begin{itemize}\item[(i)] $\tilde u\equiv u_F\mod z^s$ near any fiber component $F$  with $\bar F\in\fF_0$ and \item[(ii)] $\tilde u\equiv z^s\mod z^{s+1}$ near any $F$ with $\bar F\in\fF\setminus\fF_0$.
\end{itemize} Choose also a function $f\in\cO_B(B)$ such that $\supp (\div f)\subset \supp(\div z)$. 
Letting $$X'=\Spec\cO_X(X)[\tilde u/f]$$ consider the morphisms $\pi'\colon X'\to B$ and $\sigma\colon X'\to_B X$  associated to the natural embeddings $\cO_B(B)\subset\cO_X(X)\subset\cO_{X'}(X')$. 
Then $\pi'\colon X'\to B$ is a GDF surface with the graph divisor $\cD(\pi')=\sum_{i=1}^n\Gamma_i' b_i $ obtained from $\cD(\pi)$ by attaching to each leaf $\bar F\in \fF_0$ with $\pi(F)=b_i$ a chain $L_i=[[-2,\ldots,-2,-1]]$ of length $a_i:=\ord_{b_i}(f)$. Furthermore, $\sigma$ induces a morphism of graphs  $\Gamma_i'\to\Gamma_i$ contracting  the chains $L_i$, $i=1,\ldots,n$.
\elem

\bproof 
For a component $F$ of $\pi^{-1}(b_i)$ one has 
$(f\circ\pi)|_{U_F}\equiv c_iz^{a_i}\mod\, z^{a_i+1}$  where $c_i\neq 0$. Due to (ii) if $\bar F\in\fF\setminus\fF_{0}$ then the morphism $\sigma\colon X'\to X$ is an isomorphism over $U_F$. 

If $\bar F\in\fF_{0}$ then $\sigma\colon X'\to X$ restricted to $\sigma^{-1}(U_F)$ is an $a_i$-iterated fibered modification with a reduced divisor $F$ and center at the maximal ideal $(z,u_F)$ and its infinitesimally near points. 
Indeed, let $P_0=\{z=0,\,u_F=0\}$ be the origin of the local coordinate system $(z,u_F)$ in $U_F$ near $F$. The affine modification 
of $X$ along $z=0$ with center $P_0$ amounts in $U_F$ to the extension $\cO_{U_F}(U_F)\hookrightarrow \cO_{U_F}(U_F)[u_F/z]$. Iterating one obtains the extension $\cO_{U_F}(U_F)\hookrightarrow \cO_{U_F}(U_F)[u_F/z^{a_i}]$.
Due to (i) the latter extension coincides with $\cO_{U_F}(U_F)[{\tilde u}/f]$. Hence the fiber tree $\Gamma_i'$ is obtained from $\Gamma_i$ by joining the left end of the chain $L_i=[[-2,\ldots,-2,-1]]$ of length $a_i$ to the leaf $\bar F$ of $\Gamma_i$.  By contrast, for $\bar F\in\fF\setminus\fF_0$ one has $\cO_{U_F}(U_F)=\cO_{U_F}(U_F)[\tilde u/z]$.
\eproof

\brems\la{rem: realization}
1. If in Lemma \ref{lem: geometrization} one has $\fF_0=\fF$ and any leaf $\bar F\in\fF$ with $\pi(F)=b_i$ is on the same level $m_i:=\h(\Gamma_i)$, $i=1,\ldots,n$ then $\sigma\colon X'\to X$ is the principal top-level $(A,\bar m)$-stretching where $A=\div f$ and $\bar m=(m_1,\ldots,m_n)$. One can choose a new marking $z\in\cO_B(B)$ for $X$ such that $\supp(\div f)\subset\supp (\div z)$.

2.  Let $\pi\colon X\to B$ in Lemma \ref{lem: geometrization} be a marked GDF $\mu_d$-surface, and let as before $X'=\Spec\cO_X(X)[\tilde u/f]$. Assume that the functions $\tilde u$ and $f$ are $\mu_d$-quasi-invariants. Then $\pi'\colon X'\to B$ is a marked GDF $\mu_d$-surface and $\sigma\colon X'\to X$ is $\mu_d$-equivariant. 

3.
Let $A=\div f$ where $f\in\cO_B(B)$ with $\supp(\div f)\subset\supp (\div z)$, and let $s\gg 1$. Consider an affine modification $\sigma\colon X'\to X$ over $B$ such that its effect on the graph divisors is the same as in Lemma \ref{lem: geometrization}. Then one has $\cO_{X'}(X')=\cO_X(X)[\tilde u/f]$ for a function $\tilde u\in\cO_X(X)$ such that  
\begin{itemize}
\item[{\rm (i$'$)}] $\tilde u\equiv c_F\cdot (u_F-p_{F}(z)) \mod\, z^s$ in $U_F$ for $\bar F\in\fF_{0}$, $\pi(F)=b_i$ where $c_F=(f/z^{a_i})|_F$ and $p_F\in \mathbb{k}[z]$ is a polynomial of degree $\le a_i-1$ whose coefficients encode the sequence of infinitely near centers of blowups over $F$; 
\item[{\rm (ii$'$)}] $\tilde u\equiv z^s \mod\, z^{s+1}$ in $U_F$ for $F\in\fF\setminus\fF_{0}$. 
\end{itemize}

4. Let $\pi'\colon X'\to B$ and $\pi''\colon X''\to B$ be two marked GDF $\mu_d$-surfaces  equivariantly dominating $X$ over $B$ with the same effect on graph divisors. Then by Theorem \ref{thm: GDF-cancellation-fixed-graph} there is a $\mu_d$-equivariant isomorphism of cylinders $\cX'(k)\cong_{\mu_d,B}\cX''(k)$ $\forall k\in\ZZ$. Thus, different geometric realizations of the same combinatorial stretching give rise to isomorphic cylinders. 
\erems
 
\bcor\la{cor: equiv-stretching} 
Under the assumptions of Lemma {\rm \ref{lem: geometrization}} let $\pi\colon X\to B$  be a marked GDF $\mu_d$-surface, and let $\fF_0\subset\fF$ be a $\mu_d$-invariant set of top level leaves of $\cD(\pi)$.  Then there exists a marked GDF $\mu_d$-surface $\pi'\colon X'\to B$ and an equivariant affine modification $\sigma\colon X'\to X$ over $B$ which amounts to  attaching a chain $L_F=[[-2,\ldots,-2,-1]]$ of the same length $a\ge 1$ to every leaf $\bar F\in\fF_0$ of the graph divisor $\cD(\pi)$. 
\ecor

\bproof
It suffices to apply Remark \ref{rem: realization}.2 choosing a $\mu_d$-quasi-invariant function $\tilde u\in\cO_X(X)$ of weight $-m$ as in Corollary \ref{cor: q-inv} and letting $f=z^a$.
\eproof

Due to the next proposition, a principal top-level stretching does not affect the cylinder up to an isomorphism over $B$.  

\bprop\label{prop: isomorphism} Let $\pi\colon X\to B$ be a marked GDF surface with a marking $z\in\cO_B(B)$ where $\div z=b_1+\ldots+b_n$. 
Suppose that \begin{itemize}\item[$(\alpha)$]  for $i=1,\ldots,n$ the leaves of $\Gamma_{b_i}(\pi)$ are on the same level $m_i$.
\end{itemize} 
Let $\sigma\colon X'\to X$ be a principal top-level $(A,\bar m)$-stretching as in Definition {\rm \ref{def: u-modification}} where 
$A=\div f$ for some $f\in\cO_B(B)\setminus\{0\}$. 
 Then for any
$s\gg 1$ 
there is an isomorphism of cylinders
$\varphi\colon \cX\stackrel{\cong_B}{\longrightarrow} \cX'$ such that
for every pair of special fiber components $F$ in $X$ and $F'$ in $X'$ with
$\phi(F\times\A^1)=F'\times\A^1$
one has\footnote{We let  $\varphi_*(f,g,h)=(\varphi_*f,\varphi_*g,\varphi_*h)$ where $\varphi_*=(\varphi^{-1})^*$.}
\be\label{eq: switch}
(\varphi|_{U_F\times\A^1})_*\colon (z,u_F,v)\mapsto (z,u_{F'},v') \mod\, z^{s}\,
\ee
in suitable  natural coordinates $(z,u_F,v)$ and $(z,u_{F'},v')$    in the standard affine charts $U_F\times\A^1\subset \cX$ and $U_{F'}\times\A^1\subset \cX'$, respectively.  
\eprop

\bproof (a) Choosing a new marking $z\in\cO_B(B)$ for $X$ one may suppose that $\supp(\div f)\subset\supp (\div z)$ and ($\alpha$) still holds. Performing a suitable principal $(A,\overline{-1})$-stretching as in  Lemma \ref{lem: equal-hights}(a) one may assume that 

\smallskip

\begin{itemize}\item[$(\alpha_0)$] \emph{ for all $i=1,\ldots,n$ the leaves of $\Gamma_{b_i}(\pi)$ are on the same level $m=\h(\cD(\pi))$. }
\end{itemize}

\smallskip

 Consider the Asanuma modification of the second kind  $\kappa\colon \cX''\to \cX$ associated with $f$ (see Definition \ref{defi: asanuma}), that is, $$\cO_{\cX}(\cX)\subset\cO_{\cX''}(\cX'')=\cO_{X}(X)[v/f]\,.$$ 
  Lemma \ref{lem: As-2d-kind}(a) provides an isomorphism
  \be\la{eq:beta} \beta\colon\cX\stackrel{\cong_B}{\longrightarrow} \cX''\quad\mbox{where}\quad \beta_*\colon (z,u_F,v)\mapsto (z,u_{F''},v'')\,\ee 
in suitable natural local coordinates. We claim that there is an isomorphism $\cX'\cong_B \cX''$.
Indeed, due to Remark \ref{rem: realization}.4 and condition ($\alpha_0$) one may suppose that $$\cO_{\cX'}(\cX')\cong_{\mathbb{k}[z]}\cO_{\cX}(\cX)[\tilde u/f]$$ where $\tilde u\in\cO_X(X)$ verifies conditions (i) and (ii) of Lemma \ref{lem: geometrization} with $\fF_0$ being the set of top level fiber components in $X$.
Thus, it suffices to establish an isomorphism 
\be\label{eq:iso-mm} \cO_{\cX}(\cX)[\tilde u/f] \cong_{\cO_B(B)}\cO_{\cX}(\cX)[v/f]\,.\ee 

By Lemma \ref{lem: tau} there exists $\tau\in\SAut_B\, \cX$ such that
\be\label{eq:iso-mm0} (\tau |_{U_F\times\A^1})_*\colon (z,u_F,v)\mapsto (z,v,-u_F) \mod\, z^s\ee for any top level fiber component $F$ in $X$, see (\ref{eq: tau-11}). By (i) of Lemma \ref{lem: geometrization}, $\tilde u\equiv u_F \mod\, z^s$ in $U_F$. Due to condition ($\alpha_0$) all the components of $f^*(0)$ in $\cX$ are on the top level. 
Therefore,  $\tau_*$ transforms the ideal $I=(v,f)\subset \cO_{\cX}(\cX)$ to $I'=(\tilde u,f)\subset \cO_{\cX}(\cX)$ preserving the principal ideal $(f)$.  By Lemma \ref{lem: preserving-isomorphisms}, 
$\tau$  induces an isomorphism $\tilde\tau\colon\cX'\stackrel{\cong_B}{\longrightarrow}\cX''$ which gives \eqref{eq:iso-mm}, and so, proves our claim.

Let $a$ be the maximal order of zeros of $f$.  
Letting $\tilde u'=\tilde u/f$ and $v''=v/f$  one obtains 
 \be\la{eq:iso-mm1} \tau_*\colon (z,\tilde u',v'')\mapsto (z,v'',-\tilde u') \mod\, z^{s-a}\,.\ee

Let $\cF''=\tilde\tau(\cF')\subset\cX''$.
Consider the standard affine charts $U_{F'}\times\A^1$ in $\cX'$ and $U_{F''}\times\A^1$ in $\cX''$ with local coordinates $(z,u_{F'},v')$ and $(z,u_{F''},v'')$, respectively, where 
\be\la{eq:loc-coord} u_{F''}=u_F,\,\,\, v'=v,  \quad\mbox{and}\quad u_{F'}=u_F/f\equiv \tilde u'\mod\, z^{s-a}\,.\ee
From \eqref{eq:iso-mm0} and \eqref{eq:iso-mm1} one can deduce
$$\tilde\tau\colon\cX'\stackrel{\cong_B}{\longrightarrow}\cX'',\qquad\tilde\tau_*\colon (z,u_{F'},v')\mapsto (z,v'',-u_{F''}) \mod\, z^{s-a}\,.$$ Then by \eqref{eq:beta} and \eqref{eq:loc-coord} one gets
$$\beta^{-1}\circ\tilde\tau\colon \cX'\stackrel{\cong_B}{\longrightarrow} \cX,\quad (z,u_{F'},v')\mapsto (z,v,-u_F) \mod\, z^{s-a}\,,$$  
and so,
$$\tau^{-1}\circ \beta^{-1}\circ\tilde\tau\colon \cX'\stackrel{\cong_B}{\longrightarrow} \cX,\quad (z,u_{F'},v')\mapsto (z,u_F,v) \mod\, z^{s-a}\,.$$
Thus, the isomorphism $\phi:=\tilde\tau^{-1}\circ\beta\circ\tau\colon\cX\stackrel{\cong_B}{\longrightarrow} \cX'$ verifies \eqref{eq: switch} with $s$
replaced  by $s-a$.
\eproof

As an illustration, we apply Proposition \ref{prop: isomorphism} to the Danielewski examples. 

\bexa[\emph{Danielewski surfaces revisited}]\label{exa: u-modif-Dan} Recall that the $n$th Danielewski surface $X_n$ is given in $\A^3$ with coordinates $(z,u,t_n)$ by equation $z^nt_n-u^2+1=0$, see Example \ref{exa: dani}. The function $t_n=t_0/z^n$ yields a natural affine coordinate on each component of the special fiber $z=0$. 
The morphism $\rho_n\colon X_n\to X_{n-1}$ is given by $$(z,u,t_n)\mapsto (z,u,t_{n-1}=zt_n)\,.$$  For any $i\in\{1,\ldots,n-1\}$ the morphism $\rho_{i+1}\circ\ldots\circ\rho_n\colon X_n\to X_{i}$ is a principal top level stretching. Proposition \ref{prop: isomorphism} and Corollary \ref{cor: non-isomorphism} below provide an alternative proof of the Danielewski--Fieseler Theorem (\cite{Da}, \cite{Fi}) which says that the cylinders $X_n\times\A^1$, $n\in\N$, are all isomorphic whereas the surfaces $X_n$ and $X_m$ are not if $n\neq m$.  
\eexa 

\brem\la{rem:added} If, by chance, a GDF surface $\pi\colon X\to B$ verifies condition ($\alpha_0$) with respect to some trivializing sequence \eqref{eq: seq-aff-modif-final} and some marking $z\in\cO_B(B)$ then $X$ admits a free $\mathbb{G}_a$-action along the fibers of $\pi$. In this case the conclusion of Proposition \ref{prop: isomorphism} can be derived by applying the Danielewski trick, see Section \ref{ss:DF-construction}. 
\erem

Next we provide an equivariant version of Proposition \ref{prop: isomorphism} with a similar proof. To avoid a repetition we omit certain details of the proof.

\bprop\la{prop:equiv-stretching}
Let $\pi\colon X\to B$ be a marked GDF $\mu_d$-surface with a marking $z\in\cO_B(B)\setminus\{0\}$ of weight $1$ verifying condition $(\alpha)$, see Proposition {\rm\ref{prop: isomorphism}}.
Then for any $k\in\ZZ$ and $l\in\NN$ there exist a $\mu_d$-equivariant principal  top level $(A,\bar m)$-stretching $\sigma\colon X'\to X$ where $A=\div\,z^{ld}$
and a $\mu_d$-equivariant isomorphism of cylinders $\varphi\colon \cX(k)\stackrel{\cong_{\mu_d,B}}{\longrightarrow} \cX'(k)$ verifying \eqref{eq: switch} for any component $F$ of $z^*(0)$ in $X$.
\eprop

\bproof
By virtue of Lemma \ref{lem: equal-hights}(b) one may suppose that all the components of $z^*(0)$ in $X$ are on the same  top level $m$, that is, the graph divisor $\cD(\pi)$ verifies condition ($\alpha_0$), see the proof of Proposition \ref{prop: isomorphism}.
Let $\tilde u\in\cO_X(X)$ be a $\mu_d$-quasi-invariant of weight $-m$ satisfying 
conditions {\rm (i)} and {\rm (ii)} of Corollary {\rm \ref{cor: q-inv}} with $\fF_0$ being the set of (top level)  components of $z^*(0)$ in $X$. 

The iterated Asanuma modification  of the second kind $\cX''(k)\to\cX(k)$, $v''\mapsto v=z^{ld}v''$, is $\mu_d$-equivariant along with the isomorphism 
$$\beta_k\colon\cX(k)\stackrel{\cong_{\mu_d,B}}{\longrightarrow}\cX''(k),\qquad v\mapsto v''\,.$$

By Lemma \ref{lem: tau-1}  one can find a $\mu_d$-equivariant automorphism $\tau\in\SAut_B \cX(-m)$ which interchanges modulo $z^s$, up to a sign, the functions $\tilde u$ and $v$ of the same weight $-m$ and leaves $z$ invariant. By Lemma \ref{lem:  lift}, $\tau$ admits a lift  to a $\mu_d$-equivariant isomorphism 
$$\tilde\tau\colon \cX'(-m)\stackrel{\cong_{\mu_d,B}}{\longrightarrow}\cX''(-m)\,$$ verifying \eqref{eq:iso-mm1} for $f=z^{ld}$. Now the composition $\phi_{-m}=\tilde\tau^{-1}\circ\beta_{-m}\circ\tau$ yields an isomorphism $\cX(-m)\stackrel{\cong_{\mu_d,B}}{\longrightarrow} \cX'(-m)$ verifying \eqref{eq: switch}.  By Lemma 5.4(c) one may replace the weight $-m$ by a given weight $k$. This does not affect \eqref{eq: switch} up to replacing $s$ by $s-d$. 
\eproof

We use below the following auxiliary fact.

\blem\label{lem: vertices} Let  $X$ be a normal affine surface that admits  an $\A^1$-fibration $X\to C$ over a smooth affine curve $C$, and let $\bar X\to\bar C$ be a pseudominimal completion of $X\to C$ with extended graph $\Gamma_{\rm ext}$. Then the number $v(\Gamma_{\rm ext})$ of vertices of $\Gamma_{\rm ext}$ does not depend on the choice of an $\A^1$-fibration on $X$ over an affine base. So, $v(X):=v(\Gamma_{\rm ext})$ is an invariant of $X$.  \elem

\bproof Recall (see \cite[Def.\ 2.16]{FKZ}) that every feather component $F$ of the extended divisor $D_{\rm ext}$ is born  under a blowup at a smooth point of the boundary divisor $D=\bar X\setminus X$. The unique component $D_i$ of $D$ containing the center of this blowup is called the {\em mother component} of $F$. The {\em normalization procedure} as defined in \cite[Def.\ 3.2]{FKZ} replaces  $\Gamma_{\rm ext}$ by the {\em normalized extended graph} $\Gamma_{\rm ext, norm}$ such that any feather component $F$ in $\Gamma_{\rm ext}$ becomes an extremal $(-1)$-vertex in $\Gamma_{\rm ext, norm}$ attached at its mother component $D_i$. Under this procedure the total number of vertices remains the same: $v(\Gamma_{\rm ext, norm})=v(\Gamma_{\rm ext})$. Furthermore, these graphs are assumed to be {\em standard}; the standardization procedure does not affect the number of vertices either, see \cite[\S 1]{FKZ}. By \cite[Thm.\ 3.5]{FKZ} the standard normalized extended graph $\Gamma_{\rm ext, norm}$ of $X$ is unique, that is, it does not depend on the choice of an $\A^1$-fibration on $X$ over an affine base, unless $X$ is a Gizatullin surface. In the latter case the minimal dual graph $\Gamma$ of $D$ is linear and $\Gamma_{\rm ext, norm}$ is unique up to a
{\em reversion} $\Gamma_{\rm ext, norm}\rightsquigarrow \Gamma_{\rm ext, norm}^\vee$. However, the reversion neither changes the number of vertices in $\Gamma$ nor does it in $\Gamma_{\rm ext, norm}$. The latter is due to the {\em Matching Principle} (\cite[Thm.\ 3.11]{FKZ}). According to this principle there is  a bijection between the feather components of $\Gamma_{\rm ext, norm}$ and $ \Gamma_{\rm ext, norm}^\vee$ along with their mother components. In conclusion, $v(\Gamma_{\rm ext})=v(\Gamma_{\rm ext, norm})$ is an invariant of the surface $X$.
\eproof

\bcor\la{cor: non-isomorphism}
Let $\theta: X'\to X$ be a principal geometric top-level $(A,\bar m)$-stretching between two GDF surfaces $\pi'\colon X'\to B$ and $\pi\colon X\to B$ where $A=\div f$, $f\in\cO_X(X)\setminus\{0\}$. Suppose that $f(b)=0$ for some $b\in B$ such that the fiber $\pi^{-1}(b)$ is reducible. Then $X'\not\cong X$.
\ecor

\bproof Consider 
a trivializing completion $(\hat X,D)$ for $\pi\colon X\to B$ as in \ref{rem: extended-to-completions}.1. Suppose that its degenerate fibers are situated over the points $b_1,\ldots,b_n\in B$. This completion is pseudominimal if and only if for any $i\in\{1,\ldots,n\}$ the only $(-1)$-vertices of $\Gamma_{b_i}(\pi)$ are leaves, that is, the root $v_i$ of 
 the fiber tree $\Gamma_{b_i}(\pi)$ is not a $(-1)$-vertex, see Lemma \ref{lem: standard GDF-procedure}(c). If $v_i$ is a $(-1)$-vertex then $v_i$ is a tip of an extremal linear branch $\cB_i$ of $\Gamma_{b_i}(\pi)$. There is an alternative: either $\Gamma_{b_i}(\pi)=\cB_i$ is a chain $[[-1,-2,\ldots,-2,-1]]$, or $\cB_i$ is a branch $[[-1,-2,\ldots,-2]]$  at a branching vertex $v$  of $\Gamma_{b_i}(\pi)$ with weight $\le -3$. In the former case, $\Gamma_{b_i}(\pi)$ can be contracted to a single $(0)$-vertex. In the latter case, contracting $\cB_i$ in $\Gamma_{b_i}(\pi)$  one gets a pseudominimal tree $\Gamma_{b_i, {\rm min}}(\pi)$  with $v$ as  the root. Performing these contractions of the branches $\cB_i$ for all $i=1,\ldots,n$ yields a pseudominimal SNC completion $(\bar X_{\rm min}, D_{\rm min})$ of $X$. 
 
 Notice that \begin{itemize}\item $\cB_i\subset\Gamma_{b_i}(\pi)\subset\Gamma_{b_i}(\pi')$; \item
$\Gamma_{b_i}(\pi)$ is a chain if and only if the fiber $\pi^{-1}(b_i)$ is irreducible, if and only if $\Gamma_{b_i}(\pi')$ is. \end{itemize}
By our assumption the fiber $\pi^{-1}(b_i)$ is reducible for some $i\in\{1,\ldots,n\}$. Therefore, $\Gamma_{b_i}(\pi)$ is not a chain. Then the (eventual) contraction of $\cB_i$ in both $\Gamma_{b_i}(\pi)$ and $\Gamma_{b_i}(\pi')$ leads to two pseudominimal rooted trees $\Gamma_{b_i, {\rm min}}(\pi)$ and $\Gamma_{b_i,  {\rm min}}(\pi')$ where the number of vertices in $\Gamma_{b_i,  {\rm min}}(\pi')$ is larger by $a_i=\ord_{b_i}(f)>0$ than the one in $\Gamma_{b_i,  {\rm min}}(\pi)$. In turn, the simultaneous contractions in all the special fibers as described above lead to the
 pseudominimal completions $(\bar X_{\rm min},D_{\rm min})$  and $(\bar X'_{\rm min},D'_{\rm min})$ of  $\pi\colon X\to B$ and $\pi'\colon X'\to B$, respectively, such that the corresponding extended graphs $\Gamma_{\rm ext, min}$ and $\Gamma_{\rm ext, min}'$ have different number of vertices. 
By Lemma \ref{lem: vertices} this number is invariant upon isomorphisms of affine surfaces. Hence $X\not\cong X'$, as claimed. 
\eproof

\subsection{Non-cancellation for 
GDF surfaces}\label{ss: non-cancellation for 
GDF surfaces}

The main result of this section is the following theorem.

\bthm\label{thm: GDF 1-factors}  Let $\pi\colon X\to B$ be a GDF surface. Then $X$ is a Zariski $1$-factor if and only if $\pi\colon X\to B$ admits a line bundle structure. \ethm

\bproof The `if' part follows from Proposition \ref{prop: line-bundle}. As for the `only if' part, see the following version of Proposition \ref{prop: isomorphism} which avoids assumption ($\alpha$).\eproof

\bprop\la{prop: without-alpha} 
Suppose that a GDF surface $\pi\colon X\to B$  has a reducible fiber. Then there exists a sequence of pairwise non-isomorphic GDF surfaces $\pi_{Y_k}\colon Y_k\to B$ with cylinders isomorphic over $B$ to the one of $X$: $\cY_k\cong_B\cX$ $\forall k\in\NN$. 
\eprop

\bproof 
Fix a trivializing well ordered sequence \eqref{eq: seq-aff-modif-final} of affine modifications with $X=X_m$. Let $l\in\{1,\ldots,m\}$ be the minimal index such that
$\pi_l\colon X_l\to B$ in \eqref{eq: seq-aff-modif-final} has a reducible fiber, say, $\pi_l^{-1}(b_1)$ where $b_1\in B$. So, the graph divisor $\cD(\pi_{l-1})$ is a chain divisor. Hence $\pi_l\colon X_l\to  B$ verifies condition ($\alpha$) of Proposition \ref{prop: isomorphism}.

Consider a principal top-level $(A,\bar m)$-stretching  $\sigma\colon X_l'\to X_l$ where $A=\div z$. According to Proposition \ref{prop: isomorphism} for any $s\gg 1$ there exists an isomorphism of cylinders $\varphi_l\colon \cX_l\stackrel{\cong_B}{\longrightarrow} \cX_l'$ satisfying (\ref{eq: switch}). 
To any component $F\subset X_l$ of $\pi_l^{-1}(b_i)$ there corresponds a unique 
component $F'\subset X_1'$ of ${\pi'_1}^{-1}(b_i)$ such that $\sigma(F')=P_{F'}\in F$.  Inspecting the proof of Proposition \ref{prop: isomorphism} we see that $\varphi_l(\cF)=\cF'$ where $\cF=F\times\A^1\subset\cX_l$ and $\cF'=F'\times\A^1\subset\cX_l'$. Due to (\ref{eq: switch}) one has
$$\varphi_l(F\times\{0\})=F'\times\{0\}\,.$$
   We construct  by recursion a sequence of GDF surfaces 
\be\label{eq: new-sequence} X'_m \stackrel{\varrho'_m}{\longrightarrow} X'_{N-1} \stackrel{}{\longrightarrow}  
\ldots \stackrel{}{\longrightarrow}  X'_{l+1}\stackrel{\varrho'_{l+1}}{\longrightarrow}   X'_l\,\ee
such that
\begin{itemize}\item[(i)] $\cX_j'\cong_B\cX_j$, $j=l,\ldots,m$;
\item[(ii)] 
$X_j'\not\cong X_j$  for any $j=l,\ldots,m$.
\end{itemize}

Let $\Sigma=\bigcup_{F} \Sigma_{F}$ be the center of the affine modification $\rho_{l+1}\colon X_{l+1}\to X_{l}$ from  (\ref{eq: seq-aff-modif-final}) where $\Sigma_F=\Sigma\cap F$. Set 
$$ \fF_0=\{F\,|\,\Sigma_{F}\neq\emptyset\} ,\quad\quad \fF_0'=\{F'\,|\,F\in\fF_0\}\,,$$ 
 $$\Sigma_{F'}\times\{0\}=
\varphi_l(\Sigma_{F}\times\{0\})\subset F'\times\{0\}\,,$$  and
$$
\Sigma'=\bigcup_{F'\in\fF_0'} \Sigma_{F'}
\subset X_l'\,.$$ 
Consider the fibered modification $\rho_{l+1}'\colon X_{l+1}'\to X_l'$  along the reduced divisor $z^{*}(0)$  with the reduced center $\Sigma'\subset X_l'$. 
Let $\tilde\rho_{l+1}\colon\cX_{l+1}\to\cX_l$ and $\tilde\rho_{l+1}'\colon\cX_{l+1}'\to\cX_l'$ be the Asanuma modifications of the first kind which correspond to $\rho_{l+1}$ and $\rho_{l+1}'$, respectively, see Lemma \ref{lem: as-trick}(a).
By construction,  $\varphi_l\colon\cX_l\stackrel{\cong_B}{\longrightarrow}\cX_l'$ 
sends the center and the divisor of $\tilde\rho_{l+1}$ to the center and the divisor of $\tilde\rho_{l+1}'$.  
By Lemma \ref{lem: preserving-isomorphisms}, $\varphi_l$ lifts to an isomorphism of  cylinders $\varphi_{l+1}\colon\cX_{l+1}\stackrel{\cong_B}{\longrightarrow}\cX_{l+1}'$. 
Likewise in \eqref{eq: congru} of Lemma \ref{lem: lift}, $\varphi_{l+1}$ verifies (\ref{eq: switch}) with $s$ replaced by $s-1$. Now one can apply the same argument to the isomorphism $\varphi_{l+1}\colon\cX_{l+1}\stackrel{}{\longrightarrow}\cX_{l+1}'$ instead of $\varphi_l\colon\cX_l\stackrel{}{\longrightarrow}\cX_l'$. By recursion, we arrive at a sequence (\ref{eq: new-sequence}) such that $\cX_i\cong_B\cX_i'$ for all $i=1,\ldots,m$. 

Let $Y_1=X_m' $ and $\pi_{Y_1}=\pi_{X_m'}$. To finish the proof it suffices to repeat the same construction with $z^k$ instead of $z$. This leads to a sequence of GDF surfaces 
$Y_k$, $k=1,2,\ldots$. By Lemma \ref{lem: vertices} these surfaces are pairwise non-isomorphic. The proof goes similarly as the one of Corollary \ref{cor: non-isomorphism}.
\eproof

The following is an equivariant version of Proposition \ref{prop: without-alpha}. 

\bprop\label{prop: GDF mu-d-1-factors} Let $\pi\colon X\to B$ be a marked GDF $\mu_d$-surface which has a reducible fiber.  Then there exists a sequence of pairwise non-isomorphic marked GDF $\mu_d$-surfaces $X^{(kd)}$, $k\in \NN$, whose  cylinders are $\mu_d$-equivariantly isomorphic over $B$: $$ \cX^{(kd)}\cong_{\mu_d,B}\cX\quad\forall k\in \NN\,.$$
\eprop

\bproof Consider a sequence
\eqref{eq: seq-aff-modif-final} of $\mu_d$-equivariant morphisms, and let $\pi_l\colon X_l\to B$ be the first member of \eqref{eq: seq-aff-modif-final} which has a reducible fiber. Proceeding as in the proof of Proposition  \ref{prop: without-alpha} we let  $X_l^{(kd)}$ be the GDF $\mu_d$-surface obtained from $X_l$ via a principal equivariant top-level
$(A_k,\bar m)$-stretching where $A_k=\div z^{kd}$,  $k\in\NN$. 

According to Proposition \ref{prop:equiv-stretching} for any $s\gg 1$ there is a $\mu_d$-equivariant isomorphism of cylinders $\phi_l\colon \cX_l(-l)\stackrel{\cong_{\mu_d,B}}{\longrightarrow} \cX_l^{(kd)}(-l)$  satisfying \eqref{eq: switch}. 
Repeating for any $k\in\N$ the construction from the proof of Proposition \ref{prop: without-alpha} in a $\mu_d$-equivariant fashion one arrives at a sequence of marked GDF $\mu_d$-surfaces $X^{(kd)}=X_m^{(kd)}$ with $\mu_d$-equivariantly isomorphic cylinders $\cX^{(kd)}(-m)\cong_{\mu_d,B} \cX(-m)$ where the isomorphisms satisfy \eqref{eq: switch}.  
By Lemma \ref{lem: As-2d-kind}(c) there are $\mu_d$-equivariant isomorphisms  $ \cX^{(kd)}(0)\cong_{\mu_d,B}\cX(0)\,.$ 

Arguing as in the proof of  Corollary \ref{cor: non-isomorphism} one can see that the number of vertices of the corresponding  pseudominimal extended graphs $\Gamma_{\rm ext, min}^{(kd)}$  strictly increases with $k$. By Lemma \ref{lem: vertices} the GDF surfaces $X^{(kd)}$, $k=0,1,\ldots$, are pairwise non-isomorphic. 
\eproof

\subsection{Extended graphs of Gizatullin surfaces}
The covering trick can be  extended to completions as follows.

\bsit[\emph{Covering trick for a completion}]\label{nota: covering-trick-completion} Let $\pi_Y\colon Y\to C$ be an $\A^1$-fibration over an affine curve $C$, and let $\bar\pi_Y\colon\bar Y_{\rm resolved}\to\bar C$ be a pseudominimal resolved completion of  $\pi_Y\colon Y\to C$ with extended graph $\Gamma_{\rm ext}$. Contracting the exceptional divisor $E\subset\bar Y_{\rm resolved}$ of the minimal resolution of singularities of $Y$ yields a birational morphism $\sigma\colon\bar Y_{\rm resolved}\to\bar Y$ where $\bar Y\to\bar C$ is  a completion  of $\pi_Y\colon Y\to C$ with a simple normal crossing boundary divisor. Extending a branched covering $B\to C$ as in  \ref{sit: construction} to the smooth completions $\bar B \to \bar C$ consider the normalizations  of the cross-products $\bar Y_{\rm resolved}\times_{\bar C} \bar B$ and $\bar Y\times_{\bar C} \bar B$, the respective minimal desingularizations $\hat X_{\rm resolved}\to (\bar Y_{\rm resolved}\times_{\bar C} \bar B)_{\rm norm}$ and $\hat X\to(\bar Y\times_{\bar C} \bar B)_{\rm norm}$, and the induced $\PP^1$-fibrations $\hat X_{\rm resolved}\to\bar B$ and $\hat X\to \bar B$. Recall that the branched covering construction applied to $\pi_Y\colon Y\to C$ gives a GDF surface $\pi_X\colon X\to B$ as in Definition \ref{sit: construction}. The  surface $X$ is smooth, see Lemma \ref{lem: standard GDF-procedure}(b). Hence $\hat X\to\bar B$ is
 a completion  of $X\to B$ dominated by $\hat X_{\rm resolved}\to\bar B$. The induced morphism $\hat \sigma\colon \hat X_{\rm resolved}\to\hat X$ contracts the total transform of the exceptional divisor $E$ of $\sigma\colon\bar Y_{\rm resolved}\to\bar Y$.
\esit

\bsit\label{sit: Gizatullin} Recall (see e.g., \cite{FKZ} and Section \ref{ss:ML}) that a Gizatullin surface $X$ is a normal affine surface of class $(\ML_0)$. Such a surface $X$ admits two different $\A^1$-fibrations over the affine line $\A^1$; in particular, $X\not\cong\A^1_*\times\A^1$. For any $\A^1$-fibration  $\pi\colon X\to C$ over a smooth affine curve $C$ one has $C\cong\A^1$ and $\pi$ has at most one degenerate fiber. One may assume that this is the fiber $\pi^{-1}(0)$.
\esit

In the proof of Theorem \ref{thm: Zariski-1-factors} we use the following fact. 

\blem\label{lem: Gizatullin} Let  $X$ be a Gizatullin surface. Then the following hold.
\begin{itemize}
\item[(a)] Let $\Omega(X)$ stand for the set   of isomorphism classes of the pseudominimal extended graphs $\Gamma_{\rm ext}$  of all possible $\A^1$-fibrations $\pi\colon X\to\A^1$ with the special fiber $\pi^{-1}(0)$. Then $\Omega(X)$ is finite. Hence there exists $d\in\N$ such that the multiplicities of the fiber  components of $\pi^*(0)$ in any such fibration divide $d$.
\item[(b)] Given an $\A^1$-fibration $X\to \A^1=C$ consider the GDF surface $\tilde X\to B$ obtained via the cyclic base change $\A^1=B\to C$, $z\mapsto z^d$ with $d$ as in {\rm (a)} and a subsequent normalization. Let $\tilde \Gamma_{\rm ext}$ be  the extended graph of a pseudominimal completion of $\tilde X$. Then the set $\tilde\Omega(X, d)$ of isomorphism classes of the graphs $\tilde \Gamma_{\rm ext}$ for all possible $\A^1$-fibrations $X\to \A^1=C$ is finite.
\end{itemize}
\elem

\bproof  (a) By Lemma \ref{lem:  vertices}
the graphs $\Gamma_{\rm ext}$ in $\Omega(X)$ have all the same number $v(X)$ of vertices. Notice that the number of non-isomorphic graphs on a given set of vertices is finite. Furthermore, given an $\A^1$-fibration $\pi\colon X\to \A^1$ along with a pseudominimal resolved completion $\bar\pi\colon\bar X\to\bar \PP^1$ the multiplicities of the fiber components of $\bar\pi^{-1}(0)$  can be deduced in a combinatorial way from the associated extended graph $\Gamma_{\rm ext}$. Hence there is $d\in\N$ divisible by all these multiplicities for all possible $\A^1$-fibrations $\pi\colon X\to \A^1$. 

To show (b) it suffices  to restrict to the $\A^1$-fibrations on $X$ with a fixed  pseudominimal extended graph $\Gamma_{\rm ext}$. Let $\pi\colon X\to C=\A^1$ be such an $\A^1$-fibration. Recall that $B\cong C\cong\A^1$,  the only possible  degenerate fiber of $\pi$ is $\pi^{-1}(0)$, and the base change  $B\to C$ is $z\mapsto z^d$. 

By \ref{nota: covering-trick-completion}
the extended graph $\hat \Gamma_{\rm ext}$ is dominated by $\hat \Gamma_{\rm ext,resolved}$. 
In turn, the pseudominimal extended graph $\tilde\Gamma_{\rm ext}$
in $\tilde\Omega(X,d)$ is dominated by $\hat \Gamma_{\rm ext}$ and also by $\hat \Gamma_{\rm ext,resolved}$.
This yields an upper bound for the number of vertices $v(\tilde\Gamma_{\rm ext})\le v(\hat \Gamma_{\rm ext,resolved})$. 
We claim that $v(\hat \Gamma_{\rm ext,resolved})$ is bounded above by a function depending only on $d$ and on $\Gamma_{\rm ext}$, and so, only on $X$, as desired.  

To show the claim notice that for any vertex of  $\Gamma_{\rm ext}$ there is  at most $d$ vertices of $\hat \Gamma_{\rm ext, resolved}$ such that the corresponding curves in $\hat X_{\rm resolved}$ dominate the one in $\bar X_{\rm resolved}$. Hence it suffices to find an upper bound on the number of the remaining vertices of $\hat \Gamma_{\rm ext, resolved}$ which correspond to the curves in $\hat X_{\rm resolved}$ contracted in $\bar X_{\rm resolved}$. 

Let $E'$ and $E''$ be two fiber components of the extended divisor $D_{\rm ext}$ with respective multiplicities $m'$ and $m''$ that meet
 in $\bar X$. Choose local coordinates $(x,y)$ in $\bar X$ centered at the intersection point  $E'\cap E''=\{p\}$ with $E'$ and $E''$ as the axes. Then the germ of the cross-product $\bar X\times_{\bar C} \bar B$ near $p$ is given by equation $z^d-x^{m'}y^{m''}=0$. Likewise, if $E''=S$ then the germ of $\bar X\times_{\bar C} \bar B$ near $p$ is given by equation
$z^d=x^{m'}$. Normalizing such a surface germ produces, in both cases,  at most $d$ cyclic quotient singularities of type uniquely determined by $d,m'$, and $m''$. The resolution graphs of these singular points are the Hirzebruch-Jung strings uniquely determined by $d$ and $\Gamma_{\rm ext}$. It follows that the number of vertices in the total preimage in $\hat \Gamma_{\rm ext, resolved}$ of the edge $[E,E']$ of  $\Gamma_{\rm ext}$ is bounded above in terms of $d$ and $\Gamma_{\rm ext}$.  Finally, $v(\hat \Gamma_{\rm ext, resolved})$ is bounded above by a function of $d$ and  $\Gamma_{\rm ext}$, as claimed. 
\eproof

\brem\label{rem: finite-number-components}
It follows that a Gizatullin surface $X$ admits at most a finite number of maximal families of pairwise non-equivalent $\A^1$-fibrations  $X\to\A^1$. 

There is a remarkable sequence  $(X_n)_{n\in\N}$ of Gizatullin surfaces called the {\em Danilov-Gizatullin surfaces}.  Given $n\in\NN$ there exists a deformation family $\cF_n\to\cS_n$  of $\A^1$-fibrations  $X_n\to\A^1$ which are pairwise non-equivalent modulo the $(\Aut X_n)$-action where $\dim\cS_n$ strictly  grows with $n$ (\cite[Thms.\ 1.0.1, 1.0.5, and Ex.\ 6.3.21]{FKZ0}). 
\erem

\subsection{Zariski 1-factors and affine $\A^1$-fibered surfaces}
The following is the main result of Section \ref{sec:final}. 

\bthm\label{thm: Zariski-1-factors} Let $\pi\colon X\to C$ be a normal $\A^1$-fibered affine surface over a smooth affine curve $C$. Then $X$ is a Zariski $1$-factor if and only if $\pi\colon X\to C$ admits a structure of a parabolic $\mathbb{G}_m$-surface. \ethm 

The ``if'' part follows from Theorem \ref{thm: Gm-surfaces-are-zar-factors}.
The ``only if'' part is proven in the next proposition.

\bprop\label{prop: only if} Let $\pi\colon X\to C$ be an $\A^1$-fibration on a normal affine surface $X$ over a smooth affine curve $C$. If $X$ is a Zariski $1$-factor then $\pi\colon X\to C$ admits a structure of a parabolic $\mathbb{G}_m$-surface.\eprop

\bproof  Assume that $X\to C$ does not admit a structure of a parabolic $\mathbb{G}_m$-surface.
We are going to construct an infinite sequence of  normal affine surfaces $X^{(nd)}$ non-isomorphic to $X$ such that the cylinders $\cX^{(nd)}$ and $\cX$ are isomorphic, thus showing that $X$ cannot be a Zariski 1-factor.

Consider all possible $\A^1$-fibrations $X\to Z$ on $X$ over smooth affine curves $Z$ along with the their pseudominimal extended graphs $\Gamma_{\rm ext}$. 
By Lemma \ref{lem: Gizatullin}(a)
the set $\Omega(X)$ of the isomorphism classes of graphs $\Gamma_{\rm ext}$ is finite. So, there is $d\in\NN$ which divides the multiplicities of the fiber components in any $\A^1$-fibration $X\to Z$. 

Applying to the given $\A^1$-fibration $X\to C$ the branched covering trick of Lemma \ref{lem: br-covering} of degree $d$
one obtains a marked GDF $\mu_d$-surface $\tilde X\to B$.  By Proposition \ref{prop: Gm-surfaces}, $\tilde X\to B$ does not admit a line bundle structure. 
By Proposition \ref{prop: GDF mu-d-1-factors}
there is a sequence 
of pairwise non-isomorphic marked GDF $\mu_d$-surfaces $\tilde X^{(kd)}\to B$ such that for all $k\in\N$ the cylinders $\tilde \cX^{(kd)}(0)$ and $\tilde \cX(0)$ are $\mu_d$-equivariantly  isomorphic over $B$ while  $v(\tilde X^{(kd)})\to\infty$ as $k\to\infty$ where $v$ stands as before for the number of vertices in the extended graph of a pseudominimal completion.  
Passing to the quotients under the $\mu_d$-actions yields a sequence  of  $\A^1$-fibered 
normal affine surfaces  $X^{(kd)}=\tilde X^{(kd)}/\mu_d\to C$ with  cylinders isomorphic over $C$: $$\tilde \cX^{(kd)}(0)/\mu_d\cong_C\tilde\cX(0)/\mu_d=\cX\quad \forall k\in\N\,.$$ 
We claim that under our assumptions the surface $X^{(kd)}$ is not isomorphic to $X$ for any $k\gg 1$. 
Suppose to the contrary that $X^{(kd)}\cong X$ for an infinite set $I$ of values of $k\ge 1$. Then $X$ admits at least two different $\A^1$-fibrations over affine bases, that is, $X$ is a Gizatullin surface. Indeed, otherwise any isomorphism $\phi\colon X^{(kd)}\stackrel{\cong}{\longrightarrow} X$ sends the $\A^1$-fibration $X^{(kd)}\to C$ to  the unique $\A^1$-fibration $X\to C$.  It can be lifted via the base change $B\to C$ and a normalization to an isomorphism 
$\tilde\phi\colon \tilde X^{(kd)}\stackrel{\cong}{\longrightarrow} \tilde X$. This gives a contradiction since $v(\tilde X^{(kd)})> v(\tilde X)$ for $k\gg 1$.

Thus, under our assumptions $X$ and also $X^{(kd)}\cong X$, $k\in I$, are Gizatullin surfaces. Hence $C\cong \A^1$,
and so, $B\cong\A^1$ too. By Lemma \ref{lem: Gizatullin}(b) the set $\tilde\Omega(X,d)=\tilde\Omega(X^{(kd)},d)$, $k\in I$, is finite.
In particular,  for any $k\in I$ the pseudominimal extended graph $\tilde\Gamma^{(kd)}_{\rm ext}$ associated with the GDF surface $\tilde\pi\colon\tilde X^{(kd)}\to B$ (which is a cyclic cover of $X^{(kd)}$) belongs to the finite set $\tilde\Omega(X^{(kd)},d)=\tilde\Omega(X,d)$. Since the set $I\subset\N$ is infinite this contradicts the fact that $v(\tilde\Gamma^{(kd)}_{\rm ext})=v(\tilde X^{(kd)})\to\infty$ as $k\to\infty$.  
Hence $X\not\cong X^{(kd)}$ for all  $k\gg 1$. \eproof

\section{Classical examples}\la{sec:examples}

In this section we analyse from our viewpoint  examples of non-cancellation due to Danielewski \cite{Da}, Fieseler \cite{Fi}, Wilkens \cite{Wi}, tom Dieck \cite{tD}, 
and Miyanishi--Masuda \cite{MM1}. We retrieve certain classification results for Danielewski-Fieseler surfaces in $\A^3$ due to Dubouloz and Poloni (\cite{DP0}, \cite{Pol}) and derive new explicite examples of non-cancellation. See also \cite[\S 4]{FM} for further examples. 

\bnota\la{ex: DS-again-2} Recall that a \emph{bush} is a rooted tree such that all the branches at the root vertex are chains, see Remark \ref{rem: DPD}.2.
Let $\Gamma_{d,m}$ be the bush with $d\ge 1$ branches of equal lengths $m\ge 1$. Let $\pi_{d,m}\colon X_{d,m}\to B=\Spec \mathbb{k}[z]\cong\A^1$ be a Danielewski-Fieseler surface with the unique special fiber $\pi_{d,m}^{-1}(0)$ such that $\Gamma_0(\pi_{d,m})\cong\Gamma_{d,m}$. Thus, the fiber $\pi_{d,m}^{*}(0)$ is reduced and has $d$ components $F_1,\ldots,F_d$, all on level $m$.\enota

\bexa\la{ex:DP} Let $X_{d,m,g}$ be a  surface 
given in $\A^3=\Spec \mathbb{k}[z,u,t]$ by an equation  
\be\label{eq: E1} z^mt- g_m(z,u)=0\quad\mbox{where}\quad g_m(z,u)= b_0(u) + b_1(u)z + \ldots + b_{m-1}(u)z^{m-1}\, \ee
for some $b_i\in \mathbb{k}[u]$, $i=1,\ldots,m-1$ with
$\deg b_i\le d-1$ and some monic polynomial $b_0\in \mathbb{k}[u]$ of degree $d$ with simple roots. Then $X_{d,m,g}$ is a surface of type $X_{d,m}$, see  \cite[\S 3.1]{DP0}.
 The surfaces $X_{d,m,g}$ were classified in \cite[Thm. 6.1(1)]{Pol}; see also \cite{Dai},   \cite{ML}, and \cite{MJP} for some particular cases.  For the reader's convenience we reproduce this classification with a new proof. 
\eexa

\bprop\la{prop: GDWilkens}
\begin{itemize}
\item[{\rm (a)}] For any surface $X_{d,m}$ as in {\rm \ref{ex: DS-again-2}} there exists a polynomial $g\in\mathbb{k}[z,u]$ as in Example {\rm \ref{ex:DP}} such that $X_{d,m}\cong_B X_{d,m,g}$.
Vice versa, any surface $X_{d,m,g}$ as in  {\rm \ref{ex:DP}} is a GDF surface $X_{d,m}$ as in {\rm \ref{ex: DS-again-2}}.
\item[{\rm (b)}] Up to isomorphism over $B=\Spec \mathbb{k}[z]$ the cylinder $\cX_{d,m,g}=:\cX_{d,m}$ does not depend on the choice of $g$ in {\rm (\ref{eq: E1})}.
Furthermore, $\cX_{d,m}\cong_B \cX_{d',m'}$ if and only if $d=d'$. 
\item[{\rm (c)}] If $X_{d,m,g}\cong X_{d',m',h}$ then $d=d'$ and $m=m'$. For $d,m\ge 2$ the following are equivalent:
\begin{itemize}\item[$\bullet$] 
$X_{d,m,g}\cong X_{d,m,h}$;
\item[$\bullet$]  there is a commutative diagram
\be
 \bdi\la{diagr:blowups}
X_{d,m,g}&\rTo<{}>{\cong} &  X_{d,m,h} \\
\dTo<{} &  & \dTo>{} \\
B&\rTo>{}<{\cong} & B\,\, 
\edi
\ee
\item[$\bullet$] 
there exist $\alpha,\,\lambda\in \mathbb{k}^*$ and $\beta,\,\gamma\in \mathbb{k}[z]$ with $\deg \beta\le m-1$ such that
$$h(z,u)= (g(\lambda z, \,\alpha u+ \beta (z))-\gamma (z))/\alpha^d\, $$ 
where $\gamma$ is uniquely defined by $d,m, \alpha, \lambda$, and  $\beta$ and  the affine transformation
$ u \mapsto \alpha u + \beta (0)$ sends the roots of $b_0=g(0,u)$ to the roots of $c_0:=h(0,u)$. 
\end{itemize}
\end{itemize}
\eprop

\bproof Consider a pseudominimal SNC completion $(\bar X_{d,m}, D_{d,m})$ of $X_{d,m}$ with projection $\bar\pi_{d,m}\colon\bar X_{d,m}\to\bar B=\PP^1$ extending $\pi_{d,m}$, the fiber at infinity $\bar F_\infty\subset D_{d,m}$, the section  at infinity  $S\subset D_{d,m}$, and 
the unique  reduced degenerate fiber $\bar\pi^*(0)=C_0+\sum_{i=1}^d\cB_i$ where $C_0$ is the root  of $\Gamma_0(\pi_{d,m})\cong\Gamma_{d,m}$ of weight $C_0^2=-d$, and $\cB_i=C_{i,1}+\ldots+C_{i,m-1}+\bar F_i$, $i=1,\ldots,d$ are chains of length $m$ with the sequence of weights $[[-2,\ldots,-2,-1]]$ and the $(-1)$-tips $\bar F_i$. 

Let $\sigma\colon \bar X_{d,m}\to \bar X_{d,m-1}$ be the contraction  
of $\bar F_1,\ldots,\bar F_d$. The image of $ C_{i,m-1}$ acquires the weight $-1$ on $\bar X_{d,m-1}$. 
Iterating this procedure leads to a sequence \eqref{eq: seq-contractions}:
\be\la{eq: E11}
\bar X_{d,m} \stackrel{\bar\varrho_m}{\longrightarrow} \bar X_{d,m-1} \stackrel{}{\longrightarrow}  
\ldots \stackrel{}{\longrightarrow}   \bar X_{d,1} \stackrel{\bar\varrho_1}{\longrightarrow} 
\bar X_{d,0}=\PP^1\times\PP^1\,
\ee along with the corresponsding trivializing sequence (\ref{eq: seq-aff-modif-final}) for $X_{d,m}$ where $\bar X_{d,l}$ is a pseudominimal completion of $X_{d,l}$ and $$\pi_{d,0}\colon X_{d,0}=\A^2\to B=\A^1,\qquad (z,u)\mapsto z\,.$$
The affine modification $X_{d,1}\to X_{d,0}=\Spec\mathbb{k}[z,u]$  has divisor $z=0$ and a reduced center consisting of $d$ distinct points, say, $(0,\alpha_1), \ldots , (0,\alpha_d)$. Let $b_0\in \mathbb{k}[u]$ be the monic polynomial of degree $d$ with simple roots $\alpha_1,\ldots,\alpha_d$. 
One has $$X_{d,1}\cong_B X_{d,1,b_0}\quad\mbox{where}\quad X_{d,1,b_0}=\{zt_1-b_0(u)=0\}\subset\A^3$$ is given by (\ref{eq: E1}) with $m=1$ and $t=t_1$. 

Assume by recursion that $X_{d,l}\cong_B X_{d,l,g_l}$ where $X_{d,l,g_l}$ is given in $\A^3=\Spec\mathbb{k}[z,u,t_l]$ by an equation $z^lt_l-g_l(z,u)=0$ as in  (\ref{eq: E1}).
The surface $X_{d,l+1}$ in (\ref{eq: seq-aff-modif-final}) is obtained from $X_{d,l}=X_{d,l,g_l}$ by a fibered modification  with the reduced divisor $z^*(0)$ consisting of $d$ disjoint components, say, $F_{1,l},\ldots,F_{d,l}\subset X_{d,l,g_l}$ (where the closure $\bar F_{i,l}$ in $\bar X_{d,l}$ is the image of $C_{i,l}$) and a reduced center
$Z$ which consists of $d$ points, say, $$x_i=(0,\alpha_i,\beta_i)\in F_{i,l},\quad i=1,\ldots,d\,.$$
Let $b_l\in \mathbb{k}[u]$ be the polynomial of the minimal possible degree (so, $\deg b_l\le d-1$) such that $b_l (\alpha_i)+\beta_i=0$, $i=1,\ldots,d$. The three surfaces  in $\A^3=\Spec\mathbb{k}[z,u,t_l]$,
$$X_{d,l,g_l}, \quad \{z=0\}, \quad\mbox{and}\quad \{b_l(u)+t_l=0\}$$ meet transversely at the points $x_i$, $i=1,\ldots,d$. Hence for the ideal $I\subset\cO_{X_{d,l}}(X_{d,l})$ of $Z$ one has $I=(z, b_l(u)+t_l)$. It follows that
$$\cO_{X_{d,l+1}}(X_{d,l+1})=\cO_{X_{d,l}}(X_{d,l})[t_{l+1}]\quad\mbox{where}\quad t_{l+1}=(b_l(u)+t_l)/z\,,$$ see Definition
\ref{def: aff-modif}. 
By the inductive hypothesis one obtains 
$$t_l=g_l(z,u)/z^l=(b_0(u) + b_1(u)z + \ldots + b_{l-1}(u)z^{l-1})/z^{l}\,$$ where $t_l$ fits in (\ref{eq: E1}) with $m=l$.
Therefore, $X_{d,l+1}=X_{d,l+1,g_{l+1}}$ is given in $\A^3=\Spec\mathbb{k}[z,u,t_{l+1}]$ by the equation
$$z^{l+1}t_{l+1}-g_{l+1}(z,u)=0\quad\mbox{where}\quad g_{l+1}(z,u)=b_0(u) + b_1(u)z + \ldots + b_{l}(u)z^{l}\,.$$ This proves the first assertion in  (a).
Repeating the same argument in the reversed order gives the second (alternatively, see \cite[\S 3.1]{DP0}). 

By (a) two GDF surfaces $X_{d,m,g}$ and $X_{d,m,h}$ with the same $d$ and $m$ have the same graph divisor. According to Theorem 
\ref{thm: GDF-cancellation-fixed-graph} 
the cylinders over these surfaces are isomorphic over $B=\Spec\mathbb{k}[z]$. 

For any $m'> m$ the affine modification $$\sigma_{m',m}\colon X_{d,m'}\to X_{d,m},\qquad\sigma_{m',m}=\rho_{m+1}\circ\cdots\circ\rho_{m'}$$ is a principal top level stretching.  Hence by Proposition \ref{prop: isomorphism} one has 
$\cX_{d,m}\cong_B \cX_{d,m'}$. Alternatively, the latter follows by the Danielewski's argument, see Remark \ref{rem:added}.

The number $d$ of components of the divisor $z^*(0)$ in $\cX_{d,m}$ is invariant upon isomorphisms of cylinders over $B$. This yields (b). 

The extended graph $\Gamma_{\rm ext}=\bar F_\infty\cup S\cup\Gamma_{d,m}$ of the pseudominimal completion $(\bar X_{d,m},D_{d,m})$ of $X_{d,m}$ has $v(X_{d,m})=dm+3$ vertices. According to Lemma \ref{lem: vertices}, $v(X_{d,m})$ is an invariant of the surface, as well as the Picard number $\rho(X_{d,m})=d-1$, see \eqref{eq:Picard-number}. 
Hence $X_{d,m}\cong X_{d',m'}$ implies $dm=d'm'$ and $d=d'$, and so, $m=m'$. 

For $d,m\ge 2$ the dual graph of the boundary divisor $\hat D_{d,m}$ is minimal and nonlinear.
It follows that $X_{d,m}$ 
is a surface of class $(\ML_1)$, that is, it 
admits a unique $\A^1$-fibration over an affine base.  Hence any isomorphism $X_{d,m,g}\cong X_{d,m,h}$  fits in a commutative diagram \eqref{diagr:blowups}. 

Next we follow the line of the proof of Lemma 5.12 in \cite{FKZ-uniqueness}.
Assume there is an isomorphism $\varphi\colon X_{d,m,g}\stackrel{\cong}{\longrightarrow}  X_{d,m,h}$. The induced 
birational map $\bar\phi\colon\bar X_{d,m,g}\dashrightarrow \bar X_{d,m,h}$ fits in the commutative diagram
$$\bdi
\bar X_{d,m,g} &\rDashto<{\bar\phi}>{} &  \bar X_{d,m,h}\\
\dTo<{\sigma} &&\dTo>{\sigma'} \\
W &\rDashto<{\bar\psi}>{} & W'\\
\dTo<{} &&\dTo>{} \\
\PP^1 &\rTo<{\cong}>{\eta} & \PP^1\, 
\edi$$
where $\sigma$ is the contraction of the disjoint  union of chains\footnote{ For a subgraph $\Gamma'$ of a graph $\Gamma$ we let $\Gamma\ominus\Gamma'$ denote the graph obtained from $\Gamma$ by deleting $\Gamma'$ along with the edges of $\Gamma$ incident to $\Gamma'$.} $\cB_i\ominus\bar F_i$ on $\bar X_{d,m,g}$, $i=1,\ldots,d$ with a sequence of weights $[[-2,\ldots,-2]]$ of length $m-1$  along with the root $C_0$ to a normal surface singularity of $W$, and $\sigma'$ is a similar contraction on $\bar X_{d,m,h}$. Clearly, $\bar\phi$ is regular in the points of $S\setminus ( C_0\cup \bar F_\infty)$. So, $\bar\psi$ is biregular over $\PP^1\setminus\{\infty\}$ outside the isolated normal singularities of $W$ and $W'$. By the Riemann extension theorem, $\bar\psi$ extends across the singular point of $W$ yielding a biregular isomorphism $\bar\psi\colon W\setminus \sigma(\bar F_\infty)\to W'\setminus \sigma'(\bar F'_\infty)$. It follows that $\bar\phi$ also extends to the minimal resolutions of singularities yielding an isomorphism $\bar X_{d,m,g}\setminus \bar F_\infty\stackrel{\cong}{\longrightarrow} \bar X_{d,m,h}\setminus \bar F'_\infty$. 

Contracting  on the surfaces $\bar X_{d,m,g}\setminus (\bar F_\infty\cup S)$ and $\bar X_{d,m,h}\setminus (\bar F'_\infty\cup S')$ the unions of chains $\bigcup_{i=1}^d \cB_i$ and $\bigcup_{i=1}^d \cB'_i$, respectively, yields an isomorphism 
$$\phi\colon X_{d,0,g}=\Spec\mathbb{k}[z,u]\stackrel{\cong}{\longrightarrow} X_{d,0,h}=\Spec\mathbb{k}[z,u],\,\,\, (z,u) \mapsto (\lambda z, \alpha u+ \beta(z))$$
where $\alpha, \lambda\in \mathbb{k}^*$ and $\beta\in \mathbb{k}[z]$.
It sends the roots of $b_0(u)=g(0,u)$ into the roots of $c_0(u)=h(0,u)$.  Hence
 $$c_0(u)= \alpha^{-d} b_0\circ\phi(0,u) =\alpha^{-d} b_0 (\alpha u +d_0)\,.$$ 
The automorphism $\phi\in\Aut\A^2$ extends to 
an automorphism $\Phi\in\Aut \A^3$, 
\be\label{eq: Sh1} 
\Phi\colon (z,u, t) \mapsto \left(\lambda z, \alpha u+ \beta(z),{\frac{\alpha^d}{\lambda^m}}t+\gamma(z)\right) \,, \ee with a unique polynomial $\gamma\in \mathbb{k}[z]$ 
depending on $d,m, \alpha, \beta, \lambda$
such that  the resulting equation $z^d t-h(z,u)=0$ of the surface $X_{d,m,h}=\Phi(X_{d,m,g})$ has again the required form (\ref{eq: E1}). The uniqueness of $\gamma$ follows from the relation
$$\alpha^d h(z,u)= g(\lambda z,\alpha u+ \beta (z))- \lambda^m\gamma (z)\,.$$ This relation shows as well that $h(z,u)$ is independent of the higher order terms of $\beta (z)$. The rest of the proof of (c) is easy and can be left to the reader.
\eproof

\bsit\la{rem: Wilkens} Proposition \ref{prop: GDWilkens} was established by Wilkens (\cite{Wi}) for $\mathbb{k}=\C$, $d=2$, $m\ge 2$, and $g(z,u)$ in (\ref{eq: E1}) of the form $g(z,u)=h(z)u+u^2$ where $h\in \mathbb{k}[z]$ is a polynomial of degree $\deg h< m$ with $h(0)\neq 0$. Generalizing the examples of Danielewski-Fieseler \cite{Da,Fi} and tom Dieck \cite{tD}, Miyanishi and Masuda  (\cite{MM1}) considered yet another instance of surfaces in $\A^3_\C$ given by (\ref{eq: E1}).
\esit

\bexa[\emph{Masuda-Miyanishi \cite{MM1}}]\la{ex: MM}
Given natural $d,m$ with $d > m\ge 2$ and a homogeneous polynomial 
\be\la{eq: MM} g(z,u)= u^d + a_2 u^{d-2}  z^2+ \ldots + a_d z^d\,\ee 
consider  the 
surface $X(d, m,g)$ in $\A^3$ defined by the equation
$$z^m t - g(z,u) - 1=0\,.$$  Then $\pi=z|_{X(d,m,g)}\colon X(d,m,g)\to B=\Spec\mathbb{k}[z]$ is a Danielewski-Fieseler surface with a unique special fiber $\pi^{-1}(0)$ consisting of $d$ components on level $m$. The examples of tom Dieck (\cite{tD}) correspond to the case $g(z,u)=u^d$. By a suitable triangular automorphism $(z,u,t)\mapsto (z,u,t-b(z,u))$ one can eliminate the last $d-m$ terms in (\ref{eq: MM}) and reduce the equation to  (\ref{eq: E1}). The next result of Miyanishi and Masuda is a particular case of Proposition \ref{prop: GDWilkens}. 
\eexa

\bprop\label{thm: MM} {\rm (\cite[Thm.\ 2.8]{MM1})} With the notation as in {\rm \ref{ex: MM}} the following
hold.
\begin{itemize}\item[{\rm (a)}] 
$\cX(d, m,g)\cong_B
\cX(d, m', h)$ for any $d, m, m'$ and $g,h\in \mathbb{k}[z,u]$ as in \eqref{eq: MM}.
\item[{\rm (b)}]  $X(d, m,g)\cong X(d, m', h)$ if and only if $m = m'$ and $ h(z,u)=g(\lambda z,u)$ for some $\lambda \in \mathbb{k}^*$.
\end{itemize}
\eprop

\end{document}